%% file: idla_mated_crt.tex
\documentclass[12pt]{amsart}
\usepackage[dvipsnames]{xcolor}
\usepackage{fullpage,graphicx,psfrag,amsmath,amsfonts, amssymb,enumitem, hyperref}
\usepackage{forest}
\usetikzlibrary{positioning}

\hypersetup{
	colorlinks=true,
	linkcolor=NavyBlue,
	urlcolor = RoyalBlue,
	citecolor = PineGreen
}
\usepackage{makecell}

\renewcommand*{\P}{\mathbb P}

\newcommand{\cl}{\mathbf{cl}}
\newcommand{\inte}{\mathbf{int}}

\input defs.tex

\newcommand{{\eps}}{\varepsilon}
\newcommand{{\Geps}}{\mathcal{G}^{{\eps}}}
\newcommand{{\VGeps}}{\mathcal{VG}^{{\eps}}}
\newcommand{{\meps}}{\hyperref[eq:meps]{m_{{\eps}}}}

\newcommand{\Mrho}[1]{\hyperref[eq:largest-exit-time-rv]{\mathcal{M}(#1)}}

\newcommand{\Heps}[1]{\hyperref[eq:sle-cell]{H_{#1}^{{\eps}}}}

\newcommand{\gr}[2]{\hyperref[eq:greens-kernel]{\operatorname{gr}^{#1}_{#2}}}
\newcommand{\Gr}[2]{\hyperref[eq:greens-function]{\operatorname{Gr}^{#1}_{#2}}}

\newcommand{\taueps}[2]{\hyperref[eq:exit-time]{\tau_{#1}^{#2, {\eps}}}}
\newcommand{\Qeps}[1]{\hyperref[eq:expected-exit-time]{\mathfrak{Q}_{#1}^{{\eps}}}}
\newcommand{\qeps}[1]{\hyperref[eq:normalized-exit-time]{\mathfrak{q}_{#1}^{{\eps}}}}
\newcommand{\olqeps}[1]{\hyperref[eq:normalized-exit-time]{\overline{\mathfrak{q}}_{#1}^{{\eps}}}}

\newcommand{\TT}{\hyperref[eq:cluster-stopping-time]{T_1}}

\newcommand{\Bdeltap}{\hyperref[eq:outer-inner-neighborhoods]{B_{\delta}^+}}
\newcommand{\Bdeltam}{\hyperref[eq:outer-inner-neighborhoods]{B_{\delta}^-}}

\newcommand{\Bdeltapp}[1]{\hyperref[eq:outer-inner-neighborhoods]{B_{#1}^+}}
\newcommand{\Bdeltamm}[1]{\hyperref[eq:outer-inner-neighborhoods]{B_{#1}^-}}

\newcommand{\wteps}{\hyperref[eq:discrete-least-supersolution]{w^{{\eps}}_t}}
\newcommand{\olwteps}{\hyperref[eq:discrete-least-supersolution]{\overline{w}^{{\eps}}_t}}

\newcommand{\Lambdateps}{\hyperref[eq:discrete-cluster]{\Lambda^{{\eps}}_t}}

\numberwithin{equation}{section}

\title{Internal DLA on mated-CRT maps}
\author{Ahmed Bou-Rabee}
\author{Ewain Gwynne}

\begin{document}

	\begin{abstract}
	We prove a shape theorem for internal diffusion limited aggregation on mated-CRT maps, a family 
	of random planar maps which approximate Liouville quantum gravity (LQG) surfaces. 
	The limit is an LQG harmonic ball, which we constructed in a companion paper.
	We also prove an analogous result for the divisible sandpile. 
	\end{abstract}
	\maketitle

	\section{Introduction} \label{sec:intro}
	\begin{figure}
	\includegraphics[width=0.35\textwidth]{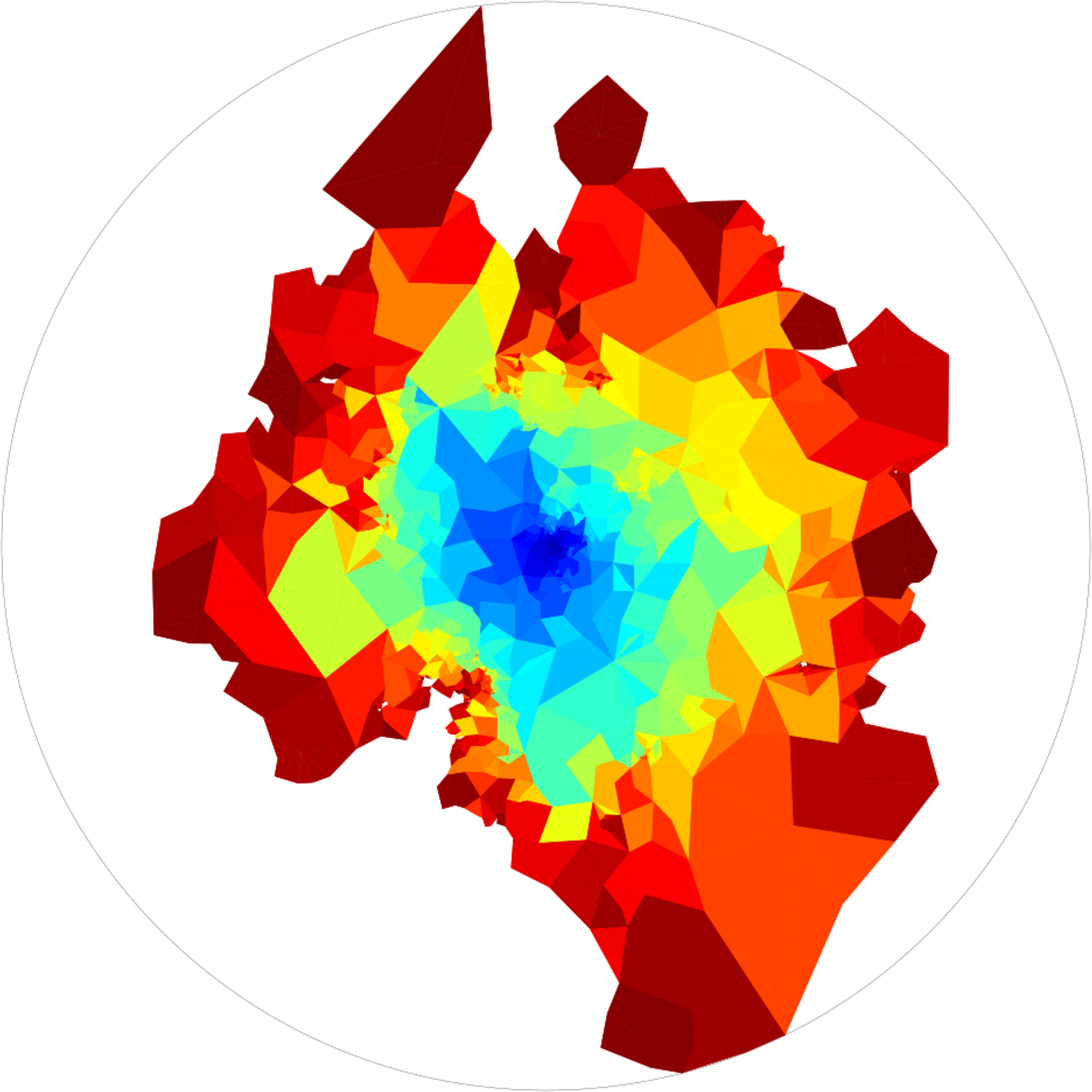} 
	\includegraphics[width=0.35\textwidth]{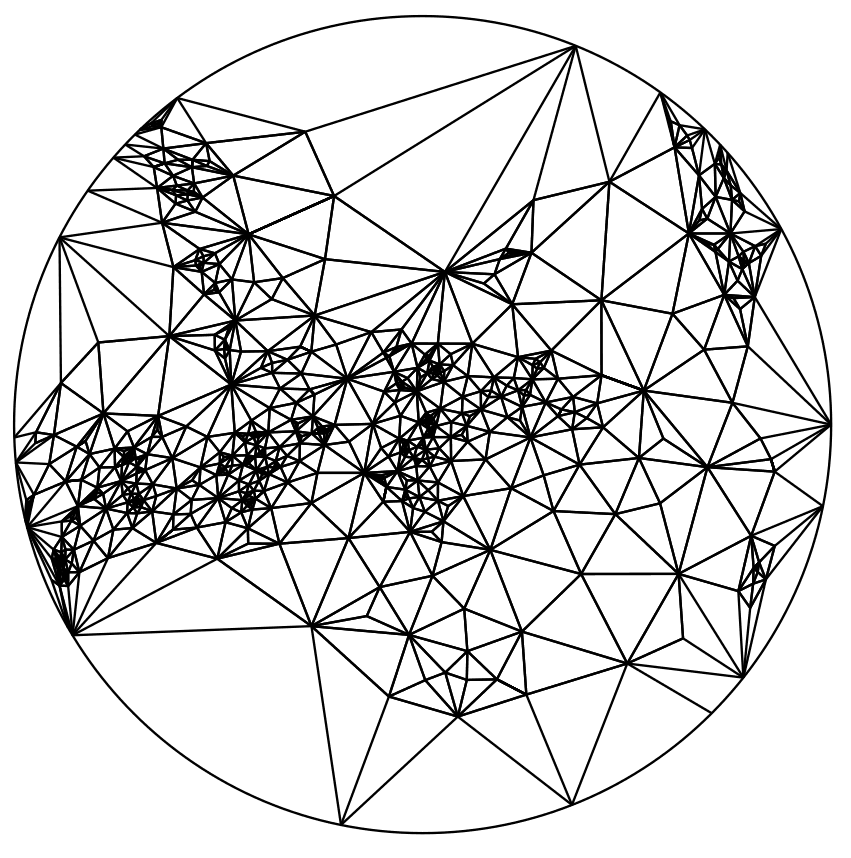}
	\caption{\textbf{Left:} IDLA on the Tutte embedding of a finite $\gamma = \sqrt{2}$ mated-CRT map with boundary, with $10^4$ vertices. 
		Walkers are added until the first time one of them hits a boundary vertex.  
		Cells are colored according to the first time one of their vertices is hit by a random walk. 
		\textbf{Right:} Illustration of the graph upon which IDLA is run. The finite mated-CRT map with boundary can be constructed from a pair of conditioned Brownian motions via a similar procedure as in Definition~\ref{def:mated-crt-map}. Away from the boundary, it locally looks like the infinite mated-CRT map considered in this paper, but its Tutte embedding is easier to define and simulate. See Remark~\ref{remark:tutte} for further discussion.
	}
	\label{fig:IDLA-tutte-embedding}
\end{figure}
In this article we study the large-scale behavior of a random growth model (IDLA) on a random graph (the mated-CRT map)
which approximates a random fractal surface embedded in the plane (Liouville quantum gravity, LQG). We establish, for each $\gamma \in (0,2)$, the convergence in probability of the IDLA cluster on the $\gamma$-mated-CRT map to a so-called $\gamma$-LQG harmonic ball (Theorem~\ref{theorem:convergence-of-IDLA}). Harmonic balls are random subsets of the plane characterized by a mean-value property for harmonic functions with respect to the LQG measure (Definition~\ref{def:harmonic-ball}). In a companion work \cite{bou2022harmonic}, harmonic balls are constructed via Hele-Shaw flow and it is shown that they are neither Lipschitz domains nor LQG metric balls. 

\subsection{Background}
We now briefly and somewhat informally describe the aforementioned objects, delaying the more technical definitions to Section \ref{sec:preliminaries} below. Minimal prior knowledge of LQG is needed to read this article.

\subsubsection{Internal DLA} \label{subsubsec:idla-informal-intro}
{\it Internal diffusion limited aggregation} (Internal DLA or IDLA) was first introduced to model the physical process of `material removal'
by Meakin-Deutch in \cite{meakin1986formation}. It was also independently introduced by Diaconis-Fulton in  \cite{diaconis1991growth}  as a special case of a commutative algebra.

In its simplest form, IDLA produces an increasing sequence of subsets of the vertex set of a graph $G$ (often taken to be a lattice such as $\Z^2$) as follows. Fix a root vertex $o$ for $G$. Send out an explorer starting from $o$ and allow it to travel according to a simple random walk on $G$ until it reaches an unoccupied vertex, at which point it stops and occupies that vertex. This process is repeated, each explorer starts at the root and walks randomly until it reaches an unoccupied vertex. For $n \geq 1$, the time $n$ IDLA {\it cluster} is the collection of occupied vertices after $n $ iterations of this. 

Following remarkable work by several groups of mathematicians, the large $n$ behavior of the IDLA clusters is now well understood on several families of graphs: $\Z^d$
\cite{lawler1992internal, lawler1995subdiffusive,asselah2013logarithmic, asselah2013sublogarithmic,jerison2012logarithmic, jerison2013internal, jerison2014internal, lucas2014limiting, idla-uniform-starting-points, darrow2020convergence},
Cayley graphs of groups with polynomial \cite{blachere2001internal} and  exponential \cite{blachere2007internal, huss2008internal} growth, supercritical percolation clusters \cite{shellef2010idla, duminil2013containing}, Sierpinski gasket graphs \cite{chen2017internal}, comb lattices \cite{huss-sava-comb, asselah2016fluctuations}, and cylinders \cite{jerison2014chapter, levine2019long, silvestri2020internal}. In most of these cases, the scaling limit of the IDLA clusters is described by a metric ball in the corresponding ambient space (the main result of the paper will give an example where the scaling limit of IDLA is not a metric ball).

We would like to examine what happens when $G$ is taken to be a {\it random} graph embedded in the plane. Specifically, instead of taking some deterministic lattice
which approximates Euclidean space, we take $G$ to be a graph which approximates a {\it Liouville quantum gravity} (LQG) surface.

\subsubsection{Liouville quantum gravity}	
LQG was introduced by Polyakov in the 1980s in the context of bosonic string theory~\cite{polyakov-qg1} as a model of a `random two-dimensional Riemannian manifold'. LQG is too rough to be defined as a manifold in the rigorous sense, although the following loose description can be made precise via various regularization procedures.
For $\gamma \in (0,2)$ and a domain $D\subset\C$, a $\gamma$-LQG surface parameterized by $D$ is the two-dimensional Riemannian manifold with Riemannian metric tensor $e^{\gamma h} \, (dx^2+dy^2)$, where $dx^2+dy^2$ is the Euclidean metric tensor and $h$ is a variant of the Gaussian free field (GFF) on $D$. 

Of particular relevance to this work is the associated volume form, the {\it $\gamma$-Liouville measure} or {\it $\gamma$-LQG measure}. This is a Radon measure
on $D$ which is (informally) given by 
\begin{equation} \label{eq:lqg-measure}
\mu_h(dz) = e^{\gamma h(z)} \, d z, 
\end{equation}
where $dz$ denotes two-dimensional Lebesgue measure~\cite{kahane1985chaos,rhodes-vargas-log-kpz,duplantier2011liouville}. We will give a precise definition of $\mu_h$ and the relevant variants of the GFF in Section \ref{sec:preliminaries}. For now,	the reader only needs to know that $\mu_h$ is a random non-atomic, locally finite Radon measure on $D$ which assigns positive mass to every open subset of $D$ and is mutually singular with respect to Lebesgue measure.
The works~\cite{ding2020tightness,gm-uniqueness} showed that an LQG surface can also be endowed with a canonical random metric, but we will not need this metric in the present paper.

\subsubsection{Random planar maps and the mated-CRT map} \label{subsubsec:mated-crt-map-intro}
\begin{figure}[ht!]\begin{center}
		\includegraphics[scale=.8]{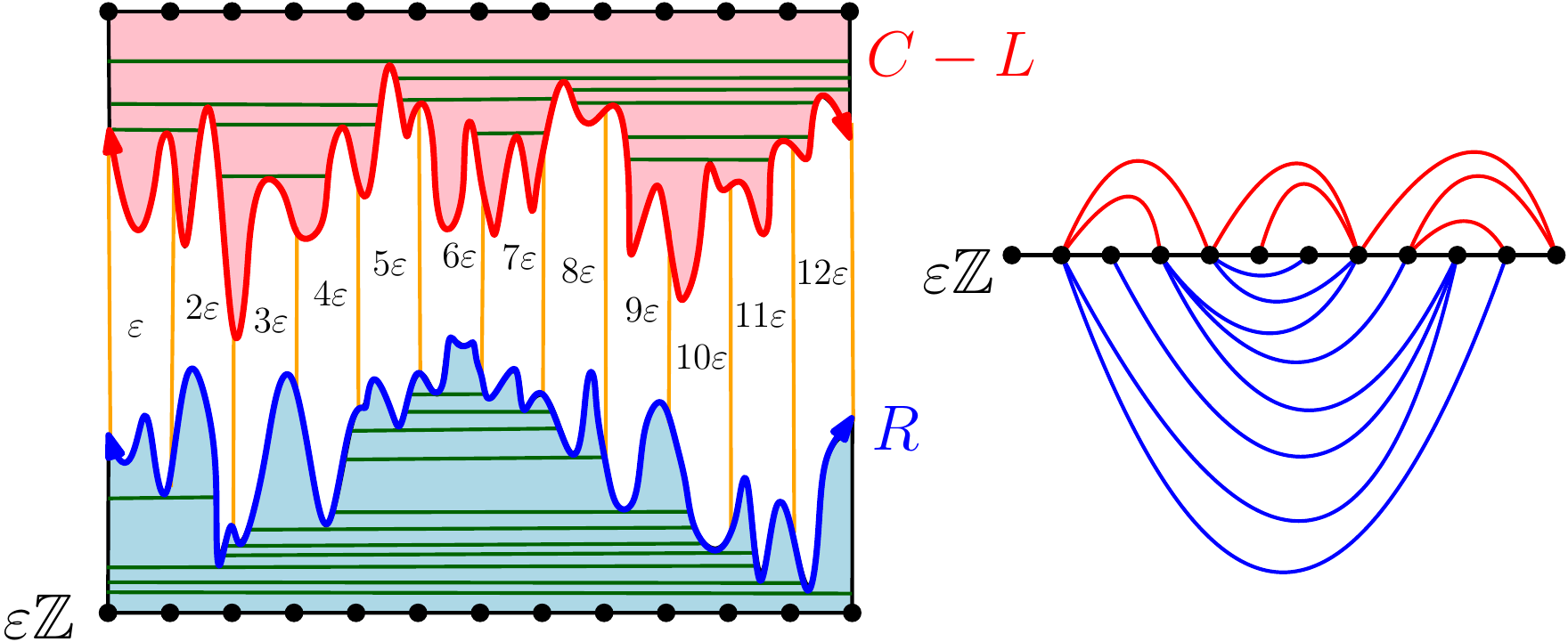}
		\caption{{\textbf{Left:} A sample construction of the mated-CRT map. We draw the graphs of $C-L$ and $ R$ in the interval $[0, 12 {\eps}]$ 
				where $C$ is a large constant chosen so that the graphs do not intersect. Each vertical strip containing the interval $[a-{\eps}, a]$ for $a \in {\eps} \Z$  corresponds to a vertex
				in the mated-CRT map. The adjacency condition \eqref{eq:edge-map} for $L$ (resp.\ $R$) corresponds to two vertices $a,b \in {\eps} \Z$ being adjacent if there is a horizontal line segment above the graph of $C-L$ (resp.\ below the graph of $ R$) which intersects the graph only in the vertical strips which contain the intervals $[a-{\eps}, a]$ and $[b-{\eps},b]$.
				In the figure, we have illustrated, for each pair $(a,b)$ for which this adjacency condition holds, the lowest (resp.\ highest) such horizontal line segment in green.
				\textbf{Right:} A planar embedding of the mated-CRT map on the left, under which it is a triangulation. Edges arising from the adjacency condition for $L$ (resp.\ $R$) which do not join consecutive vertices are shown in red (resp.\ blue). Edges joining consecutive vertices are shown in black. A similar illustration appeared as~\cite[Figure 1]{gms-harmonic}.}}
		\label{figure:mated-crt-map-illustration}
	\end{center}
\end{figure}

LQG is widely expected, and in some cases rigorously known, to describe the scaling limit of 
{\it random planar maps}. A \textit{planar map} is a graph embedded in $\C$ in such a way that no two edges cross, viewed modulo orientation preserving homeomorphisms $\C \rightarrow \C$. 
Uniform random planar maps (including uniform triangulations, quadrangulations, etc.) converge to $\sqrt{8/3}$-LQG surfaces in the Gromov-Hausdorff sense~\cite{legall-uniqueness,miermont-brownian-map,lqg-tbm1,lqg-tbm2} and, in the case of triangulations, when embedded via the so-called \textit{Cardy embedding}~\cite{hs-cardy-embedding}. Similar convergence results are expected to hold for various types of non-uniform random planar maps toward $\gamma$-LQG with $\gamma\not=\sqrt{8/3}$. For example, random planar maps sampled with probability proportional to the Ising model partition function are expected to converge to $\sqrt 3$-LQG (see, e.g.,~\cite[Section 3.1]{ghs-mating-survey} for more on such conjectures).

Mated-CRT maps are a one-parameter family of random planar maps, indexed by $\gamma\in(0,2)$, whose connection to LQG is better understood than for most other random planar maps. In particular, it was proven by Gwynne-Miller-Sheffield \cite{gms-tutte, gms-random-walk} that the re-scaled counting measure on vertices of the $\gamma$-mated-CRT  map converges to the $\gamma$-LQG measure when the map is embedded into $\C$ via the \emph{Tutte embedding} (i.e., the embedding where the position of each vertex is the average of the positions of its neighbors). Building on this, it was also shown by Berestycki-Gwynne in \cite{berestycki2020random} that random walk on the mated-CRT map converges to so-called \emph{Liouville Brownian motion}, see Section \ref{subsec:lbm-and-rw}
for the definition.

The mated-CRT map is built via a pair of correlated linear Brownian motions as in Figure \ref{figure:mated-crt-map-illustration}. 
\begin{definition}[Mated-CRT map] \label{def:mated-crt-map}
	Fix $\gamma \in (0,2)$ and consider a two dimensional Brownian motion $(L_t,R_t)_{t\in\R}$ started at the origin with correlation coefficient given by $ - \cos ( \pi \gamma^2/ 4)$.  For ${\eps} > 0$, the {\it $\gamma$-mated-CRT map} with cell size ${\eps}$ associated to $(L_t, R_t)_{t\in\R}$ is the graph ${\Geps}$ whose vertex set is ${\eps} \Z = {\VGeps}$, with two vertices $a < b$ connected by an edge if and only if
	\begin{equation} \label{eq:edge-map}
	\left( \inf_{t \in [a - {\eps},a]} X_t \right) \vee \left( \inf_{t \in [b - {\eps}, b]} X_t \right)\le \inf_{t \in [a, b - {\eps}]} X_t
	\end{equation}
	where $X$ can be either $L$ or $R$. There are two edges connecting $a$ and $b$ if $|b-a| > {\eps}$ and \eqref{eq:edge-map}
	is satisfied for both $L$ and $R$.  See Figure \ref{figure:mated-crt-map-illustration} for an illustration and a description of the planar map structure of $\mathcal {\Geps}$.
\end{definition}

By Brownian scaling, the law of ${\Geps}$ (as a graph) does not depend on ${\eps}$. However, ${\eps}$ is a convenient scaling parameter when talking about limits, as we will see just below.
The definition of the mated-CRT map is a semi-discrete generalization of so-called \emph{mating-of-trees bijections} between random planar maps and random walk excursions.
For instance, the \emph{Mullin bijection}~\cite{mullin-maps,bernardi-maps,shef-burger} is a discrete analog of Definition~\ref{def:mated-crt-map} which produces a planar map decorated by a spanning tree from a two-dimensional random walk excursion.  See~\cite{ghs-mating-survey} for a survey of mating-of-trees bijections and their applications.

The main reason why the mated-CRT map is more tractable than other random planar maps is that, due to the results of~\cite{duplantier2014liouville}, it has an \textit{a priori} embedding into $\C$ defined in terms of LQG and Schramm-Loewner evolution (SLE). We will now discuss this embedding.

\subsubsection{SLE/LQG embedding}
In this paper, we work with an embedding of the mated-CRT map constructed via a space-filling Schramm-Loewner evolution (SLE) curve. We remark, however, that by \cite{gms-tutte, berestycki2020random}, our results extend 
to the so-called Tutte (or harmonic) embedding of the mated-CRT map --- see Figure \ref{fig:IDLA-tutte-embedding} and Remark \ref{remark:tutte}. 

Let $h$ be the random distribution on $\C$ corresponding to a $\gamma$-quantum cone (a slight modification of the whole-plane GFF).
Let $\mu_h$ be the $\gamma$-LQG measure associated to $h$. 
Let $\eta$ be a whole-plane space-filling SLE, independent of $h$, with parameter $\kappa = 16/\gamma^2 \in (4,\infty)$.
We parameterize $\eta$ so that $\eta(0) = 0$ and $\mu_h(\eta([a,b])) = b-a$ for every $a,b\in\R$ with $a < b$. 
We will review the definitions of $h$ and $\eta$ in Section \ref{sec:preliminaries} below.
For now, the unfamiliar reader can just think of $\eta$ as a random continuous, space-filling curve in $\C$ from $\infty$ to $\infty$ parameterized by a locally finite Radon measure. 

\begin{definition}[SLE/LQG embedding of the mated-CRT map~\cite{duplantier2014liouville}] \label{def:sle-lqg-embedding}
	For $\gamma \in (0,2)$ and ${\eps} > 0$, let ${\Geps}$ be the ${\eps}$-mated CRT map with vertex set ${\VGeps}$ as defined in Definition \ref{def:mated-crt-map}.
	The SLE/LQG embedding is defined as follows. Each vertex $a \in {\VGeps} = {\eps}\Z$ corresponds to a \emph{cell},
	\begin{equation} \label{eq:sle-cell}
	\Heps{a} := \eta([a-{\eps}, a]) .
	\end{equation} 
	Two distinct vertices $a,b \in {\VGeps}$ are connected by one (resp. two) edges if and only if $\Heps{a} \cap \Heps{b}$ has one
	(resp. two) connected components which are not singletons.  This provides an embedding
	of ${\Geps}$ into $\C$ by sending each vertex $a \in {\VGeps}$ to the point $\eta(a) \in \Heps{a}$ then drawing edges between adjacent vertices in such a way that no two edges cross (the precise locations of the edges is unimportant for our purposes). 
\end{definition}

The equivalence of the above SLE/LQG description of the mated-CRT map and the Brownian motion description, Definition \ref{def:mated-crt-map}, 
is a consequence of the main result of the seminal work~\cite{duplantier2014liouville}. See, in particular, \cite[Lemma 8.8]{duplantier2014liouville}. For an illustration of the SLE/LQG embedding, see Figure \ref{fig:sle-lqg-embedding}. 

\begin{figure}
	\includegraphics[width=0.75\textwidth]{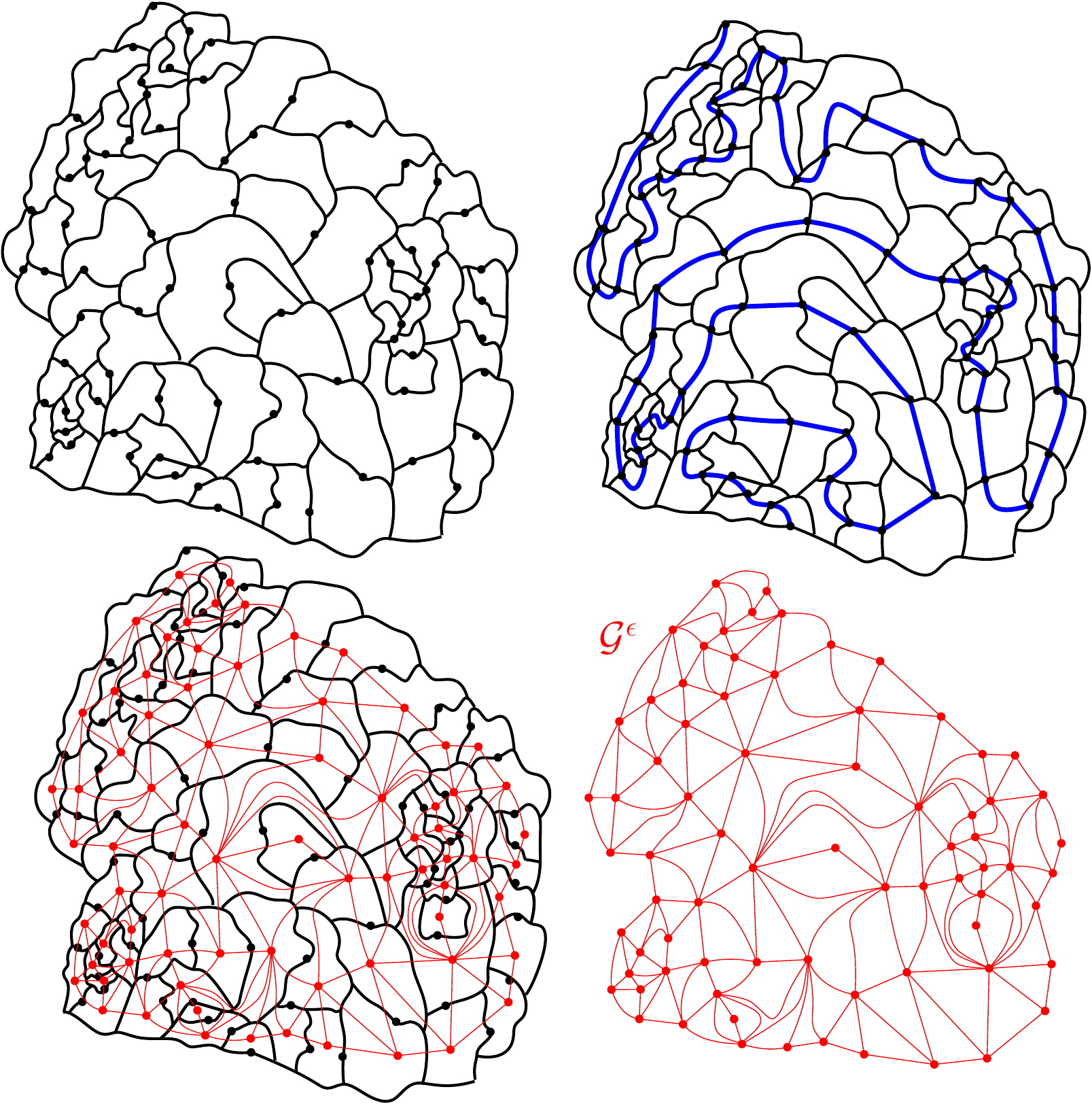}
	\caption{\textbf{Top left:} A space-filling SLE curve $\eta$ for $\kappa \geq 8$ divided into cells $\eta([a - {\eps}, a])$ for a collection of $a \in {\eps} \Z$. \textbf{Top right:} The
		same curve with a blue path showing the order in which cells are traversed by $\eta$.	\textbf{Bottom left:} A point in each cell, corresponding to a vertex in ${\Geps}$, is displayed in red
		and red edges are drawn to adjacent neighbors. \textbf{Bottom right:} The same as bottom left but with the black edges removed --- this illustrates the embedding of ${\Geps}$ into $\C$.
		A similar figure has appeared previously as~\cite[Figure 4]{gms-harmonic}.} \label{fig:sle-lqg-embedding}
\end{figure}

\subsubsection{Harmonic balls}	\label{subsubsec:harmonic-ball-intro}
In this article we show that IDLA on the SLE/LQG embedding of the mated-CRT map converges to a {\it harmonic ball}.
\begin{definition}[Harmonic ball] \label{def:harmonic-ball}
	Fix $\gamma \in (0,2)$ and let $\mu_h$ be the LQG measure \eqref{eq:lqg-measure} where $h$ is a variant of the Gaussian free field.
	A $\gamma$-LQG {\it harmonic ball} is a domain $\Lambda \subset \C$ which satisfies the following mean-value property for harmonic functions:
	\begin{equation}
	\int_{\Lambda} f(y) d \mu_h(y) = \mu_h(\Lambda) f(0), \quad \forall \mbox{$f:\C \to \R$ harmonic in a neighborhood of $\cl(\Lambda)$}. 	
	\end{equation}
\end{definition}
In a companion paper \cite{bou2022harmonic}, it is shown that if $\gamma\in(0,2)$ and $h$ is a suitable variant of the GFF on $\C$, then there is a unique growing family of LQG harmonic balls $\{\Lambda_t\}_{t>0}$, parameterized so that $\mu_h(\Lambda_t)=t$. These harmonic balls are constructed directly in the continuum via a certain optimization problem involving the measure $\mu_h$. This optimization problem, as well as several important properties of harmonic balls proved in \cite{bou2022harmonic}, are used in our proof. A statement of the facts about LQG harmonic balls which will be needed later appears in Theorem \ref{theorem:harmonic-balls} below. It is also shown in~\cite{bou2022harmonic} that LQG harmonic balls are neither Lipschitz domains nor LQG metric balls.

\subsection{Main results}
Fix $\gamma \in (0,2)$ and for ${\eps} > 0$ let ${\Geps}$ be the ${\eps}$-mated CRT map as in Definition \ref{def:mated-crt-map},
with vertices ${\VGeps}$ embedded into $\C$ via the SLE/LQG embedding as in Definition \ref{def:sle-lqg-embedding}. 
For a set $X \subset \C$ and $\delta > 0$ we denote the $\delta$-outer and $\delta$-inner neighborhoods of $X$ by 
\begin{equation} \label{eq:outer-inner-neighborhoods}
\begin{aligned}
\Bdeltap(X) &:= \{ x \in \C : \dist(x, X) \leq \delta\} \\
\Bdeltam(X) &:= \{ y \in X : \dist(y, X^c) \geq \delta\} 
\end{aligned}
\end{equation}
where $\dist$ denotes Euclidean distance.

\subsubsection{Internal DLA}
Our notation for the IDLA cluster on ${\Geps}$ is as follows. Begin with $A_{{\eps}}(0) = \emptyset$. For $m\geq 1$, inductively let $A_{{\eps}}(m)$ be the union of $A_{{\eps}}(m-1)$ and the first point at which a simple random walk on ${\Geps}$ started at the origin $ 0 \in {\VGeps} $ 
exits $A_{{\eps}}(m-1)$. By filling in the mated-CRT map cells corresponding to vertices in $A_{\eps}$ (Definition~\ref{def:sle-lqg-embedding}) we may map a subset $A_{{\eps}} \subset {\VGeps}$ to a subset of $\C$ via 
\begin{equation} \label{eq:vertex-scaling}
\overline A_{{\eps}} := \bigcup_{a \in A_{{\eps}}} \Heps{a}.
\end{equation}
\begin{theorem} \label{theorem:convergence-of-IDLA} 
	Let $\{\Lambda_t\}_{t\geq 0}$ be the growing family of $\gamma$-LQG harmonic balls with $\mu_h(\Lambda_t) = t$, as discussed just after Definition \ref{def:harmonic-ball}, with $h$ the field corresponding to the $\gamma$-quantum cone, as in Definition~\ref{def:sle-lqg-embedding}.
	For each $\delta > 0$ and $t > 0$, it holds with probability tending to 1 as ${\eps}\to 0$ that the union of the cells corresponding to the time $\lfloor t{\eps}^{-1}\rfloor$-IDLA cluster approximates $\Lambda_t$ in the following sense:
	\[
	\Bdeltam(\Lambda_t)	\subset \overline A_{{\eps}}(\lfloor t {\eps}^{-1} \rfloor) \subset \Bdeltap(\Lambda_t),
	\]
	where $B_{\delta}^{\pm}$ are as in \eqref{eq:outer-inner-neighborhoods}.
\end{theorem}

\begin{remark}[Convergence under the Tutte embedding] \label{remark:tutte}
	One may  also prove a version of Theorem~\ref{theorem:convergence-of-IDLA} for the disk version of the mated-CRT map under the Tutte embedding, as depicted in Figure~\ref{fig:IDLA-tutte-embedding}.
	The quantum disk is the canonical LQG surface with the topology of the disk~\cite{hrv-disk,duplantier2014liouville}. 
	In~\cite{gms-tutte}, the authors introduce a disk version of the mated-CRT map with a distinguished boundary and define its Tutte embedding (the definition of the Tutte embedding is easier for planar maps with boundary, which is why one wants to consider the disk version of the mated-CRT map). 
	They also show that the Tutte embedding of the disk version of the mated-CRT map is close, in the uniform distance, to its SLE/LQG embedding. 
	Using the fact that the generalized function representing the quantum disk behaves like a GFF away from its boundary, one can extend the results of~\cite{bou2022harmonic} to define a growing family of harmonic balls on the quantum disk started at a uniformly sampled interior point and stopped at the first time when they hit the boundary. 
	Via another absolute continuity argument, one can extend the result of Theorem~\ref{theorem:convergence-of-IDLA} to get that the IDLA clusters on the mated-CRT map with the disk topology, stopped at the first time they hit the boundary, converge to harmonic balls on the quantum disk under the SLE/LQG embedding. 
	We will not provide the details of these absolute continuity arguments in this paper, but see~\cite[Section 7.4]{berestycki2020random} and Theorem \ref{theorem:harmonic-balls} below for similar arguments.
	Using the results of~\cite{gms-tutte}, one gets that IDLA clusters on the disk version of the mated-CRT map also converge to harmonic balls under the Tutte embedding. 
\end{remark}

\subsubsection{Divisible sandpile}\label{subsubsec:div-sandpile-informal-intro}
Our proof of convergence of IDLA proceeds by first establishing an analogous result for an auxiliary particle system which may be thought of as the `expected value' of IDLA. 

The {\it divisible sandpile} is a deterministic diffusion process on a graph $G$,
which, to the best of our knowledge, was first introduced by Zidarov \cite[page 108-118]{zidarov1990inverse}
and (independently) studied by Levine in his thesis \cite{levine-thesis}. Start with masses (real numbers) on vertices of a graph and redistribute the mass according to the following rule. Whenever a vertex has mass $t$ larger than one, it is {\it unstable} and {\it topples} by distributing the excess mass, $(t-1)$, to each of its neighbors equally. This process, called {\it stabilizing}, continues until every vertex has mass less than or equal to one. Although stabilizing may take infinite time, if we have a finite amount of mass on an infinite graph, the end configuration exists and does not depend on the order in which unstable sites are toppled \cite{levine2016divisible}. From another perspective, the divisible sandpile is simply the Jacobi iterative algorithm, \eg, \cite[Chapter 10]{golub-matrix-computations}, for solving the Poisson equation.

In one variant of the model, called the {\it single-source} divisible sandpile, we start with a pile of mass $t > 0$ at a root vertex $o$ in $G$ and stabilize. The final collection of vertices which contain some mass is the {\it cluster}. In each IDLA convergence theorem recalled in Section~\ref{subsubsec:idla-informal-intro}, it has been shown that the limit shape of the single-source divisible sandpile cluster coincides with that of IDLA. This relationship was first observed and established on $\Z^d$ by Levine-Peres in \cite{levine2009strong}. See ~\cite{ruszel-divisible} for a survey relating IDLA and the divisible sandpile. 

In our setting, the IDLA cluster also shares the same limit shape as the divisible sandpile. Denote by $D_{{\eps}}(t)$ the divisible sandpile cluster starting with mass $t > 0$ at the origin in ${\Geps}$. 
\begin{theorem} \label{theorem:convergence-of-div-sandpile}	
	Theorem \ref{theorem:convergence-of-IDLA} holds with the divisible sandpile cluster 
	started with mass $t {\eps}^{-1}$ at the origin,  $\overline D_{{\eps}}(t {\eps}^{-1})$, in place of the IDLA cluster 
	started with $\lfloor t {\eps}^{-1} \rfloor$ explorers at the origin, 
	$\overline A_{{\eps}}(\lfloor t {\eps}^{-1} \rfloor)$. 
\end{theorem}
The divisible sandpile has an explicit description in terms of a discrete optimization problem, which we recall in Section \ref{sec:divisible-sandpile} below. 
In particular, from this optimization problem it is easy to deduce that the divisible sandpile cluster enjoys a certain discrete mean-value property; the clusters are 
{\it discrete harmonic balls} \cite[Equation (5)]{levine2010scaling}. From this, it is no surprise that divisible sandpile clusters converge to continuum harmonic balls.

We would also like to point out that the divisible sandpile appears to be the {\it only} way to find discrete harmonic balls. In fact, this property of the divisible sandpile was used to establish logarithmic fluctuations of IDLA around its limit shape in \cite{jerison2012logarithmic, jerison2013internal, jerison2014internal}.
\subsection{Open problems}
We suggest several directions for further study. The most pressing problem is to extend our convergence result to {\it other} 
random planar maps. This problem is mentioned in the companion work~\cite[Problem 1]{bou2022harmonic}.
\begin{problem} \label{prob:planar-map}
	Show that the scaling limit of IDLA and the divisible sandpile on random planar maps, other than mated-CRT maps, is also described by $\gamma$-LQG harmonic balls.
	For example, on a uniform random planar map establish convergence of IDLA to a $\sqrt{8/3}$-LQG harmonic ball. 
\end{problem}
Our proofs of Theorems~\ref{theorem:convergence-of-IDLA} and~\ref{theorem:convergence-of-div-sandpile} rely heavily on the convergence of random walk on the mated-CRT map to Liouville Brownian motion. This fact has not been established for random planar maps besides the mated-CRT map. We speculate that a solution to Problem~\ref{prob:planar-map} which does not use the convergence of random walk might rely on special symmetries of random planar maps, for example, the peeling process for uniform planar maps \cite{angel-peeling, curien-legall-peeling, curien-peeling-notes}. 

The next two questions concern IDLA on the mated-CRT map. While we have established a qualitative convergence theorem much more quantitative results are known on $\Z^d$. Lawler~\cite{lawler1995subdiffusive} used fine estimates on the Green's function in $\Z^d$ to establish subdiffusive fluctuations of IDLA around its limit shape. Later, Asselah-Gaudillère~\cite{asselah2013logarithmic, asselah2013sublogarithmic} and independently Jerison-Levine-Sheffield~\cite{jerison2012logarithmic, jerison2013internal} established logarithmic fluctuations. We do not expect logarithmic fluctuations in our setting, mainly because the cell sizes in the mated-CRT map have  a `multi-fractal' behavior. See Figure \ref{fig:idla-interface}.
\begin{problem} \label{prob:rate of convergence}
	Provide bounds on the rate of convergence of IDLA on the mated-CRT map. 
\end{problem}
\begin{figure}
	\includegraphics[width=0.45\textwidth]{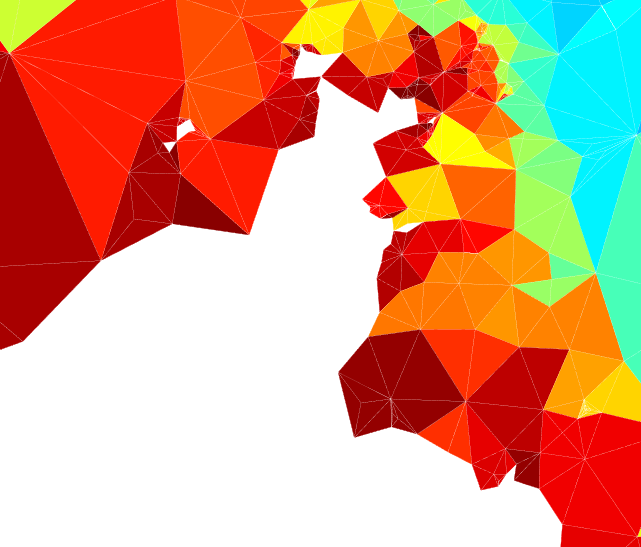}
	\caption{A close up of the interface in Figure \ref{fig:IDLA-tutte-embedding}.} \label{fig:idla-interface}
\end{figure}

After proving logarithmic fluctuations in $\Z^d$, in a remarkable work, Jerison-Levine-Sheffield showed that the scaling limit of the fluctuations themselves exists.
Namely, fluctuations of IDLA on $\Z^d$ converge weakly to a variant of the Gaussian free field \cite{jerison2014internal, jerison2014chapter}. Darrow extended 
the analysis of Jerison-Levine-Sheffield on $\Z^2$ to the case of multiple sources \cite{darrow2020scaling}.
\begin{problem} \label{prob:fluctuations}
	Do the fluctuations of IDLA on the mated-CRT map around its limit shape have a scaling limit? 
\end{problem}

\subsection{Paper and proof outline}
Our proof uses three main inputs; the pairing of IDLA and the divisible sandpile, 
regularity properties of the mated-CRT map, and the convergence of random walk on the mated-CRT map to Liouville Brownian motion. 

The connection between IDLA and the divisible sandpile was implicit in the first work on IDLA by Lawler-Bramson-Griffeath \cite{lawler1992internal}
and made explicit in the thesis of Levine \cite{levine-thesis}. In particular, following the arguments of \cite{lawler1992internal},
a lower bound on the divisible sandpile cluster implies the same lower bound for the IDLA cluster. 
The divisible sandpile, being deterministic and enjoying a variational description, \ie, the {\it least action principle} for its {\it odometer} (see Section \ref{sec:divisible-sandpile} below for definitions), is more tractable than IDLA. 

In the cases considered so far, \eg, \cite{levine2009strong,  huss-sava-comb, MR3957187}, 
convergence of the divisible sandpile is proved by explicitly bounding its odometer via sharp estimates on stopped Green's functions.  These sharp estimates are unavailable in our setting and hence we proceed with a somewhat more robust method. Roughly, we use the fact that the mated-CRT map has an elliptic Harnack inequality and random walk converges to Brownian motion to show that the divisible sandpile odometer converges to its continuum counterpart, the odometer for {\it Hele-Shaw flow} (the solution to the optimization problem used to construct harmonic balls). 

We start by recalling in Section \ref{sec:preliminaries} several mated-CRT map estimates and the invariance principle for random walk 
proved in \cite{berestycki2020random, gwynne2019harmonic}. In this section we also record some basic properties of harmonic balls and Hele-Shaw flow on an LQG surface which we proved in a companion paper \cite{bou2022harmonic}.

In Section \ref{sec:dirichlet-problem}, we use the estimates proved in \cite{berestycki2020random, gwynne2019harmonic} to establish a 
general tightness lemma (Lemma \ref{lemma:discrete-potential-tightness}) for bounded functions with bounded graph Laplacian 
on the mated-CRT map. We then use this lemma to establish the convergence of solutions to the discrete Dirichlet problem on the mated-CRT map to solutions of the continuum Dirichlet problem (Lemma \ref{lemma:discrete-dirichlet-problem}) and convergence of the discrete Green's kernel to the continuum Green's function (Lemma \ref{lemma:uniform-convergence-of-greens-function}). 

In Section \ref{sec:divisible-sandpile}, we define the odometer for the divisible sandpile model on the mated-CRT map and state some basic properties of it which are proved in the Appendix, Section \ref{sec:obstacle-appendix}. 
We then prove convergence of the divisible sandpile odometer to its continuum counterpart. Our proof method is similar to, but easier than, that used to show convergence of the {\it Abelian sandpile odometer} \cite{bourabee-sandpile-2021, pegden-smart-sandpile-2013, barles-souganidis-convergence-1991}.
In Section \ref{sec:divisible-sandpile} we also observe that the convergence of the divisible sandpile odometer implies a lower bound on the divisible sandpile cluster: \ie, for fixed $t , \delta > 0$ it holds with probability tending to 1 as ${\eps}\to 0$ that
\begin{equation} \label{eq:div-lower-outline}
\Bdeltam(\Lambda_t) \subset \overline D_{{\eps}}( t {\eps}^{-1} ).
\end{equation}

In Section \ref{sec:idla-lower-bound} we apply the technique of Lawler-Bramson-Griffeath \cite{lawler1992internal} together with the divisible sandpile lower bound~\eqref{eq:div-lower-outline} and random walk estimates proved in \cite{berestycki2020random} to prove an analogous lower bound for the IDLA cluster.
In fact, the lower bound holds for the {\it stopped} IDLA cluster where random walkers which exit (a mated-CRT map approximation of) $\Lambda_t$ are stopped. 
By the Abelian property of IDLA, Lemma \ref{lemma:abelian-IDLA}, the IDLA cluster can be built by first generating the stopped IDLA cluster and then restarting
the stopped random walkers. 
Thus, an upper bound of the form 
\begin{equation} \label{eq:idla-upper-outline}
\overline A_{{\eps}}(\lfloor t {\eps}^{-1} \rfloor) \subset \Bdeltap(\Lambda_t)
\end{equation} 
will follow if one can show that the restarted random walks cannot travel further than a Euclidean distance of $\delta$ away from $\Lambda_t$. 

In fact, if the limit shape has the correct size, \ie, has volume equal to the total number of walkers, 
and has a measure zero boundary, then by the lower bound and conservation of mass, the number of restarted walkers can be made arbitrarily small.  These two limit shape properties are true for an LQG harmonic ball, as proved in the companion paper \cite{bou2022harmonic}.  Therefore, the proof of the upper bound then reduces to controlling the behavior of order $\delta {\eps}^{-1}$ restarted random walks near the boundary of $\Lambda_t$.  We do this by combining the arguments 
of \cite{duminil2013containing} together with those from LQG theory. 

Specifically, it was shown by Duminil-Copin-Lucas-Yadin-Yehudayoff \cite{duminil2013containing} that if a graph is regular enough, namely is volume doubling, and has a metric which can be controlled by the Euclidean metric, then a small number of walkers cannot spread very far. The mated-CRT map does not enjoy these properties, even approximately, and thus we cannot apply the Duminil-Copin-Lucas-Yadin-Yehudayoff argument directly. Instead, we use methods from LQG theory to provide weak harmonic measure estimates for random walk on the mated-CRT map. A modification of the iterative argument of Duminil-Copin-Lucas-Yadin-Yehudayoff then yields the upper bound.  See the beginning of Section \ref{sec:idla-upper-bound} for a detailed outline of the argument. 

In Section \ref{sec:div-upper-bound} we use the harmonic measure estimates proved in Section \ref{sec:idla-upper-bound} to prove an upper bound on the divisible sandpile cluster. The argument here is a simplified, discrete version of the upper bound for LQG harmonic balls from \cite[Section 6]{bou2022harmonic}. 

All of the aforementioned results were proved conditional on $t$ being small. In Section \ref{sec:proof-of-convergence} we combine the upper and lower bounds together with a scaling argument to remove this constraint on $t$ and complete the proof.

\subsection{Notation and conventions}
\begin{itemize}
	\item Inequalities/equalities between functions/scalars are interpreted pointwise.
	\item For a set $D \subset \C$, $\partial D$ denotes its topological boundary, $\cl(D) = D \cup \partial D$ its closure, 
	and $\inte(D)$ its interior. 
	\item For a set $A \subset {\VGeps}$, $\partial A$ is the set of vertices in ${\VGeps} \setminus A$ which are joined by an edge to
	a vertex in $A$ and $\cl(A) = A \cup \partial A$ is its closure.  
	The set $\inte(A)$ are the vertices in $A$ which do not share an edge with a vertex in $A^c$. 
	\item For two sets $X, Y \subset \C$, we write $X \Subset Y$ if $\cl(X) \subset Y$. 
	\item $B_r(x)$ denotes the open ball of Euclidean radius $r > 0$ centered at $x \in \C$, when $x$ is omitted, the ball is centered at 0. 
	When $D \subset \C$, we write $B_{r}(D) = \{x \in \C: \dist(x, D) < r\}$.
	\item For $0 < r_1 < r_2$, denote an open annulus centered at $z \in \C$ by
	\begin{equation} \label{eq:annulus}
	\A_{r_1, r_2}(z) = B_{r_2}(z) \setminus \cl(B_{r_1}(z))
	\end{equation}
	and $\A_{r_1, r_2} := \A_{r_1,r_2}(0)$. 
	\item Let $\{E^{r}\}_{z > 0}$ be a one-parameter family of events. 
	We say that $E^r$ occurs with polynomially high probability as $r \to 0$
	if there exists $p>0$ such that $\P[E^r] \geq 1 - O(r^p)$. 
	\item For two sets $X, Y \subset \C$, we define $\dist(X,Y) = \inf_{x \in X, y \in Y} \dist(x,y)$ where $\dist$ denotes the Euclidean distance between two points. 
	\item For a set $D \subset \C$, $C(D)$ denotes the set of continuous functions on  $D$. 
\end{itemize}

\subsection*{Acknowledgments}
We thank an anonymous referee for helpful comments on an earlier version of this article.
A.B. thanks Charlie Smart and Bill Cooperman for useful discussions. 
A.B. was partially supported by NSF grant DMS-2202940 and a Stevanovich fellowship. 
E.G. was partially supported by a Clay research fellowship.

	\section{Preliminaries} \label{sec:preliminaries}
	In this section we recall several preliminary results on the Gaussian free field, Liouville quantum gravity, space-filling SLE, harmonic balls, random walk on the mated-CRT map, and Liouville Brownian motion. Of note, we record the convergence of random walk on the mated-CRT map to Liouville Brownian motion \cite{berestycki2020random}
and the regularity properties of $\gamma$-LQG harmonic balls established in \cite{bou2022harmonic}.

\subsection{Gaussian free field}
In this paper, we work with a particular kind of Gaussian free field $h$ which arises naturally from the scaling limit of random planar maps. 
This field corresponds to an LQG surface called a {\it $\gamma$-quantum cone}. For our purposes we just define the field and give its scaling properties
--- a complete definition of the surface is given in \cite{duplantier2014liouville}.  

We start by defining the {\it whole-plane} Gaussian free field (GFF) $h^{\C}$ which is the centered Gaussian random generalized 
function on $\C$ with covariances
\begin{equation} \label{eq:cov-kernel}
\Cov(h^\C(x), h^\C(y)) := \log \frac{\max(|x|,1) \max(|y|,1)}{|x-y|} ,\quad\forall x,y\in\C .
\end{equation} 
The GFF $h^{\C}$ is not defined pointwise as the covariance kernel in~\eqref{eq:cov-kernel} diverges to $\infty$ as $x\to y$. 
However, for $z\in\C$ and $r>0$, one can define the average of $h^{\C}$ over the circle of radius $r$ centered at $z$, which we denote by $h^{\C}_r(z)$~\cite[Section 3.1]{duplantier2011liouville}. 	The whole plane GFF is sometimes defined modulo additive constant. Our choice of covariance in~\eqref{eq:cov-kernel} corresponds to fixing this additive constant so that $h^\C_1(0) = 0$ (see, e.g.,~\cite[Section 2.1.1]{vargas2017lecture}). 

The circle average of $h^{\C}$, $t \to h^{\C}_{e^{-t}}(0)$ is a Brownian motion, \cite[Section 3.1]{duplantier2011liouville}.
The field corresponding to the $\gamma$-quantum cone is constructed by enforcing its circle-average to correspond to a particular conditioned Brownian motion. 
Fix $\gamma \in (0,2)$ and let 
\begin{equation} \label{eq:Q}
Q := \frac{2}{\gamma} + \frac{\gamma}{2}.
\end{equation}

\begin{definition}[Circle average embedding of the $\gamma$-quantum cone] \label{def:quantum-cone}
	The circle average embedding of the $\gamma$-quantum cone is the distribution $h$ defined as follows. 
	Let $B$ be a standard linear Brownian motion and let $\hat{B}$ be an independent standard linear Brownian motion conditioned so that $\hat{B}_t + (Q-\gamma) t > 0$ for all $t > 0$.
	Let
	\begin{equation}
	A_t := \begin{cases}
	B_t + \gamma t , \quad &t \geq 0 \\
	\hat{B}_{-t} + \gamma t ,\quad &t < 0 .
	\end{cases}
	\end{equation}  
	The distribution $h$ is defined so that $t \to h_{e^{-t}}(0)$ has the same law as the process $A$ and $h - h_{|\cdot|}$ is independent 
	from $h_{|\cdot|}$ and has the same law as the analogous process for the whole-plane GFF, $h^{\C}$. 
\end{definition}

We note that the field $h$ in Definition \ref{def:quantum-cone} has the property that $\sup \{ r > 0 : h_r(0) + Q \log r = 0 \} = 1$. Furthermore, it is immediate from Definition~\ref{def:quantum-cone} that $h$ restricted to the unit disk agrees in law with the corresponding restriction of a whole-plane GFF plus $ - \gamma \log |\cdot|$, normalized so that its circle average over the disk is zero:
\begin{equation} \label{eq:gff}
h|_{B_1} \overset{d}{=} ( h^{\C} - \gamma \log |\cdot| )|_{B_1} .
\end{equation}

The law of the whole-plane GFF, viewed modulo additive constant, is invariant under complex affine transformations of $\C$. This implies the following scaling property for $h^{\C}$,
\begin{equation} \label{eq:gff-scaling}
h^{\C} \overset{d}{=} h^{\C}(x\cdot +y) - h^{\C}_{|x|}(y), \quad \forall x \in \C \setminus \{0\}, \quad \forall  y \in \C. 
\end{equation}
The distribution $h$ has a similar scale invariance. For $b > 0$, let 
\begin{equation} \label{eq:cone-scale-radius}
R_b := \sup \{ r > 0 : h_r(0) + Q \log r = \frac{1}{\gamma} \log b\}.
\end{equation}
Then, by \cite[Proposition 4.13(i)]{duplantier2014liouville},
\begin{equation} \label{eq:quantum-cone-scaling}
h(R_b \cdot) + Q \log R_b - \frac{1}{\gamma} \log b \overset{d}{=} h. 
\end{equation}

\subsection{Liouville quantum gravity} \label{subsec:lqg}
\textit{Liouville quantum gravity (LQG)} is a one-parameter family of random fractal surfaces which were introduced by Polyakov in the 1980s in the context of bosonic string theory~\cite{polyakov-qg1}. 
We give some basic properties of LQG and refer the interested reader to the introductory texts \cite{berestycki2021gaussian,sheffield-icm,gwynne2020random}.

Let $\mu_h$ denote the {\it $\gamma$-LQG area (Liouville) measure} associated to $h$, the $\gamma$-quantum cone defined in \eqref{def:quantum-cone}. One of the several ways of defining $\mu_h$ is as the a.s.\ weak limit 
\begin{equation} \label{eq:measure-lim}
\mu_h = \lim_{\epsilon \rightarrow 0} \epsilon^{\gamma^2/2} e^{\gamma h_\epsilon(z)} \,dz
\end{equation}
where $dz$ denotes Lebesgue measure and $h_\epsilon(z)$ is the circle average~\cite{duplantier2011liouville, sheffield2016field}.
In fact, the measure $\mu_{\tilde h}$ can be constructed for any random generalized function $\tilde h$ of the form  $h + f$ where $f$ is a possibly random continuous function
and $h$ is the whole-plane GFF, defined in \eqref{eq:gff-scaling}. In particular, this includes the $\gamma$-quantum cone from Definition~\ref{def:quantum-cone}.
For later use, we record some basic properties of the LQG measure. 
\begin{fact}[LQG measure] \label{fact:lqg-measure}
	The LQG area measure $\mu_h$ satisfies the following properties.
	\begin{enumerate}[label=\Roman*.]
		\item {\bf Radon measure.} A.s., $\mu_h$ is a non-atomic Radon measure.
		\item {\bf Locality.} For every deterministic open set $U \subset \C$,
		$\mu_h(U)$ is given by a measurable function of 
		$h \vert_{U}$.
		\item {\bf Weyl scaling.} A.s., $e^{\gamma f} \cdot \mu_h = \mu_{h + f}$
		for every continuous function $f: \C \to \R$. 
		\item {\bf Conformal covariance.} A.s., the following is true. Let $U , \tilde U \subset \C$ be open and let $\phi$ be a conformal
		map from $\tilde U$ to $U$. Then, with $Q$ as in~\eqref{eq:Q}, 
		\begin{equation} \label{eq:measure-coord}
		\mu_{h \circ \phi + Q \log |\phi'|}(A) = \mu_h(\phi(A)) \quad \mbox{for all Borel measurable $A \subset \tilde U$}.
		\end{equation}
	\end{enumerate}
\end{fact}

The first three properties in Fact~\ref{fact:lqg-measure} are immediate from the definition~\eqref{eq:measure-lim}. The conformal covariance property was proven to hold a.s.\ for a fixed conformal map in~\cite[Proposition 2.1]{duplantier2011liouville} and extended to all conformal maps simultaneously in~\cite{sheffield2016field}.

\subsection{Harmonic balls}
A harmonic ball is a domain which satisfies the mean-value property for harmonic functions with respect to the LQG measure --- Definition \ref{def:harmonic-ball}. 
These balls are constructed in a companion work via the following optimization problem. 
For each $t > 0$ and ball $B_r = B_r(0)$, let 
\begin{equation} \label{eq:least-super-solution}
\overline w_t^{B_r} = \inf\{ w \in C(\cl(B_r)) : \Delta w \leq \mu_h \mbox{ in $B_r$}  \mbox{ and }  w \geq  - t G_{B_r}(0, \cdot) \mbox{ in $\cl(B_r)$}\},
\end{equation}
where $C(\cl(B_r))$ denotes the set of continuous functions on the closed ball, $\Delta w$ is interpreted in the distributional sense, $G_{B_r}$ is the Green's function for standard Brownian motion on $B_r$ with zero boundary conditions, and the infimum is pointwise. 
The definition in \eqref{eq:least-super-solution} is the {\it obstacle problem} formulation of Hele-Shaw flow~\cite{sakai1984solutions} and may be thought of as the continuum 
analogue of the odometer for the divisible sandpile, defined below in Section \ref{sec:divisible-sandpile}.

\begin{theorem}[Combination of Theorems 1.1 and 5.5 and Lemma 3.2 in \cite{bou2022harmonic}] \label{theorem:harmonic-balls}
	On an event of probability one, there exists a unique family of harmonic balls $\{\Lambda_t\}_{t > 0}$
	satisfying the following properties:
	\begin{enumerate}[label=(\alph*)]
		\item For each $t > 0$, $\mu_h(\Lambda_t) = t$, $\mu_h(\partial \Lambda_t) = 0$, and $\Lambda_t$ is equal to the interior of its closure.
		\item The domains $\Lambda_t$ are bounded, connected, open, contain the origin, increase continuously in $t$ (in the Hausdorff topology), and satisfy $\cap_{t > 0} \Lambda_t = \{0\}$.
		\item If, for some $r > 0$, $\Lambda_t \subset B_r$, then, for all $\rho > r$, 
		\[
		\Lambda_t = \inte \left( \cl \left( \{ x \in B_\rho : \overline w_t^{B_\rho}(x) > - t G_{B_\rho}(0, x)\} \right) \right).
		\] 
		Moreover, $\overline w_t^{B_\rho}$ is continuous and
		satisfies
		\[
		\begin{cases}
		\Delta \overline w_t^{B_\rho} = \mu_h |_{\Lambda_t}  \quad \mbox{on $B_\rho$} \\
		\overline w_t^{B_\rho} = 0 \quad \mbox{on $\partial B_\rho$}.
		\end{cases}
		\]
	\end{enumerate}
	
\end{theorem}

Strictly speaking, \cite{bou2022harmonic} considers the case where $h$ is not the quantum cone but rather a whole-plane GFF plus a log-singularity.
Transferring the proof to the quantum cone is a standard absolute continuity argument which we now provide. 
\begin{lemma} \label{lemma:existence-up-to-fixed-time}
	On an event of probability one, there exists a unique family of harmonic balls $\{\Lambda_t\}_{0 < t < T}$
	which satisfies the properties of Theorem \ref{theorem:harmonic-balls} (with $t$ restricted to be in $(0,T)$), where 
	\[
	T = \sup \{ t > 0 : \Lambda_t \subset B_{1/2}\}.
	\]
\end{lemma}
\begin{proof}
	By \eqref{eq:gff}, the restrictions to the unit disk of the quantum cone field $h$ and the whole-plane GFF plus $-\gamma \log |\cdot|$ agree in law.
	Since harmonic balls depend locally on the field~\cite[Proposition 1.3]{bou2022harmonic}, this equality in law together with  \cite[Theorems 1.1 and 5.5 and Lemma 3.2]{bou2022harmonic} directly implies the existence of a family of harmonic balls $\{ \Lambda_t \}_{0 < t < T}$ satisfying the properties in Theorem \ref{theorem:harmonic-balls}. 	
	Uniqueness of such a family is guaranteed by \cite[Proposition 5.6]{bou2022harmonic} 
	(that proposition is stated for a family of harmonic balls defined for all times, but the proof still works if the harmonic balls are only defined up to some fixed time). 
\end{proof}
We will now extend the above construction to all time using the scale invariance property of the $\gamma$-quantum cone.
The following is the quantum cone analog of \cite[Lemma 5.7]{bou2022harmonic} (which gives a similar statement for the LQG harmonic balls associated with the whole-plane GFF).  
\begin{lemma} \label{lemma:scale-invariance-up-to-fixed-time}
	For $b >0$, define the random radius $R_b > 0$ as in~\eqref{eq:cone-scale-radius}.
	A.s., there exists a unique family of harmonic balls 
	$\{\Lambda_{t}\}_{0 < t < b T^b}$
	such that $\{R_b^{-1} \Lambda_{b t}\}_{0 < t < T^b}$
	has the same law as $\{\Lambda_t\}_{0 < t < T}$, 
	where
	$T$ is as in Lemma \ref{lemma:existence-up-to-fixed-time}, 
	and 
	\[
	T^b = \sup\{ t > 0 : R_b^{-1} \Lambda_{b t} \subset B_{1/2} \}.
	\]
	In fact, we have the equality of joint laws
	\[
	\left( h, \{\Lambda_t\}_{0 < t < T} \right) \overset{d}{=}\left( h(R_b \cdot) + Q \log R_b - \frac{1}{\gamma} \log b, \{R_b^{-1} \Lambda_{b t}\}_{0 < t < T^b} \right).
	\] 
\end{lemma}
\begin{proof}
	Fix $b > 0$ and let
	\[
	h^b := h(R_b \cdot) + Q \log R_b - \frac{1}{\gamma} \log b,
	\]
	so that, as in \eqref{eq:quantum-cone-scaling}, $h^b \overset{d}{=} h$. By Lemma \ref{lemma:existence-up-to-fixed-time}, a.s.\ there exists a unique family of harmonic balls $\{\Lambda^b_t\}_{0 < t < T^b}$ which satisfies the properties of Theorem \ref{theorem:harmonic-balls} for the measure $\mu_{h^b}$ (with $t$ restricted to be in $(0,T^b)$), where
	\[
	T^b = \sup\{ t > 0 : \Lambda^b_{t} \subset B_{1/2} \}.
	\]
	Define
	\[
	\Lambda_{b t}  :=   R_b  \Lambda^b_t, \quad \forall 0 < t < T^b
	\]
	and note that by the properties of $\Lambda^b_t$, the LQG coordinate change formula, and Weyl scaling (Fact \ref{fact:lqg-measure}), a.s.,
	\[
	t = \mu_{h^b}(\Lambda^b_t) = b^{-1} \mu_h(\Lambda_{ b t}), \quad \forall 0 < t < T^b.
	\]
	For similar reasons, the domains $\{\Lambda_{t}\}_{0 < t < b T^b}$ are a.s.\ harmonic balls and satisfy parts (a) and (b) of Theorem \ref{theorem:harmonic-balls} for the measure $\mu_h$
	(with $t$ constrained to be less than $b T^b$). 
	These properties are enough to apply \cite[Proposition 5.6]{bou2022harmonic} (which works if the harmonic balls are only defined up to some fixed time) and conclude that $\{\Lambda_{t}\}_{0 < t < b T^b}$ is the unique such family. As $h^{b}$ has the same law as $h$, this shows the equality of joint laws. 
\end{proof}

\begin{proof}[Proof of Theorem \ref{theorem:harmonic-balls}]

	By Lemma \ref{lemma:scale-invariance-up-to-fixed-time}, a.s.\ for every integer $b \geq 0$, 
	there exists a unique family of harmonic balls $\{\Lambda_{t}\}_{0 < t < b T^b}$
	which satisfies the conditions of the theorem statement up to time $b T^b$,
	where $T^b$ is as in Lemma \ref{lemma:scale-invariance-up-to-fixed-time}. Since $T^b \overset{d}{=} T$ for each $b>0$, where $T$ is as in Lemma \ref{lemma:existence-up-to-fixed-time}, and $T$ is a strictly positive random variable, we deduce that $\lim_{b \to \infty} b T^b = \infty$ in probability. This complete the proof. 
\end{proof}

By applying the proof of Lemma \ref{lemma:scale-invariance-up-to-fixed-time} to the infinite family of harmonic balls given by Theorem \ref{theorem:harmonic-balls}, we have the following.
\begin{lemma} \label{lemma:scale-invariance}
	For each $b > 0$, we have the equality of the joint laws
	\[
	\left( h, \{\Lambda_t\}_{t > 0} \right) \overset{d}{=}\left( h(R_b \cdot) + Q \log R_b - \frac{1}{\gamma} \log b, \{R_b^{-1} \Lambda_{b t}\}_{t > 0} \right).
	\] 
	where the random radius $R_b > 0$ is as in~\eqref{eq:cone-scale-radius}. 
\end{lemma}

\subsection{Whole plane space-filling SLE}
The Schramm-Loewner evolution ($\text{SLE}_{\kappa}$) for $\kappa > 0$ was introduced by Schramm in \cite{schramm-sle}
as the only conformally invariant family of random curves which satisfy the domain Markov property. There are three phases of  $\text{SLE}_{\kappa}$: the curves are simple when $\kappa \in (0,4]$,  have self-intersections but are not space-filling for $\kappa \in (4,8)$, and are space-filling when $\kappa \geq 8$. 

In the embedding of the mated-CRT map, we work with the {\it whole plane space-filling SLE} $\eta$, a variant of SLE introduced in 
\cite[Sections 1.2.3 and 4.3]{miller2017imaginary}.  This is a random space-filling curve $\eta$ in $\C$ 
from $\infty$ to $\infty$ which fills all of $\C$ and does not intersect the interior of its past. 

For context, we informally describe how one constructs space-filling SLE. When $\kappa \geq 8$, whole plane space-filling SLE is simply ordinary SLE. In this case the 
cells $\eta([x- {\eps}, x])$ are topological rectangles and their boundaries are the union of four $\text{SLE}_{\underline{\kappa}}$ curves for $\underline{\kappa} = 16/\kappa$.
When $\kappa \in (4,8)$,  $\eta$ is built by taking ordinary $\text{SLE}_{\kappa}$ and `filling in' bubbles which are disconnected from infinity.
This results in a somewhat complicated geometry. In particular, for $\kappa \in (4,8)$ it is possible for the intersection of two cells to be an uncountable, totally disconnected Cantor-like set. Pairs of cells which intersect in this manner do not correspond to edges of the mated-CRT map in Definition~\ref{def:sle-lqg-embedding}.

In this paper, we mainly just use the fact that the law of $\eta$ is scale-invariant:
\begin{equation} \label{eq:eta-scale-invariant}
r \eta \overset{d}{=} \eta, \quad \forall r > 0, \quad \mbox{when viewed as curves modulo time parameterization}.
\end{equation}
For a detailed treatment of SLE see \cite{werner-notes, lawler-book} and for a more in-depth exposition of whole plane space-filling SLE see \cite[Section 3.6]{ghs-mating-survey}.

\subsection{Mated-CRT map} \label{subsec:mated-crt-prelims}
We defined most of what we need concerning the mated-CRT map in Section \ref{subsubsec:mated-crt-map-intro}. We provide some additional notation which we will use in the sequel. 

We fix the convention that when referring to points in $\C$, we use the letters $x,y,z$; while vertices in the mated-CRT map correspond to letters $a,b$.
Sets in $\C$ are $X, Y, D$, while sets of vertices in the mated-CRT map are $A,B$. 

For ${\eps} > 0$ and $\gamma \in (0,2)$ let ${\Geps}$ be the mated-CRT map with vertex set ${\VGeps}$ embedded in the plane via a whole plane space-filling SLE curve $\eta$ with $\kappa = 16/\gamma^2$  
as in Definitions \ref{def:mated-crt-map} and \ref{def:sle-lqg-embedding}. 

It is important to note that while a.s.\ each vertex in the mated-CRT map has a finite number of neighbors, the number of neighbors may be arbitrarily large. 
Write $\deg^{{\eps}}(a)$ for the degree of $a \in {\VGeps}$. 	We will need a way to pass between vertices in the map and points in $\C$:  for $x \in \C$, let
\begin{equation} \label{eq:plane-to-embedding}
a_x^{{\eps}} := \mbox{(smallest $a \in {\eps} \Z$ such that $x \in \Heps{a}$)}.
\end{equation}
For a set $D \subset \C$, we define 
\begin{equation} \label{eq:cells-domain-restriction}
{\Geps}(D) := \mbox{(subgraph of ${\Geps}$ induced by $\{a \in {\eps} \Z: \Heps{a} \cap D \neq \emptyset\}$)}
\end{equation}
and write ${\VGeps}(D)$ for the vertices in ${\Geps}(D)$. 
Several times throughout the paper we will work with the sets ${\VGeps}(D) \setminus {\VGeps}(\partial D)$ rather than ${\VGeps}(D)$ to account 
for the coarseness of the cells in the mated-CRT map. Indeed, even if $D$ is a smooth, bounded open set, typically ${\VGeps}(D)$ overlaps with ${\VGeps}(\partial D)$ --- see Figure \ref{fig:cell-boundary-mated-crt}.	
\begin{figure}
	\includegraphics{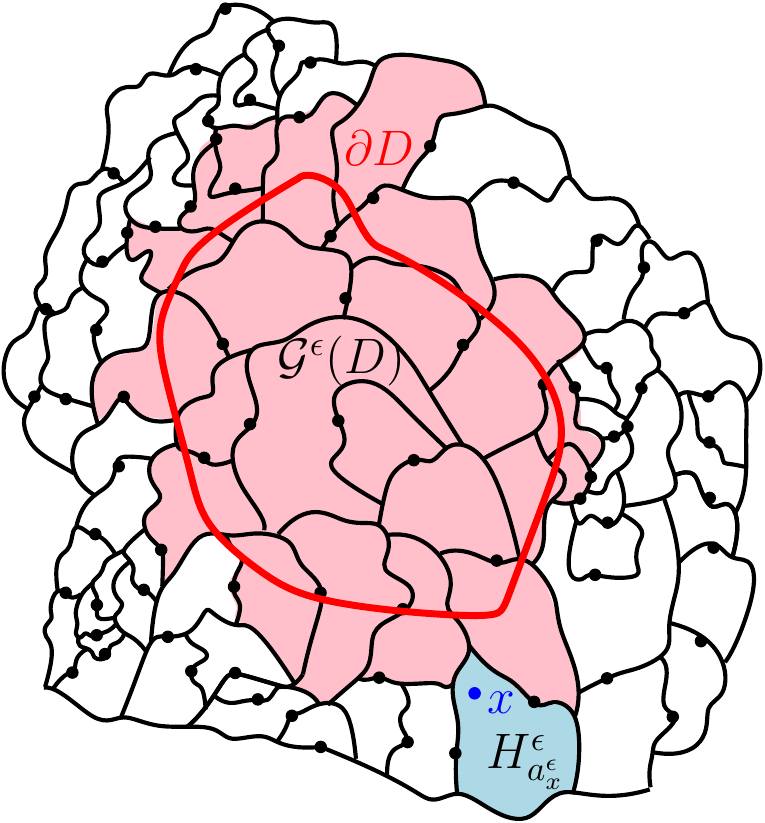}
	\caption{A set $D\subset \C$ and the set of mated-CRT map cells corresponding to the vertices in ${\Geps}(D)$, as in \eqref{eq:cells-domain-restriction}. 
		The boundary of the domain $\partial D$ is drawn in red and the cells of ${\Geps}(D)$ are filled in pink. A point $x \in \C$ and the cell containing it, $\Heps{a^{{\eps}}_x}$, is drawn 
		in light blue. A similar figure appeared in ~\cite[Figure 5]{gms-harmonic}.} \label{fig:cell-boundary-mated-crt}
\end{figure}

\subsection{Cell size and volume estimates}
In this section, we recall some estimates proved in the appendix of \cite{berestycki2020random}. 
The following estimate gives worst-case upper and lower bounds for the LQG mass of Euclidean balls.  
\begin{lemma}[Lemma A.1 in \cite{berestycki2020random}] \label{lemma:volume-growth}
	For each $\beta^+ \in (0, (2 -\gamma)^2/2)$ and $\beta^- > (2+\gamma)^2/2$
	it holds with polynomially high probability as $\delta \to 0$ that  
	\begin{equation} \label{eq:ball-volume}
	\delta^{\beta^-} \leq \mu_h(B_\delta(z)) \leq \delta^{\beta^+} ,\quad\forall z \in B_{1-\delta} .
	\end{equation}
\end{lemma}
For convenience, we fix $\beta^{\pm}$ in a fashion depending only on $\gamma$ so that when we refer to Lemma 
\ref{lemma:volume-growth} we may just refer to a single choice of $\beta^{\pm}$. 

The next lemma gives similar worst-case upper and lower bounds for the size of the cells in the mated-CRT map. 
\begin{lemma}[Lemma 2.4 in \cite{berestycki2020random}] \label{lemma:cell-size-estimate}
	Fix a small parameter $\zeta \in (0,1)$ and $\rho \in (0,1)$. With polynomially high probability as ${\eps} \to 0$, the Euclidean diameter of each mated-CRT map cell $\Heps{a}$ which intersects $B_\rho$ satisfies 
	\begin{equation}
	{\eps}^{\frac{2}{(2 - \gamma)^2} + \zeta} \leq \mathrm{diam}(\Heps{a}) \leq {\eps}^{\frac{2}{(2 + \gamma)^2} - \zeta}.
	\end{equation}
	
\end{lemma}

The following lemma is a quantitative version of the important fact that the counting measure on vertices in the mated-CRT map converges weakly to the LQG area measure.

\begin{lemma}[Lemma A.4 in \cite{berestycki2020random}] \label{lemma:mu_h-convergence}
	Fix $\rho \in (0,1)$. There exists $\alpha = \alpha(\gamma) > 0$ and $\beta = \beta(\gamma) > 0$
	such that with polynomially high probability as ${\eps} \to 0$, the following is true. 
	Let $D \subset B_{\rho}$ and let $f: B_{{\eps}^{\alpha}}(D) \to [0,\infty)$
	be a non-negative function which is ${\eps}^{-\beta}$-Lipschitz continuous and satisfies $\|f\|_{\infty} \leq {\eps}^{-\beta}$.
	For $a \in {\VGeps}(D)$, let $w^{{\eps}}_a$ be an arbitrary point of $\Heps{a} \cap \cl(D)$. 
	If we let $\mu_h$ be the $\gamma$-LQG area measure induced by $h$, then 
	\begin{equation}
	{\eps}^{-1} \int_D f(z)d \mu_h(z) - {\eps}^{-1 + \alpha} \leq
	\sum_{a \in {\VGeps}(D)} f(w^{\eps}_a) \leq 
	{\eps}^{-1} \int_{B_{{\eps}^\alpha}(D)} f(z) d \mu_h(z)  + {\eps}^{-1 + \alpha}
	\end{equation}
	simultaneously for every choice of $D$ and $f$ as above. 
\end{lemma}

An analogous result to the above holds when cells are weighted by their degrees, in which case, since the mated-CRT map is a planar triangulation,
there is an extra factor of 6 in the limit. 
\begin{lemma}[Corollary A.6 in \cite{berestycki2020random}] \label{lemma:mu_h_degree-convergence}
	For ${\eps}> 0$, let $\mu_h^{{\eps}}$ be the measure whose restriction to each cell $\Heps{a}$ for $a \in {\eps} \Z$
	is equal to $\deg^{{\eps}}(a) \mu_h |_{\Heps{a}}$. Then as ${\eps} \to 0$, we have $\mu^{{\eps}}_h \to 6 \mu_h$ 
	in probability w.r.t.\ the vague topology. 
\end{lemma}

\subsection{Liouville Brownian motion and random walk} \label{subsec:lbm-and-rw}
\textit{Liouville Brownian motion} (LBM) is the natural diffusion associated with $\gamma$-LQG and was constructed in \cite{berestycki2015diffusion, garban2016liouville}. Roughly, LBM is obtained from ordinary Brownian motion by changing time 
so that $\mu_h$ is invariant. That is, LBM started at $z\in\C$ is defined as $B^z_{\phi^{-1}(t)}$, where $B^z_t$ is Brownian motion started at $z $ and 
\[
\phi(t) = \lim_{{\eps} \to 0} \int_{0}^{t} {\eps}^{\gamma^2/2} e^{\gamma h_{{\eps}}(B_s^z)} ds, 
\]
where $h_{{\eps}}$ is the average of $h$ over the circle of radius ${\eps}$. 	 

In this paper we consider LBM rescaled by its (annealed) median exit time from a ball,
\begin{equation} \label{eq:rescaled-lbm}
X^z_t = B^z_{\phi^{-1}(t m_0)}, \quad \forall z \in \C,
\end{equation}
where 
\begin{equation} \label{eq:mo} 
m_0 = \mbox{median}(\inf\{t > 0 : |B^z_{\phi^{-1}(t)}| \geq 1/2\}).
\end{equation}	
For $z \in \C$ and ${\eps} > 0$, let $X^{z, {\eps}} : \N \to {\VGeps}$ be a simple random walk on ${\Geps}$ started from $a_z^{{\eps}}$ (defined in \eqref{eq:plane-to-embedding}).
On occasion we will also write, for $a \in {\VGeps}$, the walk started from $a$ as $X^{a, {\eps}}$.

LBM was shown in \cite{berestycki2020random} to describe the scaling limit of random walk on the mated-CRT map. To state this precisely, we introduce the scaling factor 
\begin{equation} \label{eq:meps}
{\meps} := (\mbox{median exit time of $\eta(X^{0, {\eps}})$ from $B_{1/2}$}),
\end{equation}
where $B_{1/2}$ is the Euclidean ball of radius $1/2$ centered at 0. By (1.8) in \cite{berestycki2020random},  
\begin{equation} \label{eq:meps-scaling}
C^{-1} {\eps}^{-1} \leq {\meps} \leq C {\eps}^{-1}, \quad \mbox{for all sufficiently small ${\eps} > 0$},
\end{equation}
where $C > 1$ is a deterministic constant depending only on $\gamma$. We also extend the domain of definition of the embedded 
walk $\N \ni j \to \eta(X^{z, {\eps}}_j)$ from $\N$ to $[0, \infty)$ by piecewise linear interpolation. 
\begin{theorem}[Theorem 1.2 \cite{berestycki2020random}] \label{theorem:rw-to-lbm}
	For each compact subset $K \subset \C$ and each $z \in K$,  the conditional law of the embedded linearly interpolated walk $(\eta(X^{z, {\eps}}_{m_{{\eps}} t}))_{t \geq 0}$
	given $(h, \eta)$ converges in probability to the law of rescaled $\gamma$-LBM, defined by \eqref{eq:rescaled-lbm}, started from $z$ associated with $h$ with respect to the Prokhorov topology induced by the local uniform metric on curves $[0,\infty) \to \C$.
\end{theorem}

\subsection{Discrete potential theory} \label{subsec:discrete-potential-definitions}
Recall the notation for simple random walk from Section \ref{subsec:lbm-and-rw}. 
We denote the Green's function for simple random walk killed upon exiting a set $D \subset \C$ by 
\begin{equation} \label{eq:greens-function}
\begin{aligned}
&\Gr{{\eps}}{D}(a,b) \\
&:= \E[\mbox{number of times that $X^{a, {\eps}}$ hits $b$ before exiting ${\VGeps}(D) \setminus {\VGeps}(\partial D)$} | h, \eta] 
\end{aligned}
\end{equation}
and its normalized version, the {\it Green's kernel} by
\begin{equation} \label{eq:greens-kernel}
\gr{{\eps}}{D}(a,b) := \frac{\Gr{{\eps}}{D}(a,b)}{\deg^{\eps}(b)}.
\end{equation}
By defining 
\[
\Gr{{\eps}}{D}(a, \cdot) = \Gr{{\eps}}{D}(\cdot, a) = 0, \quad \forall a \not \in {\VGeps}(D) \setminus {\VGeps}(\partial D)
\]
we extend the domain of definition of both the Green's function and kernel to all of ${\VGeps}$. 

For $a \in {\VGeps}$, the exit time of simple random walk on the mated-CRT map from a domain $D \subset \C$ is defined by
\begin{equation}  \label{eq:exit-time}
\taueps{D}{a} := \mbox{(first exit time of $X^{a, {\eps}}$ from ${\VGeps}(D) \setminus {\VGeps}(\partial D)$)}  
\end{equation}
and its mean is
\begin{equation} \label{eq:expected-exit-time}
\Qeps{D}(a):= \E[\taueps{D}{a} | h, \eta]  = \sum_{b \in {\VGeps}(D) \setminus {\VGeps}(\partial D)} \Gr{{\eps}}{D}(a,b) .
\end{equation}
We will also need to consider the {\it normalized expected exit time}, for  $a \in {\VGeps}$
\begin{equation} \label{eq:normalized-exit-time}
\qeps{D}(a)= \sum_{b \in {\VGeps}(D) \setminus {\VGeps}(\partial D)} \gr{{\eps}}{D}(a,b).
\end{equation}
Let $\Delta^{{\eps}}$ denote the degree normalized graph Laplacian on $\mathcal{G}^{{\eps}}$, 
\begin{equation} \label{eq:degree-normalized-laplacian}
\begin{aligned}
\Delta^{{\eps}} u(a) &= \frac{1}{\deg^{{\eps}}(a)} \sum_{b \sim a} (u(b) - u(a)) \quad \forall a \in {\VGeps}, \\
&\mbox{for functions $u: {\VGeps} \to \R$}, 
\end{aligned}
\end{equation}
where the sum $b \sim a$ is over the sites $b \in {\VGeps}$ which are joined by an edge to $a$. 
As $\Delta^{{\eps}}$ is the generator of simple random walk on ${\Geps}$, for $D \subset \C$, by, \eg, \cite[Proposition 6.2.3]{lawler-limic-walks},
\begin{equation} \label{eq:discrete-exit-laplacian}
\begin{cases}
\Delta^{{\eps}} \Qeps{D}(\cdot) = -1 \quad &\mbox{in ${\VGeps}(D) \setminus {\VGeps}(\partial D)$} \\
\Qeps{D}(\cdot) = 0 \quad &\mbox{otherwise}
\end{cases} 
\end{equation}
\begin{equation} \label{eq:discrete-normalized-exit-laplacian}
\begin{cases}
\Delta^{{\eps}} \qeps{D}(\cdot) = -1/\deg^{{\eps}}(\cdot) \quad &\mbox{in ${\VGeps}(D) \setminus {\VGeps}(\partial D)$} \\
\qeps{D}(\cdot) = 0 \quad &\mbox{otherwise}
\end{cases} 
\end{equation}
and for each $a \in {\VGeps}(D) \setminus {\VGeps}(\partial D)$, 
\begin{equation} \label{eq:discrete-green-laplacian}
\begin{cases}
\Delta^{{\eps}} \Gr{{\eps}}{D}(\cdot,a) = -1\{\cdot = a\} \quad &\mbox{in ${\VGeps}(D) \setminus {\VGeps}(\partial D)$} \\
\Gr{{\eps}}{D}(\cdot,a) = 0 \quad &\mbox{otherwise}
\end{cases} 
\end{equation}
and
\begin{equation} \label{eq:discrete-normalized-green-laplacian}
\begin{cases}
\Delta^{{\eps}} \gr{{\eps}}{D}(a,\cdot) = -\frac{1\{\cdot = a\}}{\deg^{{\eps}}(a)} \quad &\mbox{in ${\VGeps}(D) \setminus {\VGeps}(\partial D)$} \\
\gr{{\eps}}{D}(a,\cdot) = 0 \quad &\mbox{otherwise}.
\end{cases} 
\end{equation}
\begin{remark} \label{remark:overload-notation-potential}
	For a subset $A$ of  ${\VGeps}$, we overload notation and 
	define $\Gr{{\eps}}{A}, 	\gr{{\eps}}{A}, \Qeps{D},$ and $\qeps{A}$
	as above with the instances of ${\VGeps}(D) \setminus {\VGeps}(\partial D)$
	replaced by $A$.  Each of \eqref{eq:discrete-exit-laplacian}, \eqref{eq:discrete-normalized-exit-laplacian}, \eqref{eq:discrete-green-laplacian}, 
	and \eqref{eq:discrete-normalized-green-laplacian} hold with this same substitution. 
\end{remark}

We also write $\Delta$ for the continuum Laplacian on $\C$. The operator $\Delta^{{\eps}}$ shares several basic properties with its continuum counterpart $\Delta$, \eg, linearity and monotonicity. Of importance to us is the following {\it maximum principle} for $\Delta^{{\eps}}$. Recall that a function $u: {\VGeps} \to \R$ is {\it subharmonic} (resp. {\it superharmonic}) on $A \subset {\VGeps}$ if $\Delta^{{\eps}} u \geq 0$
(resp. $\Delta^{{\eps}} u \leq 0$) on $A$ and is {\it harmonic} if it is both subharmonic and superharmonic.  When we need to emphasize the distinction between $\C$ and ${\VGeps}$, we qualify functions as $\Delta$-harmonic or $\Delta^{{\eps}}$-harmonic.

\begin{lemma} \label{lemma:maximum-principle}
	For every connected subset $A \subset {\VGeps}$, if $u: \cl(A) \to \R$ is subharmonic on $A$, then 
	\begin{equation} \label{eq:maximum-principle}
	\max_{a \in A} u(a) \leq \max_{b \in \partial A} u(b). 
	\end{equation}
\end{lemma} 
\begin{proof}
	This is standard. Consider a point $a \in A$ and apply the definition of subharmonicity to find a
	neighbor $\cl(A) \ni b \sim a$ for which $u(b) \geq u(a)$. 
\end{proof}

In this article, we will use discrete Green's function and exit times estimates established in \cite{berestycki2020random}. 
Fix a $\rho \in (0,1)$ --- several of our estimates will be stated only on $B_\rho$. The reason for working on $B_\rho$ is that $h|_{B_1}$ agrees in law with the corresponding restriction of a whole-plane GFF plus $\gamma\log(1/|\cdot|)$, see \eqref{eq:gff}. Consequently the constants in several of these estimates will depend on $\rho$. 
First, we recall an upper bound on the Green's kernel. 
\begin{lemma}[Lemma 5.4 in \cite{berestycki2020random}] \label{lemma:bounded-greens-function}
	There exists $\beta = \beta(\gamma) > 0$
	and $C = C(\rho, \gamma) > 0$ such that 
	with polynomially high probability as ${\eps} \to 0$,
	\[
	\gr{{\eps}}{B_{\rho}}(a,b) 
	\leq C \log \left(  \frac{1}{|\eta(a) - \eta(b)|} \right) + C
	\]
	simultaneously for all $a,b \in {\VGeps}(B_{\rho})$
	with $|\eta(a) - \eta(b)| \geq {\eps}^{\beta}$.
\end{lemma}
We will also need an upper bound on the Green's function. Recall that the Green's function is the Green's kernel multiplied by the degree. 
Since the vertex degrees on the mated-CRT map are unbounded, we cannot directly apply the previous lemma. 
Nevertheless, the following crude upper bound will suffice. 
\begin{lemma}[Lemma 5.6 in \cite{berestycki2020random}] \label{lemma:greens-function-log-bound}
	For each $\zeta \in (0,1)$, it holds 
	with polynomially high probability as ${\eps} \to 0$ that
	\[
	\sup_{a,b \in {\VGeps}(B_{\rho})} \Gr{{\eps}}{B_{\rho}}(a,b) 
	\leq {\eps}^{-\zeta}.
	\]
\end{lemma}

We also need a lower bound on both the Green's kernel and function. 
\begin{lemma}[Lemma 4.4 in \cite{berestycki2020random}] \label{lemma:greens-kernel-lower-bound}
	There exists $\omega = \omega(\gamma) > 0$ and $C = C(\rho, \gamma) > 0$
	such that with polynomially high probability as ${\eps} \to 0$ the following holds. 
	For each $a \in {\VGeps}(B_{\rho - {\eps}^{\omega}})$ and 
	each $r \in [{\eps}^{\omega}, \dist(\eta(a), \partial B_{\rho})]$, 
	\[
	\begin{aligned}
	\Gr{{\eps}}{B_r(\eta(a))}(a,b)  \geq \gr{{\eps}}{B_r(\eta(a))}(a,b) &\geq C^{-1} \log \left( \frac{r}{|\eta(a) - \eta(b)|} \wedge {\eps}^{-1} \right) - 1, \\
	\quad \forall b \in {\VGeps}(B_{r/3}(\eta(a))).
	\end{aligned}
	\]

\end{lemma}

We conclude this subsection by formulating an upper bound on the expected exit time function in terms of the coarseness
of the ${\eps}$-mated CRT map and a certain continuum quantity. To that end, we consider the following random variable (from (5.1) in \cite{berestycki2020random})
\begin{equation} \label{eq:largest-exit-time-rv}
\Mrho{D} := \sup_{z \in D} \int_{D} \left( \log \left( \frac{1}{|z-w|} \right) + 1 \right) d\mu_h(w), \quad \mbox{for $D \subset \C$}
\end{equation}
and recall the following bounds. 
\begin{lemma} \label{lemma:bounded-exit-time}
	Almost surely, 
	\[
	\sup_{D \subseteq B_{\rho}} \Mrho{D} < \infty
	\]
	and with polynomially high probability as $\delta \to 0$
	\[	
	\sup_{y \in B_{\rho}} \Mrho{B_{\delta}(y)} \leq \delta^q, 
	\]
	for some deterministic $q = q(\gamma) > 0$.
\end{lemma}

\begin{proof}
	{\it Step 1: $	\mathcal{M}(D) < \infty$.} \\
	As $\mu_h$ is a locally finite measure, it suffices to bound the log integrand in \eqref{eq:largest-exit-time-rv}.
	By \cite[Lemma A.3]{berestycki2020random}, there are constants $c_0, c_1$ depending only on $\rho$ and $\gamma$ such that for each $A > 1$,
	it holds with probability at least $1 - c_0 e^{-c_1 A}$ that 
	\begin{equation} \label{eq:intermediate-bound}
	\sup_{z \in D} \int_{D}  \log \left( \frac{1}{|z-w|} \right) d\mu_h(w) \leq A ( \mu_h(D) + e^{-c_1 A}). 
	\end{equation}
	By the Borel-Cantelli lemma, this shows $\sup_{D \subseteq B_{\rho}} \Mrho{D} < \infty$ almost surely. 
	
	{\it Step 2: $\sup_{y \in B_{\rho - \delta}} \Mrho{B_{\delta}(y)} \leq \delta^q$.} \\
	This follows by a partition of the integral into dyadic annuli as in the proof of \cite[Lemma A.3]{berestycki2020random}. An explicit proof is given in  
	the proof of Corollary 5.2 in \cite{berestycki2020random} -- this bound is exactly (5.3) there. 
\end{proof}

By \cite{berestycki2020random}, we can bound the expected exit time in terms of this random variable. 
\begin{lemma}[Proposition 5.1 in \cite{berestycki2020random}] \label{lemma:exit-time-upper-bound}
	There exists deterministic $\alpha = \alpha(\gamma) > 0$ and $C = C(\rho, \gamma) > 0$ such that with polynomially high probability as ${\eps} \to 0$
	it holds simultaneously for every Borel set $D \subset B_{\rho}$ with $\mu_h(D) \geq {\eps}^{\alpha}$
	that for all $a \in {\VGeps}(D)$
	\[
	\E[(\taueps{D}{a})^N| h, \eta]  \leq N! C^N {\eps}^{-N} \left( \sup_{z \in D} \int_{B_{{\eps}^{\alpha}}(D)}  \left ( \log \left( \frac{1}{|z-w|} \right) + 1  \right) d \mu_h(w) \right)^N
	\quad \forall N \in \N,
	\]
	in particular, for all $D \subset B_{\rho}$
	\[
	\qeps{D}(a) \leq \Qeps{D}(a)\leq C {\eps}^{-1} \Mrho{B_{{\eps}^{\alpha}}(D)}, \quad \forall a \in {\VGeps}(D).
	\]
\end{lemma}

By combining the previous two lemmas, we have the following bound on the expected exit time. 
\begin{lemma} \label{lemma:expected-exit-time-explicit-bound}
	There exists a deterministic $\alpha = \alpha(\gamma) > 0$ such that with polynomially high probability as ${\eps} \to 0$
	it holds simultaneously for every Borel set $D \subset B_{\rho}$ with $\mu_h(D) \geq {\eps}^{\alpha}$
	\[
	\qeps{D} \leq \Qeps{D} < \infty.
	\]
	Moreover,	for some deterministic $q = q(\gamma) > 0$, it holds with probability tending to 1 as ${\eps} \to 0$ and then $\delta \to 0$ that
	\[	
	\qeps{B_{\delta}(y)} \leq \Qeps{B_{\delta}(y)} \leq {\eps}^{-1} \times \delta^q,
	\]
	for every $y \in \C$ such that $B_{\delta}(y) \subset B_{\rho}$. 
\end{lemma}
\begin{proof}
	Apply Lemma \ref{lemma:bounded-exit-time} with some fixed $\rho' \in (\rho, 1)$ to Lemma \ref{lemma:exit-time-upper-bound} and absorb the constants into the factor of $\delta^q$. 
\end{proof}

	\section{Convergence of the Dirichlet problem} \label{sec:dirichlet-problem}
	In this section we use the estimates recalled in the previous section to prove uniform convergence of both the Green's kernel and the solution to the Dirichlet problem to their continuum counterparts.  Throughout this section, we fix $\rho \in (0,1)$ such that $B_{\rho} \Subset B_1$.

	\subsection{Scaling}
	Recall the definition of ${\meps}$  and recall from~\eqref{eq:meps-scaling} that ${\meps} \asymp {\eps}^{-1}$.
	For a set $D \subset B_{\rho}$, we rescale a function, $u^{{\eps}}: {\VGeps}(D) \to \R $  
	by defining $\overline{u}^{{\eps}}: D \to \R$ in a piecewise constant fashion
	\begin{equation} \label{eq:discrete-scaling}
	\overline{u}^{{\eps}}(x) := {\meps}^{-1}  u^{{\eps}}(a^{{\eps}}_x), \quad \forall x \in D, 
	\end{equation}
	where $a^{{\eps}}_x$ is as in \eqref{eq:plane-to-embedding}.  That is, $\overline{u}^{{\eps}}$ is constant within the cells of the mated-CRT map.

	Throughout this section, we work with functions which are of order ${\eps}^{-1}$. We do this because this is the order of the expected exit time function,
	Lemma \ref{lemma:exit-time-upper-bound}, and the divisible sandpile odometer in Section \ref{sec:divisible-sandpile}. 
	
	\subsection{Tightness}
	In this subsection we prove a general tightness result concerning bounded functions with bounded Laplacian.

	\begin{lemma} \label{lemma:discrete-potential-tightness}
		For every connected set $D \subset B_{\rho}$
		and $\sigma(h,\eta)$-measurable sequence  $\mathcal{E} \ni {\eps} \to 0$,
		if a sequence of functions $u^{{\eps}}: {\VGeps}(D) \to \R $ indexed by ${\eps} \in \mathcal E$ satisfies
		\[
		\begin{aligned}
		|\Delta^{{\eps}} u^{{\eps}}| \leq C &\quad \mbox{on ${\VGeps}(D) \setminus {\VGeps}(\partial D)$} \\
		u^{{\eps}} \leq C {\eps}^{-1} &\quad \mbox{on ${\VGeps}(D)$} 
		\end{aligned}
		\]
		uniformly over ${\eps} \in \mathcal{E}$ for some $\sigma(h,\eta)$-measurable $C > 0$, then the following is true. 
		Almost surely, there exists a $\sigma(h,\eta)$-measurable subsequence $\mathcal{E}_0 \subset \mathcal{E}$ and a function $\overline{u} \in C(D)$ such that for every compact subset $O \Subset D$, 
		we have that  $\overline{u}^{{\eps}} \to \overline{u}$ uniformly in $O$ as $\mathcal{E}_0 \ni {\eps} \to 0$. 
		
		Moreover, if $D$ is simply connected, $u^{{\eps}}$ is discrete harmonic in ${\VGeps}(D) \setminus {\VGeps}(\partial D)$, and if $u^{{\eps}}$ has H\"{o}lder continuous boundary data, in the sense that there exists $\chi \in (0,1]$ and
		$C' > 0$ so that 
		\begin{equation} \label{eq:holder-boundary-data}
		|u^{{\eps}}(a) - u^{{\eps}}(b)| \leq C'( {\eps} \vee |\eta(a) - \eta(b)|)^{\chi}, \quad \forall a,b \in {\VGeps}(\partial D), \quad \forall {\eps}  \in (0,1)
		\end{equation}
		then the convergence occurs uniformly in $\cl(D)$. 
	\end{lemma}
	
	Lemma \ref{lemma:discrete-potential-tightness} will follow from the Arz\'ela-Ascoli theorem and the following
	H\"{o}lder continuity of $\Delta^{{\eps}}$-harmonic functions on the mated-CRT map. 
	\begin{lemma}[Theorem 3.9 in \cite{gwynne2019harmonic}] \label{lemma:discrete-holder-continuity}
		There exists $\xi = \xi(\rho, \gamma) \in (0,1)$ and $A = A(\rho, \gamma) > 0$ 
		such that the following holds with polynomially high probability as ${\eps} \to 0$. 
		Let $D \subset B_{\rho}$ be a connected domain and let $f^{{\eps}}: {\VGeps}(D) \to \R$ be 
		$\Delta^{{\eps}}$-harmonic on ${\VGeps}(D) \setminus {\VGeps}(\partial D)$. Then, 
		\begin{equation} 
		|f^{{\eps}}(a) - f^{{\eps}}(b)| \leq A \left( \sup_{a' \in {\VGeps}(D)} |f^{{\eps}}(a')| \right) \left( \frac{{\eps} \vee |\eta(a) - \eta(b)|}{\dist(\eta(a), \partial D)} \right)^\xi, \quad \forall a,b \in {\VGeps}(D).
		\end{equation} 
		Moreover, if $D$ is simply connected and if $f^{{\eps}}$ has H\"{o}lder continuous boundary data in the sense of \eqref{eq:holder-boundary-data}, then 
		there exists a $C > 0$ such that
		\begin{equation} \label{eq:boundary-estimate}
		|f^{{\eps}}(a) - f^{{\eps}}(b)| \leq \max \left\{ C, A \left( \sup_{a' \in {\VGeps}(D)} |f^{{\eps}}(a')| \right) \right \}( {\eps} \vee |\eta(a) - \eta(b)|)^{\xi}, \quad \forall a,b \in {\VGeps}(D).
		\end{equation}
	\end{lemma}

	\begin{proof}[Proof of Lemma \ref{lemma:discrete-potential-tightness}] 
		The conclusions of Lemmas \ref{lemma:cell-size-estimate}, \ref{lemma:expected-exit-time-explicit-bound}, and \ref{lemma:discrete-holder-continuity}  
		hold with probability approaching one as $\mathcal{E} \ni {\eps} \to 0$. In particular, by the Borel-Cantelli lemma, almost surely, there exists a deterministic subsequence $\mathcal{E}_1 \subset \mathcal{E}$ such that the conclusions of all of these lemmas hold for all sufficiently small ${\eps} \in \mathcal{E}_1$. We henceforth restrict to such an event. 
		
		We first prove interior convergence: let $O \Subset D \subset B_{\rho}$ be given. By the Arz\'ela-Ascoli theorem, and the boundedness assumption on $\overline{u}^{{\eps}}$, it suffices to prove the following estimate:
		almost surely, for each small $k > 0$, there exists possibly random choices of $\delta = \delta(k) > 0$ and ${\eps}_1 = {\eps}_1(k, \delta) \in \mathcal{E}_1$ so that for all ${\eps} \in (0, {\eps}_1) \cap \mathcal{E}_1$, 
		\begin{equation} \label{eq:equicontinuity}
		|x-y| \leq \delta \implies |\overline{u}^{{\eps}}(x) - \overline{u}^{{\eps}}(y)| \leq k, \quad \forall x,y \in O.
		\end{equation}

		We now prove \eqref{eq:equicontinuity}. Let $k > 0$ be given, fix $\delta > 0$ to be determined below, in a manner depending on $k$. By Lemma \ref{lemma:cell-size-estimate}, 	
		\begin{equation} \label{eq:cell-size-estimate}
		\sup_{a \in {\VGeps}(B_{\rho})} \diam(\Heps{a}) \leq \delta^2, \quad \forall {\eps} \in \mathcal{E}_1 \cap (0, {\eps}_1)
		\end{equation}
		where ${\eps}_1$ depends on $\delta$ through Lemma \ref{lemma:cell-size-estimate}.  
		
		As $O \Subset D$, there exists $\delta_1 > 0$ so that $O + B_{\delta_1} \Subset D$.	
		Suppose $\delta \in (0,  \min(\delta_1^{2}, \delta_1)/100 )$ and let $x \in O$ be given.  Let $\underline{\delta} = \delta^{1/2} + \delta$ and
		$f^{{\eps}}:{\VGeps}(\cl(B_{\underline{\delta}}(x))) \to \R$ be the harmonic extension of $u^{{\eps}}$,
		\[
		\begin{cases}
		\Delta^{{\eps}} f^{{\eps}} = 0 \quad \mbox{on ${\VGeps}(B_{\underline{\delta}}(x)) \setminus {\VGeps}(\partial B_{\underline{\delta}}(x))$} \\
		f^{{\eps}} = u^{{\eps}} \quad \mbox{on ${\VGeps}(\partial B_{\underline{\delta}}(x))$} .
		\end{cases}
		\]
		By the maximum principle (Lemma \ref{lemma:maximum-principle}) and our uniform boundedness assumption on $u^{{\eps}}$, 
		\begin{equation} \label{eq:harmonic-extension-bounded}
		\sup_{a' \in {\VGeps}(\cl(B_{\delta}(x)))} |f^{{\eps}}(a')| \leq \sup_{a' \in {\VGeps}(\cl(B_{\delta}(x)))} |u^{{\eps}}(a')| \leq  C {\eps}^{-1}  , \quad \forall {\eps} \in (0, {\eps}_1) \cap \mathcal{E}.
		\end{equation}

		By the assumption $|\Delta^{{\eps}} u^{{\eps}}| \leq C$ and \eqref{eq:discrete-exit-laplacian}, the functions $(u^{{\eps}} - f^{{\eps}}) \pm C \times \Qeps{B_{\underline{\delta}}(x)}$ are super (resp. sub) harmonic on ${\VGeps}(B_{\underline{\delta}}(x)) \setminus {\VGeps}(\partial B_{\underline{\delta}}(x))$ and identically equal to 0 on ${\VGeps}(\partial B_{\underline{\delta}}(x))$ for each ${\eps}$. Therefore, by the maximum principle,  
		\[
		|u^{{\eps}}-f^{{\eps}}| \leq  C \times \Qeps{B_{\underline{\delta}}(x)}  \quad \mbox{on ${\VGeps}(\cl(B_{\underline{\delta}}(x)))$}
		\]
		for all ${\eps} \in (0, {\eps}_1) \cap \mathcal{E}_1$.
		Thus,  we have, for such ${\eps}$, by the triangle inequality, 
		\begin{equation} \label{eq:harmonic-decomposition}
		\sup_{a,b \in {\VGeps}(B_{\delta}(x))} |u^{{\eps}}(a) - u^{{\eps}}(b)|  \leq \sup_{a,b \in {\VGeps}(B_{\delta}(x))} |f^{{\eps}}(a) - f^{{\eps}}(b)| + 2 C \times \sup_{a\in  {\VGeps}(B_{\delta}(x))} \Qeps{B_{\underline{\delta}}(x)}(a) .
		\end{equation}
		By Lemma  \ref{lemma:expected-exit-time-explicit-bound}, we may take ${\eps}_1$ smaller, depending only on $\delta$, so that
		\begin{equation} \label{eq:second-term-bound}
		2 C \times \sup_{a\in  {\VGeps}(B_{\delta}(x))} \Qeps{B_{\underline{\delta}}(x)}(a)   \leq {\eps}^{-1} \delta^q, \quad \forall {\eps} \in (0, {\eps}_1) \cap \mathcal{E}_1
		\end{equation}
		for a deterministic $q  = q(\gamma) > 0$. 
		
		It remains to bound the first term in \eqref{eq:harmonic-decomposition}: with $A = A(\rho, \gamma) > 0$ and $\xi \in (0,1)$ as in Lemma \ref{lemma:discrete-holder-continuity},
		\begin{align*}
		&\sup_{a,b \in {\VGeps}(B_{\delta}(x))} |f^{{\eps}}(a) - f^{{\eps}}(b)| \\
		&\leq  A \left( \sup_{a' \in {\VGeps}(B_{\delta}(x))} |f^{{\eps}}(a')| \right) \sup_{a,b \in {\VGeps}(B_{\delta}(x))} \left( \frac{{\eps} \vee |\eta(a) - \eta(b)|}{\dist(\eta(a), \partial B_{\underline{\delta}}(x))} \right)^\xi \quad \mbox{(Lemma \ref{lemma:discrete-holder-continuity})} \\
		&\leq A C {\eps}^{-1} \sup_{a,b \in {\VGeps}(B_{\delta}(x))} \left( \frac{{\eps} \vee |\eta(a) - \eta(b)|}{\dist(\eta(a), \partial B_{\underline{\delta}}(x))} \right)^\xi \quad \mbox{(by \eqref{eq:harmonic-extension-bounded})}  \\
		&\leq A C {\eps}^{-1} \left( \frac{\delta + \delta^2}{\delta^{1/2} - \delta^2} \right)^\xi \quad \mbox{(by \eqref{eq:cell-size-estimate} and since $\underline\delta = \delta^{1/2} + \delta$)}  \\
		&\leq {\eps}^{-1} k/2 \quad \mbox{(for small $\delta = \delta(k, \xi)$)},  \stepcounter{equation}\tag{\theequation}\label{eq:first-term-bound}
		\end{align*}
		for all ${\eps} \in (0, {\eps}_1) \cap \mathcal{E}_1$. 
		
		By plugging \eqref{eq:second-term-bound} and \eqref{eq:first-term-bound} into~\eqref{eq:harmonic-decomposition}, we get that for sufficiently small $\delta = \delta(k, \xi)$,
		\[
		\sup_{y \in {\VGeps}(B_{\delta}(x))} |u^{{\eps}}(y) - u^{{\eps}}(x)|  \leq {\eps}^{-1} k, \quad \forall {\eps} \in (0, {\eps}_1) \cap \mathcal{E}_1.
		\]
		By combining this inequality with \eqref{eq:cell-size-estimate} and the definition~\eqref{eq:discrete-scaling} of $\overline{u}^{{\eps}}$, we have \eqref{eq:equicontinuity}.

		We now assume the conditions of the moreover clause. Since we have assumed $u^{{\eps}}$ is harmonic, we can apply \eqref{eq:boundary-estimate} of Lemma \ref{lemma:discrete-holder-continuity} to get the analogue of \eqref{eq:equicontinuity} up to the boundary: for each $k > 0$, there exists $\delta = \delta(k,C',\chi) > 0$ such that for each small enough ${\eps} \in \mathcal E_1$,
		\begin{equation} \label{eq:up-to-boundaryequicontinuity}
		|x-y| \leq \delta \implies |\overline{u}^{{\eps}}(x) - \overline{u}^{{\eps}}(y)| \leq k, \quad \forall x,y \in {\VGeps}(D).
		\end{equation} 
		And so, we may conclude using the Arz\'ela-Ascoli theorem. 
	\end{proof}

	\subsection{Uniform convergence of the Green's kernel} \label{subsec:green-kernel}
	For $z \in \C$, let 
	\[
	X^{z, {\eps}} : \N \to {\VGeps}
	\]
	be a simple random walk on ${\Geps}$ started from the vertex $a^{{\eps}}_z \in {\VGeps}$, defined in \eqref{eq:plane-to-embedding}, 
	stopped upon exiting ${\VGeps}(B_{\rho}) \setminus {\VGeps}(\partial B_{\rho})$. 
	Let $X^{z}$ be LBM started at $z$ stopped upon exiting $B_{\rho}$
	and  let
	\[
	\hat{X}^{z, {\eps}}: [0,\infty) \to \C
	\]
	be the process 
	$t \to \eta(X^{z, {\eps}}_{{\meps} t})$ extended from ${\meps}^{-1} \N$ to $[0,\infty)$ by piece-wise linear interpolation. 
	The goal of this section is to prove the following. 
	\begin{lemma} \label{lemma:uniform-convergence-of-greens-function}
		Fix a deterministic choice of $r>0$ and $z\in\C$ such that $B_r(z) \subset B_\rho$ and 
		$y \in B_r(z)$. Let $g^{{\eps}}(\cdot) := {\eps}^{-1} \gr{{\eps}}{B_r(z)}(a^{{\eps}}_y , \cdot)$ and define $\overline{g}^{\eps}$ as in~\eqref{eq:discrete-scaling}.
		For each small $\delta > 0$ and $O \subset \cl(B_r(z)) \setminus \{y\}$ which lies at positive distance from $y$, 
		with probability tending to 1 as ${\eps} \to 0$, 
		\begin{equation} \label{eq:green-uniform}
		\sup_{x \in O} \left|\overline{g}^{{\eps}}(x)  - \frac{m_0}{6} \times G_{B_r(z)}(y , x)\right| \leq \delta,
		\end{equation}
		where $G_{B_r(z)}(y,x)$ is the continuum Green's function for standard Brownian motion killed upon exiting $B_r(z)$.
	\end{lemma}
	
	Lemma \ref{lemma:uniform-convergence-of-greens-function} will be shown as a consequence of the following result and Lemma \ref{lemma:discrete-potential-tightness}.
	
	\begin{lemma} \label{lemma:integral-convergence-of-greens-function}
		For each deterministic choice of $z\in\C$ and $r > 0$ such that $B_r(z) \subset B_{\rho}$, each $y \in B_r(z)$, and each bounded, H\"{o}lder-continuous, $\sigma(h, \eta)$-measurable function $f: B_{r}(z) \to \R$, 
		\[
		\begin{aligned}
		{\meps}^{-1}  \sum_{x \in {\VGeps}(B_r(z))} \Gr{{\eps}}{B_r(z)}(a^{{\eps}}_{y}, x) f(\eta(x)) \to m_0 \times \int_{B_r(z)} G_{B_r(z)} (y, x) f(x) d \mu_h(x), 
		\end{aligned}
		\]
		in probability as ${\eps} \to 0$, where $G_{B_r(z)}$ is the continuum Green's function, and $a^{{\eps}}_{y}$ is as in \eqref{eq:plane-to-embedding}.  
	\end{lemma}
	\begin{proof}
		This follows from convergence of random walk on ${\Geps}$ to LBM and the fact that the discrete and continuum Green's functions can be represented in terms of these objects. 
		
		The conclusions of Lemmas \ref{lemma:expected-exit-time-explicit-bound} and \ref{lemma:cell-size-estimate}
		hold with probability approaching one as ${\eps} \to \infty$. Restrict to the event that these occur for all small enough ${\eps}$. 
		
		Let the deterministic parameters $y \in B_r(z) \subset B_{\rho}$ and $f: B_r(z) \to \C$ a bounded and H\"{o}lder continuous function which is $\sigma(h, \eta)$-measurable be given. Consider the bounded, continuous 
		functional on the space of continuous curves $w:   [0,\infty)  \to B_{\rho}$, 
		\[
		\phi^{f}_T(w) := \int_0^T  f(w_t)  dt, \quad \forall T > 0 .
		\]   
		By boundedness and 
		Theorem \ref{theorem:rw-to-lbm}, 
		\begin{equation} \label{eq:rw-to-lbm}
		\lim_{{\eps} \to 0} \E[ \phi^{f}_T(\hat{X}^{y, {\eps}}) | h , \eta] = \E\left[ \int_0^T f(X^y_t) dt | h, \eta\right], \quad \forall T > 0 ,
		\end{equation}
		where the convergence is in probability.
		
		The term on the left in \eqref{eq:rw-to-lbm} can be used to control its discrete counterpart.
		Indeed, as $f$ is H\"{o}lder continuous and bounded and the cell size of the mated-CRT map is controlled by Lemma \ref{lemma:cell-size-estimate}, for all ${\eps}$ small depending only on $h, \eta$, 
		\[
		{\meps}^{-1} \sum_{t=0}^{\lceil T {\meps} \rceil} f(\eta(X^{y, {\eps}}_t))  - C {\eps}^{q} \leq \phi^{f}_T(\hat{X}^{y, {\eps}}) \leq {\meps}^{-1} \sum_{t=0}^{\lceil T {\meps} \rceil} f(\eta(X^{y, {\eps}}_t))  + C {\eps}^{q}
		\]
		for all $T > 0$, where $C = C(f) > 0$, $q = q(\gamma, f) > 0$, and $a^{{\eps}}_y$ is as in \eqref{eq:plane-to-embedding}.  Thus, by \eqref{eq:rw-to-lbm},
		\begin{equation} \label{eq:green-trunc-weak-convergence}
		\lim_{{\eps} \to 0} {\meps}^{-1}\E\left[ \sum_{t=0}^{\lceil T {\meps} \rceil } f(\eta(X^{y, {\eps}}_t))  | h , \eta\right] =  \E\left[ \int_0^T f(X^y_t) dt | h, \eta\right], \quad \forall T > 0,
		\end{equation}
		in probability as ${\eps} \to 0$. By Lemmas \ref{lemma:bounded-exit-time}, \ref{lemma:exit-time-upper-bound}  and the boundedness of $f$, with $\taueps{B_r(z)}{a^{{\eps}}_y}$ the first exit time as in~\eqref{eq:exit-time},
		\[
		|{\meps}^{-1}\ \sum_{t=0}^{\lceil T {\meps} \rceil } f(\eta(X^{y, {\eps}}_t)) |
		< {\meps}^{-1} \sup_{x \in B_r(z)} |f(x)| \taueps{B_r(z)}{a^{{\eps}}_y} 
		< \infty, \quad \forall T  >0, \quad \forall {\eps} \in (0,1),
		\]
		and ${\meps}^{-1} \E[\taueps{B_r(z)}{a^{{\eps}}_y}  | h, \eta]$ 
		is a.s.\ bounded uniformly in ${\eps}$. 
		Therefore, as \eqref{eq:green-trunc-weak-convergence} holds for each $T > 0$, we may take $T \to \infty$ and interchange the limits to see that
		\begin{align*}
		{\meps}^{-1} \sum_{b \in {\VGeps}(B_r(z))} \Gr{{\eps}}{B_r(z)}(a^{{\eps}}_y, b) f(\eta(b))  
		&= {\meps}^{-1}\E\left[ \sum_{t=0}^{\infty } f(\eta(X^{y, {\eps}}_t))  | h , \eta\right]  \\
		&\overset{{\eps} \to 0}{\rightarrow}  \E\left[ \int_0^\infty f(X^y_t) dt | h, \eta\right] \\
		&= m_0 \times \int_{B_r(z)} G_{B_r(z)} (y, x) f(x) d \mu_h(x),
		\end{align*}
		completing the proof. 	
	\end{proof}
	
	We now use the integral convergence of the Green's function (Lemma~\ref{lemma:integral-convergence-of-greens-function}) together with Lemma \ref{lemma:mu_h_degree-convergence} and tightness to show uniform convergence of the Green's kernel away from its pole. 
	
	\begin{proof}[Proof of Lemma \ref{lemma:uniform-convergence-of-greens-function}]
		Let a sequence of ${\eps}$s tending to zero be given. Fix $y \in B_r(z) \subset B_{\rho}$ and let $O \Subset B_{r}(z) \setminus \{y\}$ be open.
		Denote by  $\mathcal{P}(B_r(z))$ the set of polynomials with rational coefficients in $B_r(z)$. 
		
		There exists a deterministic subsequence $\mathcal{E}$ along which a.s.\ the convergence in Lemma  \ref{lemma:mu_h_degree-convergence}
		occurs and the convergence in Lemma \ref{lemma:integral-convergence-of-greens-function}  
		occurs for each $f \in \mathcal{P}(B_r(z))$.
		We may further assume, by the Borel-Cantelli lemma, that along this subsequence the conclusions of Lemma \ref{lemma:bounded-greens-function}
		and Lemma \ref{lemma:cell-size-estimate} hold.  Restrict to the subsequence $\mathcal{E}$ and the event of the preceding two sentences.

		By Lemma \ref{lemma:bounded-greens-function},  $\gr{{\eps}}{B_{\rho}}(a^{{\eps}}_y , \cdot)$ is uniformly bounded in ${\VGeps}(O)$ by a random constant along $\mathcal{E}$, and by \eqref{eq:discrete-normalized-green-laplacian} is $\Delta^{{\eps}}$-harmonic in ${\VGeps}(O)$. Thus, by Lemma \ref{lemma:discrete-potential-tightness}, there is a random subsequence along which $\overline g^{\eps}$ converges uniformly in $O$ to some continuous function $\overline g^* : O \to \R$ (possibly depending on the subsequence). 
		
		Fix such a subsequence $\mathcal{E}_1 \subset \mathcal{E}$, and subsequential limit $\overline{g}^*$. Let $\mu_h^{{\eps}}$ be the measure on $\C$ whose restriction to each cell $\eta([x-{\eps}, x])$ for $x \in {\eps} \Z$
		is equal to $\deg^{{\eps}}(x)$ times $\mu_h |_{\eta([x-{\eps},x])}$. As we have assumed a.s.\ convergence of the  
		measure $\mu_h^{{\eps}}$ to $6 \mu_h$ along $\mathcal{E}_1$ with respect to the local Prokhorov topology, we have that as $\mathcal{E}_1 \ni {\eps}  \to 0$, 
		\[
		\int_{O} f(x) \overline{g}^{{\eps}}(x) d \mu^{{\eps}}_h(x)
		\to 6 \times \int_{O} f(x) \overline{g}^*(x) d \mu_h(x),
		\quad \forall \mbox{$f:\mathcal{P}(B_r(z)) \to \R$},
		\]
		almost surely. 
		By unpacking the definitions of $\overline{g}^{{\eps}}$ and $\mu_h^{{\eps}}$,
		and using Lemma \ref{lemma:cell-size-estimate} to control the approximation error, we have that a.s.\ as $\mathcal E \ni {\eps} \to 0$, 
		\[
		\begin{aligned}
		&\left | \int_{O} f(x) \overline{g}^{{\eps}}(x) d \mu^{{\eps}}_h(x) - 	{\meps}^{-1} \sum_{b \in {\VGeps}(O)} f(\eta(b)) \Gr{{\eps}}{B_{r}(z)}(a^{{\eps}}_y, b) \right|  \leq C {\eps}^{q}, \\
		&\quad \mbox{for constants $C = C(f)$ and $q = q(\gamma, f) > 0$,} \\
		& \quad \forall \mbox{$f:\mathcal{P}(B_r(z)) \to \R$}.
		\end{aligned}
		\]
		Combining the previous two indented equations together with Lemma \ref{lemma:integral-convergence-of-greens-function} 
		shows that $\overline{g}^*(x) = \frac{m_0}{6} \times G_{B_{r}(z)}(y, x)$. Hence, \eqref{eq:green-uniform} holds when $O\Subset B_r(z) \setminus \{y\}$.
		
		After identifying the limit, we now argue that the convergence occurs up to $\partial B_r(z)$. By a diagonal argument, we may arrange so that 
		$\overline g^{\eps}$ converges uniformly to $ \frac{m_0}{6} \times G_{B_{r}(z)}(y, x)$ along $\mathcal{E}$ on each compact subset of $B_r(z) \setminus \{y\}$. 
		Let $s > 0$ be given and let $r_1 \in (0, r)$ be such that $y$ lies at positive distance from $\cl(\A_{r_1, r}(z))$ and
		\[
		\sup_{x \in \cl(\A_{r_1, r}(z))} G_{B_{r}(z)}(y, x) < s/4
		\]
		Indeed, this is possible as $G_{B_{r}(z)}(y, \cdot)$ is smooth up to the boundary away from its pole and is identically zero on $\partial B_r(z)$. 
		By the maximum principle, Lemma \ref{lemma:maximum-principle}, and local uniform convergence we have that 
		\[
		\sup_{b \in {\VGeps}(\A_{r_1, r}(z))} \gr{{\eps}}{B_{\rho}}(a^{{\eps}}_y , b)
		=
		\sup_{b \in {\VGeps}(\partial B_{r_1}(z))} \gr{{\eps}}{B_{\rho}}(a^{{\eps}}_y , b)
		< s/2,
		\]
		for all ${\eps}$ sufficiently small. This implies convergence up to the boundary, $\partial B_r(z)$. 
	\end{proof}
	
	\subsection{Convergence of the Dirichlet problem}
	We follow a similar strategy as in the previous subsection to prove the following. 
	\begin{lemma} \label{lemma:discrete-dirichlet-problem}
		Fix a deterministic choice of $z\in\C$ and $r>0$ such that $B_r(z) \subset B_{\rho}$, a bounded, deterministic H\"{o}lder-continuous function $\psi: B_{r}(z) \to \R$, and a deterministic
		H\"{o}lder continuous function $\phi: \cl(B_r(z)) \to \R$. 
		The rescaled solution $\overline{f}^{\eps}$ (as in \eqref{eq:discrete-scaling}) to the discrete Dirichlet problem
		\[
		\begin{cases}
		\Delta^{{\eps}} f^{{\eps}}(a) = \psi(\eta(a)) \quad\mbox{for $a \in {\VGeps}(B_r(z)) \setminus {\VGeps}(\partial B_r(z))$} \\
		f^{{\eps}}(b) = {\meps} \times  \phi(\eta(b)) \quad \mbox{for $b \in {\VGeps}(\partial B_r(z))$}
		\end{cases}
		\]
		converges in probability in the uniform topology on $\cl(B_r(z))$ to the solution of the continuum Dirichlet problem, 
		\[
		\begin{cases}
		\Delta \overline{f} = m_0 \psi \times \mu_h  \quad\mbox{on $B_r(z)$} \\
		\overline{f}  = \phi \quad \mbox{on $\partial B_r(z)$}.
		\end{cases}
		\]
		Moreover, a.s.\ the limiting function $\overline{f}$ is H\"{o}lder continuous in $\cl(B_{r}(z))$. 
	\end{lemma}
	We first prove a pointwise convergence result concerning harmonic functions. 
	
	\begin{lemma} \label{lemma:convergence-of-harmonic-functions}
		For each deterministic choice of $z\in\C$, $r>0$ such that $B_r(z) \subset B_{\rho}$ and each deterministic H\"{o}lder continuous function $\psi: \cl(B_{r}(z)) \to \R$
		the following occurs. For $y \in \C$, recall that $\taueps{B_r(z)}{a_y^{{\eps}}}$
		is the first time simple random walk started at $a_y^{{\eps}} \in {\VGeps}$ exits ${\VGeps}(B_r(z)) \setminus {\VGeps}(\partial B_r(z))$. For $y \in B_{r}(z)$, let
		\[
		\tau_{B_r(z)}^{y} = \min \{t \geq 0 : X^{y}_t \not \in B_r(z) \}
		\]
		denote the first time LBM started at $y$ exits $B_r(z)$. 
		Then, 
		\[
		\lim_{{\eps} \to 0}  \E \left[ \psi \left( \eta \left( X^{a^{{\eps}}_y,{\eps}}_{\taueps{B_r(z)}{a_y^{{\eps}}}} \right) \right)  | h, \eta \right]  = \E \left[ \psi \left( X^y_{\tau_{B_r(z)}^{y}} \right) | h, \eta \right], \quad \forall y \in B_{r}(z),
		\]
		where the convergence is in probability as ${\eps} \to 0$. 
	\end{lemma}
	\begin{proof}
		The proof is nearly identical to that of Lemma \ref{lemma:integral-convergence-of-greens-function}, but for completeness we sketch the proof. 
		
		Let the parameters $y \in B_{r}(z)$ and $\psi: \cl(B_{r}) \to \R$, a H\"{o}lder continuous function, be given. 
		Consider the bounded
		functional on the space of continuous curves $w: C([0,\infty)) \to \cl(B_{r}(z))$ defined by
		\[
		\phi^{\psi}(w) := \psi(w(\tau)),
		\]  
		where $\tau$ is the first time $w(t) \not \in B_r(z)$.
		Recall from the beginning of Section \ref{subsec:green-kernel} that $\hat X^{y,{\eps}}$ is the piece-wise linear interpolation 
		of random walk started at $a^{{\eps}}_y$ on ${\VGeps}$ stopped upon exiting ${\VGeps}(B_\rho) \setminus {\VGeps}(\partial B_{\rho})$. 	
		By Theorem \ref{theorem:rw-to-lbm}, 
		\begin{equation} \label{eq:rw-to-lbm-harmonic}
		\lim_{{\eps} \to 0} \E \left[ \phi^{\psi}(\hat{X}^{y, {\eps}}) | h , \eta \right] = \E \left[ \psi(X^y_{\tau^{y}}) | h, \eta \right],
		\end{equation}
		in probability as ${\eps} \to 0$. 
		We conclude by observing that the term on the left-hand-side  of \eqref{eq:rw-to-lbm-harmonic} is asymptotically close to its discrete counterpart.
		Indeed, as $\psi$ is H\"{o}lder continuous on $\cl(B_{r}(z))$ and the cell-size of the mated-CRT map can be controlled by Lemma \ref{lemma:cell-size-estimate},
		\[
		|\phi^{\psi}(\hat{X}^{y, {\eps}}) -  \psi(\eta(X^{y}_{\tau^{{\eps}, y}}))| \leq C \times {\eps}^{q},
		\]
		in probability as ${\eps} \to 0$, for some $C = C(\psi) > 0$ and $q = q(\gamma, \psi) > 0$.
	\end{proof}

	This pointwise convergence is used together with tightness and integral convergence of the Green's function to prove uniform convergence of solutions 
	to the discrete Dirichlet problem. 
	
	\begin{proof}[Proof of Lemma \ref{lemma:discrete-dirichlet-problem}] 
		Decompose $f^{{\eps}} = f_1^{{\eps}} + f_2^{{\eps}}$ where 
		\[
		\begin{cases}
		\Delta^{{\eps}} f_1^{{\eps}}(a) = 0 \quad\mbox{for $a \in {\VGeps}(B_r(z)) \setminus {\VGeps}(\partial B_r(z))$} \\
		f_1^{{\eps}}(b) = {\meps} \phi(\eta(b)) \quad \mbox{for $b \in {\VGeps}(\partial B_r(z))$}
		\end{cases}
		\]
		and
		\[
		\begin{cases}
		\Delta^{{\eps}} f_2^{{\eps}}(a) = \psi(\eta(a)) \quad\mbox{for $a \in {\VGeps}(B_r(z)) \setminus {\VGeps}(\partial B_r(z))$} \\
		f_2^{{\eps}}(b) = 0 \quad \mbox{for $b \in {\VGeps}(\partial B_r(z))$}.
		\end{cases}
		\]
		\medskip
		
		{\it Step 1: Pointwise convergence.} \\
		Let a deterministic sequence $\mathcal{E}$ be given. Recall from \cite[Corollary 6.2.4]{lawler-limic-walks} that both $f_1^{{\eps}}$ and $f_2^{{\eps}}$
		can be expressed in terms of the Green's function of random walk.  Hence, by Lemma \ref{lemma:convergence-of-harmonic-functions} and Lemma \ref{lemma:integral-convergence-of-greens-function} respectively, 
		there is a deterministic subsequence $\mathcal{E}_1 \subset \mathcal{E}$ along which a.s.\
		$\overline{f}_1^{{\eps}}$ and $\overline{f}_2^{{\eps}}$  
		converge pointwise at every rational point in $B_r(z)$ to
		$\overline{f}_1$ and $\overline{f}_2$ where
		\[
		\begin{cases}
		\Delta \overline{f}_1 = 0 \quad\mbox{on $B_r(z)$} \\
		\overline{f}_1  = \phi \quad \mbox{on $\partial B_r(z)$}
		\end{cases}
		\]
		and
		\[
		\begin{cases}
		\Delta \overline{f}_2 = m_0 \mu_h \times \psi \quad\mbox{on $B_r(z)$} \\
		\overline{f}_2  = 0 \quad \mbox{on $\partial B_r(z)$}
		\end{cases}
		\]
		respectively. 
		\medskip
		
		{\it Step 2: Local uniform convergence.} \\ 
		We now show local uniform convergence in $B_r(z)$. By the maximum principle (Lemma \ref{lemma:maximum-principle}), $f_1^{{\eps}} \leq {\meps} \sup_{x \in B_r(z)} |\phi(x)|$.
		Moreover, as 
		\[
		f_2^{{\eps}} \pm \left( \sup_{x \in B_r(z)} |\psi(x)| \right) \times \Qeps{B_r(z)}
		\]
		is super (resp.\ sub) harmonic in ${\VGeps}(B_r(z)) \setminus {\VGeps}(\partial B_r(z))$ and zero on ${\VGeps}(\partial B_r(z))$, the maximum principle also implies that $|f_2^{{\eps}}| \leq \left(\sup_{x \in B_r(z)} |\psi(x)| \right) \Qeps{B_r(z)}$ 
		By the previous two sentences and Lemma \ref{lemma:expected-exit-time-explicit-bound} we have that $\overline{f}_1^{{\eps}}$ and $\overline{f}_2^{{\eps}}$ 
		are bounded by a random constant along a deterministic subsequence of $\mathcal{E}_1$. Thus, by Lemma \ref{lemma:discrete-potential-tightness} there is a random subsubsequence $\mathcal{E}_2 \subset \mathcal{E}_1$ along which $\overline{f}_1^{{\eps}}$ converges uniformly in $\overline{B}_r(z)$ and  $\overline{f}_2^{{\eps}}$ locally uniformly in $B_r(z)$ to continuous functions $\overline{f}_1^* : \overline{B}_r(z) \to \R$ and $\overline{f}_2^* : \overline{B}_r(z) \to \R$.
		By Step 1 we have that  $\overline{f}^*_1 = \overline{f}_1$ 
		and $\overline{f}^*_2 = \overline{f}_2$.
		
		\medskip
		{\it Step 3: Uniform convergence.} \\
		We now argue that $\overline{f}_2^{{\eps}}$ converges uniformly in $\cl(B_r(z))$ to $\overline{f_2}$ by a standard barrier argument. 
		Define
		\begin{equation}
		M_{\psi} = \sup_{a' \in B_r(z)} |\psi(a')|.
		\end{equation}
		and recall from \eqref{eq:discrete-exit-laplacian} that the function $\Qeps{B_r(z)}: {\VGeps}(\cl(B_r(z))) \to \R$ 
		satisfies
		\[
		\begin{cases}
		\Delta^{{\eps}} \Qeps{B_r(z)}(a) = -1 \quad\mbox{for $a \in {\VGeps}(B_r(z)) \setminus {\VGeps}(\partial B_r(z))$} \\
		\Qeps{B_r(z)}(b) = 0 \quad \mbox{for $b \in {\VGeps}(\partial B_r(z))$}.
		\end{cases}
		\]
		By the same argument as Steps 1 and 2, $M_{\psi} \times \Qeps{B_r(z)}$ converges locally uniformly in $B_r(z)$
		along $\mathcal{E}_2$
		to $M_{\psi} \times \overline{\mathfrak{Q}}_{B_r(z)}$ where
		\[
		\begin{cases}
		\Delta \overline{\mathfrak{Q}}_{B_r(z)} = -1 \quad\mbox{on $B_r(z)$} \\
		\overline{\mathfrak{Q}}_{B_r(z)}  = 0 \quad \mbox{on $\partial B_r(z)$}.
		\end{cases}
		\]
		By, for example, \cite[Proposition 2.5]{bou2022harmonic}, $\overline{\mathfrak{Q}}_{B_r(z)}$ is H\"{o}lder continuous in $\cl(B_r(z))$.
		Since $f^{{\eps}}_2 \pm M_{\psi} \times \Qeps{B_r(z)}$ is super (resp.\ sub) harmonic in ${\VGeps}(B_{r}(z)) \setminus {\VGeps}(\partial B_{r}(z))$ and is $0$ on ${\VGeps}(\partial B_{r}(z))$, we deduce convergence up to the boundary using the maximum principle exactly as in the proof of  Lemma \ref{lemma:uniform-convergence-of-greens-function}.
	\end{proof}

	\subsection{Convergence of the normalized expected exit time}
	Recall that for a domain $D \subset B_{\rho}$, the normalized exit time is defined as $\qeps{D}(a)= \sum_{b \in {\VGeps}(D) \setminus {\VGeps}(\partial D)} \gr{{\eps}}{D}(a,b)$
	and by \eqref{eq:discrete-normalized-exit-laplacian} solves the following discrete Dirichlet problem, 
	\[
	\begin{cases}
	\Delta^{{\eps}} \qeps{D}(\cdot) = -(\deg^{{\eps}})^{-1}(\cdot) \quad &\mbox{in ${\VGeps}(D) \setminus {\VGeps}(\partial D)$} \\
	\qeps{D}(\cdot) = 0 \quad &\mbox{otherwise}.
	\end{cases} 
	\]
	As the function $(\deg^{{\eps}})^{-1}$ is not H\"{o}lder continuous, Lemma \ref{lemma:discrete-dirichlet-problem} does not directly imply 
	that the normalized expected exit time converges. Nevertheless, we can use tightness and convergence of the Green's kernel to show the following. 
	
	\begin{lemma} \label{lemma:convergence-of-the-normalized-exit-time}
		Fix a deterministic choice of $z \in \C$ and $r > 0$ such that $B_r(z) \subset B_{\rho}$. As ${\eps} \to 0$, $\olqeps{B_r(z)}$ converges 
		uniformly in $\cl(B_r(z))$
		to $\overline{\mathfrak q}_{B_r(z)}$
		in probability
		where
		\[
		\begin{cases}
		\Delta \overline{\mathfrak q}_{B_r(z)} = -m_0/6 \times \mu_h \quad \mbox{on $B_r(z)$} \\
		\overline{\mathfrak q}_{B_r(z)} = 0 \quad \mbox{ on $\partial B_{r}(z)$}.
		\end{cases}
		\]
	\end{lemma}
	\begin{proof}
		Fix a deterministic choice of $z \in \C$ and $r > 0$ such that $B_r(z) \subset B_{\rho}$. We split the proof into steps. 
		\medskip
		
		{\it Step 1: Identify a subsequence.} \\
		Let a sequence of ${\eps}$s tending to $0$ be given and fix a deterministic subsequence $\mathcal{E}$ along which the convergence result in Lemma \ref{lemma:uniform-convergence-of-greens-function}
		occurs in $\cl(B_r(z))$ and Lemmas \ref{lemma:mu_h-convergence} and \ref{lemma:expected-exit-time-explicit-bound} occur almost surely. 
		Also suppose that along this subsequence the convergence in Lemma \ref{lemma:discrete-dirichlet-problem} occurs for 
		$\phi \equiv 0$ and for $\psi$ within a countable family of bump functions $\mathcal F(B_r(z))$ with the following property.
		For each rational $\delta' >0$ and rational $x' \in B_r(z)$ such that $B_{\delta'}(x') \subset B_r(z)$, there is a function in $\mathcal F(B_r(z))$ which takes values in $[0,1]$, is 1 on $B_{\delta'}(x')$, and is 0 on $B_r(z) \setminus B_{2 \delta'}(x')$.

		By Lemma \ref{lemma:expected-exit-time-explicit-bound}
		and our choice of subsequence, we have that that $\olqeps{B_r(z)}$
		is uniformly bounded by a random constant along $\mathcal{E}$. Since $\deg^{{\eps}} \geq 1$, by \eqref{eq:discrete-normalized-exit-laplacian}, 
		the unscaled function also has bounded $\Delta^{{\eps}}$-Laplacian. Thus, by Lemma \ref{lemma:discrete-potential-tightness}, there 
		is a random subsequence $\mathcal{E}_0$ along which $\olqeps{B_r(z)}$ converges. In the remainder of the proof we use convergence 
		of the Green's kernel away from its pole to identify the subsequential limit. 
		\medskip	
		
		{\it Step 2: Convergence away from the singularity.} \\
		Let  $x \in B_r(z) \cap \Q^2$ and a rational $\delta >0$ such that $B_{\delta}(x) \Subset B_r(z)$ be given. 
		We would like to apply the uniform convergence of the Green's kernel away from its pole (Lemma \ref{lemma:uniform-convergence-of-greens-function}). 
		So, we `carve out' the singularity at the pole of $\gr{{\eps}}{B_r(z)}(\cdot, \cdot)$ by considering, 
		\[
		\begin{aligned}
		\mathfrak q^{{\eps}, \delta^-}_{B_r(z)}(\cdot) &= \sum_{b \in {\VGeps}(B_r(z)) \setminus ({\VGeps}(\partial B_r(z)) \cup {\VGeps}(B_\delta(x)))} \gr{{\eps}}{B_r(z)}(\cdot, b) \\
		\mathfrak q^{{\eps}, \delta^+}_{B_r(z)}(\cdot) &= \sum_{b \in {\VGeps}(B_\delta(x))} \gr{{\eps}}{B_r(z)}(\cdot, b).
		\end{aligned}
		\]	
		Observe that the same argument leading to \eqref{eq:discrete-normalized-exit-laplacian}
		shows that 
		\begin{equation} \label{eq:trunc-laplacian-delta-}
		\begin{cases}
		\Delta^{{\eps}} \mathfrak q^{{\eps}, \delta^-}_{B_r(z)}(a) &= -(\deg^{{\eps}}(a))^{-1}  \\
		&\times 1\{a \in {\VGeps}(B_r(z)) \setminus ({\VGeps}(\partial B_r(z)) \cup {\VGeps}(B_\delta(x))) \}\\
		& \quad\mbox{on ${\VGeps}(B_r(z)) \setminus {\VGeps}(\partial B_r(z))$} \\
		\mathfrak q^{{\eps}, \delta^-}_{B_r(z)}  = 0 &\quad \mbox{otherwise.}
		\end{cases} 
		\end{equation}
		and
		\begin{equation}\label{eq:trunc-laplacian-delta+}
		\begin{cases} 
		\Delta^{{\eps}} \mathfrak q^{{\eps}, \delta^+}_{B_r(z)}(a) = -(\deg^{{\eps}}(a))^{-1} 1\{a \in {\VGeps}(B_\delta(x)) \} \quad\mbox{on ${\VGeps}(B_r(z)) \setminus {\VGeps}(\partial B_r(z))$} \\
		\mathfrak q^{{\eps}, \delta^+}_{B_r(z)}  = 0 \quad \mbox{otherwise}.
		\end{cases}
		\end{equation}
		Apply the same argument as in the last paragraph of Step 1 to a countable set of rationals $\delta' > 0$ and $x' \in \Q^2$
		to see that there is a random subsequence  of $\mathcal{E}_0$ (independent of $\delta$ and $x$)  along which $\olqeps{B_r(z)}$, $\overline{\mathfrak q}^{{\eps}, \delta^-}_{B_r(z)}$, $\overline{\mathfrak q}^{{\eps}, \delta^+}_{B_r(z)}$ converge uniformly in $\overline{B}_r(z)$ to continuous functions $\overline{\mathfrak q}^{*}_{B_r(z)}$, $\overline{\mathfrak q}^{*, \delta^-}_{B_r(z)}$, $\overline{\mathfrak q}^{*, \delta^+}_{B_r(z)}$ (with the limit possibly depending on the subsequence).
		
		Fix such a subsequence $\mathcal{E}_1 \subset \mathcal{E}_0$ and limits $\overline{\mathfrak q}^{*}_{B_r(z)}$, $\overline{\mathfrak q}^{*, \delta^-}_{B_r(z)}$, $\overline{\mathfrak q}^{*, \delta^+}_{B_r(z)}$ and observe that
		\begin{equation} \label{eq:asymptotic-equal}
		\overline{\mathfrak q}^{*}_{B_r(z)} = \overline{\mathfrak q}^{*, \delta^-}_{B_r(z)} + \overline{\mathfrak q}^{*, \delta^+}_{B_r(z)}.
		\end{equation}

		Let $g^{{\eps}}(\cdot) := {\eps}^{-1} \gr{{\eps}}{B_r(z)}(a^{{\eps}}_x , \cdot)$ so that by Lemma \ref{lemma:uniform-convergence-of-greens-function}, 
		\[
		g^{{\eps}}(\cdot) \to (m_0/6) G_{B_r(z)}(x,\cdot)
		\]
		uniformly in $\cl(B_r(z)) \setminus B_{\delta}(x)$ 
		as $\mathcal{E}_1 \ni {\eps} \to 0$.  As $g^{{\eps}}$ is zero on $\partial B_r(z)$, 
		it satisfies the H\"{o}lder continuity estimate \eqref{eq:boundary-estimate} from Lemma \ref{lemma:discrete-holder-continuity} in $B_r(z) \setminus B_{\delta}(x)$.
		This allows us to use $\overline{g}^{{\eps}}$ as a test function in Lemma \ref{lemma:mu_h-convergence} so that
		\[
		\begin{aligned}
		\int_{B_{r}(z) \setminus B_{\delta}(x)} \overline{g}^{{\eps}}(y) d \mu_h(y)
		&\to \frac{m_0}{6} \int_{B_{r}(z) \setminus B_{\delta}(x)} G_{B_r(z)}(x, y) d \mu_h(y).
		\end{aligned}
		\]
		as	$\mathcal{E}_1 \ni {\eps}  \to 0$.
		
		By unpacking the definition of $\overline{g}^{{\eps}}(y)$ and using Lemma \ref{lemma:cell-size-estimate}, 
		we see that the left-hand-side of the above converges to the same limit as $\overline{\mathfrak q}^{{\eps}, \delta^-}_{B_r(z)}(x)$. 
		This implies 
		\begin{equation} \label{eq:asymptotic-away}
		\overline{\mathfrak q}^{*, \delta^-}_{B_r(z)}(x) = \frac{m_0}{6} \int_{B_{r}(z) \setminus B_{\delta}(x)} G_{B_r(z)}(x, y) d \mu_h(y).
		\end{equation}
		\medskip
		
		{\it Step 3: Remove the singularity.}\\
		We claim that we may conclude once we show that 
		\begin{equation} \label{eq:asymptotic-small}
		\lim_{\delta \to 0} \sup_{y \in B_r(z)} \overline{\mathfrak{q}}^{*, \delta^+}_{B_r(z)}(y) = 0.
		\end{equation}
		Indeed, \eqref{eq:asymptotic-away}, \eqref{eq:asymptotic-small}, and \eqref{eq:asymptotic-equal} together imply that
		\[
		\overline{\mathfrak q}^{*}_{B_r(z)}(x) = \frac{m_0}{6} \int_{B_{r}(z)} G_{B_r(z)}(x, y) d \mu_h(y).
		\]
		Since $\Delta G_{B_r(z)}(x, \cdot) = -\delta_x(\cdot)$ 
		in $B_r(z)$, the integral above 
		is equal to $\overline{\mathfrak{q}}_{B_r(z)}(x)$. This completes the proof as $x$ was an arbitrary rational and $\overline{\mathfrak{q}}^{*}_{B_r(z)}$ is continuous.

		It remains to show \eqref{eq:asymptotic-small}. Recall the definition of smooth bump functions $\mathcal{F}(B_r(z))$ from Step 1. 
		Let $\phi_{\delta} \in \mathcal{F}(B_r(z))$ be a smooth, positive bump function which is 1 on $B_{\delta}(x)$ and 0 on $B_{r}(z) \setminus B_{2 \delta}(x)$ 
		and consider
		the solution to the Dirichlet problem, $f^{{\eps}}_{\delta}: {\VGeps}(\cl(B_r(z))) \to \R$ defined by
		\[
		\begin{cases}
		\Delta^{{\eps}} f^{{\eps}}_{\delta}(a) = \phi_{\delta}(\eta(a)) \quad \mbox{for $a \in {\VGeps}(B_{r}(z)) \setminus {\VGeps}(\partial B_r(z))$} \\
		f^{{\eps}}_{\delta}(b) = 0  \quad \mbox{for $b \in {\VGeps}(\partial B_r(z))$}.
		\end{cases}
		\]	
		Since $q^{{\eps}, \delta^+}_{B_r(z)} \pm f^{{\eps}}_{\delta}$ is sub (resp.\ super) harmonic in ${\VGeps}(B_{r}(z)) \setminus {\VGeps}(\partial B_{r}(z))$ and is $0$ on ${\VGeps}(\partial B_{r}(z))$, by the maximum principle, 
		\begin{equation} \label{eq:bound-q-by-f}
		|\mathfrak q^{{\eps}, \delta^+}_{B_r(z)}| \leq |f^{{\eps}}_{\delta}|, \quad \mbox{in ${\VGeps}(B_{r}(z))$}.
		\end{equation}
		By Lemma \ref{lemma:discrete-dirichlet-problem}, $\overline{f}^{{\eps}}_{\delta}$ converges uniformly in $B_r(z)$ to $\overline{f}_{\delta}$ where
		\[
		\begin{cases}
		\Delta \overline{f}_{\delta} = m_0 \phi_{\delta} \times \mu_h \quad \mbox{on $B_r(z)$} \\
		\overline{f}_{\delta} = 0  \quad \mbox{on $\partial B_r(z)$}.
		\end{cases}
		\]
		We can represent
		\begin{align*}
		\overline{f}_{\delta}(\cdot) &= m_0 \int_{B_r(z)}  \phi_{\delta}(y) G_{B_r(z)}(\cdot, y) d \mu_h(y) \\
		&\leq m_0 \int_{B_{2 \delta}(z)}  G_{B_1}(\cdot, y) d \mu_h(y) \\
		&\leq m_0 \times \Mrho{B_{2 \delta}(z)}.
		\end{align*}
		Hence, by the second half of Lemma \ref{lemma:bounded-exit-time} 
		there exists $q = q(\gamma) > 0$ such that
		\[
		\overline{f}_{\delta} \leq \delta^q  \quad \mbox{for all $\delta$ sufficiently small},
		\]
		which by \eqref{eq:bound-q-by-f}, implies \eqref{eq:asymptotic-small}, completing the proof. 		
	\end{proof}

	\section{Convergence of the divisible sandpile odometer} \label{sec:divisible-sandpile}
	Fix $\rho \in (0,1)$  and $t > 0$. We briefly recall the single-source divisible sandpile model, introduced in Section \ref{subsubsec:div-sandpile-informal-intro}, and define its odometer. 
	In the model, we start with a continuous amount of mass $(t {\eps}^{-1})$ at the origin $0 \in {\VGeps}$ and zero elsewhere. At each time step, vertices $a \in {\VGeps}$ which have mass $\sigma(a) > 1$, where $\sigma$ is the current configuration, are unstable and {\it topple}, distributing the excess mass $(\sigma(a) - 1)$ 
	equally among the neighbors of $a$. 
	Since ${\Geps}$ is an infinite graph eventually each site $a$ will have mass at most $1$.
	Denote the final configuration of mass by $s^{{\eps}}_{t}: {\VGeps} \to [0, 1]$.
	
	The {\it odometer}, which we denote by $v^{{\eps}}_{t}: {\VGeps} \to [0, \infty)$, tracks the total amount of mass each site has emitted during this process. 
	In particular,  by definition of the graph Laplacian, we have that 
	\begin{equation} \label{eq:final-sandpile-definition}
	s^{{\eps}}_{t} = t {\eps}^{-1} \delta_0 +  \deg^{{\eps}} \times \Delta^{{\eps}} \left( \frac{v^{{\eps}}_{t}}{\deg^{{\eps}}} \right).
	\end{equation}
	(The two $\deg^{{\eps}}$ factors in the above expression are needed to ensure the above expression matches the prior description of toppling.)
	
	By definition, we see that the divisible sandpile cluster satisfies
	\begin{equation} \label{eq:cluster-definition}
	D_{{\eps}}(t {\eps}^{-1})  = \{ x \in {\VGeps} : s^{{\eps}}_{t} > 0 \}
	= \cl( \{ x \in {\VGeps} : v_t^{{\eps}} > 0 \}). 
	\end{equation}
	The reason for the closure in \eqref{eq:cluster-definition} is that there are some sites with mass
	in the cluster which have not toppled. These sites, by definition, must have a neighbor which has toppled. 
	
	A fundamental property in the study of the sandpile is the {\it least action principle},
	\begin{equation} \label{eq:least-action-principle}
	\frac{v^{{\eps}}_{t}}{\deg^{{\eps}}} = \min \{ w:{\VGeps} \to [0,\infty) : \Delta^{{\eps}} w + {\deg^{{\eps}}(0)}^{-1} t {\eps}^{-1} \delta_0 \leq (\deg^{{\eps}})^{-1}\},
	\end{equation}
	where the minimum is pointwise. See, \eg, \cite[Lemma 3.2]{levine2009strong} for a proof. This property is closely related to \eqref{eq:least-super-solution}. 
	In fact, one sees from the definition that if $D_{{\eps}}(t {\eps}^{-1})  \subset {\VGeps}(B_{\rho}) \setminus {\VGeps}(\partial B_{\rho})$,
	then we can express, using \eqref{eq:discrete-normalized-green-laplacian}, the solution to 
	\eqref{eq:least-action-principle} as 
	\begin{equation} \label{eq:odometer-to-normalized-odometer}
	\frac{v^{{\eps}}_{t}(a)}{\deg^{{\eps}}(a)} =
	\begin{cases}
	\wteps(a) + (t {\eps}^{-1}) \times \gr{{\eps}}{B_\rho}(0, a) \quad \mbox{for $a \in {\VGeps}(B_{\rho}) \setminus {\VGeps}(\partial B_{\rho})$} \\
	0 \quad \mbox{otherwise}
	\end{cases}
	\end{equation}
	where
	\begin{equation} \label{eq:discrete-least-supersolution}
	\begin{aligned}
	\wteps := &\min \{ w: {\VGeps}(\cl(B_{\rho})) \to \R : \Delta^{{\eps}} w \leq (\deg^{{\eps}})^{-1} \mbox{ on ${\VGeps}(B_{\rho}) \setminus {\VGeps}(\partial B_{\rho})$} \\
	&\quad \mbox{ and $w \geq -(t {\eps}^{-1}) \times  \gr{{\eps}}{B_\rho}(0, \cdot)$}\},
	\end{aligned}
	\end{equation}
	where the minimum is pointwise. And, in general, 
	\begin{equation} \label{eq:lower-bound-odometer}
	\frac{v^{{\eps}}_{t}}{\deg^{{\eps}}} \geq \wteps + (t {\eps}^{-1}) \times \gr{{\eps}}{B_\rho}(0, \cdot) \quad \mbox{on ${\VGeps}(B_{\rho})$}.
	\end{equation}
	
	Let $\overline \wteps$ be the re-scaled version of $\wteps$ as in \eqref{eq:discrete-scaling}.
	We will show the convergence of $\overline \wteps$ to its continuum counterpart defined by \eqref{eq:least-super-solution}. In order to ensure the divisible sandpile cluster is contained in $B_{\rho}$, we consider
	$t \in (0, \TT)$ where 
	\begin{equation} \label{eq:cluster-stopping-time}
	\TT = \sup \{t > 0 : \Lambda_t \subset B_{\rho/3} \},
	\end{equation}
	where $\Lambda_t$ is the LQG harmonic ball as in Theorem \ref{theorem:harmonic-balls}. 
	We take $B_{\rho/3}$ rather than $B_{\rho}$ in the above definition 
	to allow some room for error in the discrete estimates of Section \ref{sec:idla-upper-bound}. 
	
	\begin{theorem} \label{theorem:convergence-divisible}
		Recall the median exit time $m_0$ defined in \eqref{eq:mo}. 
		For each $t > 0$, on the event $\{t < \TT\}$ we have that $\olwteps$ converges uniformly in 
		$B_{\rho}$ as ${\eps} \to 0$ to $\frac{m_0}{6} \overline{w}^{B_\rho}_t$
		in probability, where $\overline w_t^{B_\rho}$ is as in Theorem~\ref{theorem:harmonic-balls}. 
	\end{theorem}
	The extra factor of $\frac{m_0}{6}$ in Theorem \ref{theorem:convergence-divisible} is a result of the factor ${\meps}$
	appearing in the scaling definition \eqref{eq:discrete-scaling} and the fact that since the mated-CRT map is a planar triangulation,
	$\E[\deg^{{\eps}}(0)] = 6$. 
	
	A result similar to that of Theorem \ref{theorem:convergence-divisible} was proved on $\Z^d$ by Levine and Peres in \cite{levine2009strong}
	using different methods. Levine and Peres used precise estimates on the lattice Green's function in $\Z^d$ which are unavailable in our setting. 
	Our approach is similar to (but easier than) the method used to prove convergence of the {\it Abelian sandpile} --- see \cite{pegden-smart-sandpile-2013, bourabee-sandpile-2021}. (Also see \cite{barles-souganidis-convergence-1991} for a systematic approach to proving convergence of finite-difference schemes.)

	Before proceeding we point out that convergence of the odometer implies a lower bound
	on the divisible sandpile cluster.  
	\begin{prop}\label{prop:div-sandpile-lower-bound}
		For each $\delta \in (0, \rho)$ and $t > 0$, on the event $\{t < \TT\}$, 
		it holds except on an event of probability tending to 0 as ${\eps} \to 0$  that
		\[
		\Bdeltam(\Lambda_t) \subset \overline{D}_{{\eps}}(t {\eps}^{-1}).
		\]
	\end{prop}
	\begin{proof}[Proof assuming Theorem \ref{theorem:convergence-divisible}]
		
		Fix $\delta \in (0, \rho)$ and $t > 0$ and restrict to the event that $\{t < \TT\}$.
		Let a deterministic sequence of ${\eps}$s approaching zero be given. Recall from Theorem \ref{theorem:harmonic-balls} that
		\begin{equation} \label{eq:recall-cluster-def}
		\Lambda_t = \inte \left( \overline{\{ x \in B_\rho : \overline w_t^{B_\rho} > - t G_{B_\rho}(0, \cdot)\}} \right),
		\end{equation}
		where $G_{B_\rho}$ is the Green's function for $\Delta$ on $B_{\rho}$. 
		Since $\overline w_t^{B_\rho}$ is continuous and $G_{B_\rho}(0,\cdot)$ blows up at the origin, we can pick $r \in (0,1)$ so that
		\begin{equation} \label{eq:green-singularity-blow-up}
		\inf_{x \in B_{2 r}} t G_{B_\rho}(0, x) > \sup_{y \in B_{r}} -\overline w_t^{B_\rho}(y) .
		\end{equation}
		Since $\gr{{\eps}}{B_\rho}$ is superharmonic, by the maximum principle, Lemma \ref{lemma:maximum-principle},
		\begin{equation} \label{eq:singularity-blow-up}
		\inf_{a \in {\VGeps}(B_{r'})}\gr{{\eps}}{B_\rho}(0, a) 
		=
		\inf_{b \in {\VGeps}(\partial B_{r'})}\gr{{\eps}}{B_\rho}(0, b), \quad \forall r' \in (0, \rho). 
		\end{equation}
		
		Let $g^{{\eps}}(\cdot) := {\eps}^{-1} \gr{{\eps}}{B_\rho}(0 , \cdot)$ and define $\overline{g}^{\eps}$ as in~\eqref{eq:discrete-scaling}.
		By Theorem \ref{theorem:convergence-divisible}, there exists a deterministic subsequence, $\mathcal{E}$, along which a.s.\
		$\olwteps$  converges uniformly in $B_{\rho}$ to $\frac{m_0}{6} \overline w_t^{B_\rho}$. 
		Further, by Lemma \ref{lemma:uniform-convergence-of-greens-function} we may arrange so that, a.s.\ along $\mathcal{E}$,
		$\overline{g}^{{\eps}}$ converges uniformly in $\A_{r, \rho}$ to $\frac{m_0}{6} \times G_{B_\rho}(0 , \cdot)$.  By~\eqref{eq:recall-cluster-def} and the continuity of the two functions involved, there exists $s > 0$ so that
		\begin{equation} \label{eq:strictly-above-in-interior}
		\begin{aligned}
		\overline w_t^{B_\rho}(x) &> - t G_{B_\rho}(0, x) + s  \\
		&\mbox{for all $x \in \Lambda_t$ such that $\dist(x, \partial \Lambda_t) \geq \delta/4$}.
		\end{aligned} 
		\end{equation}
		We may further assume that a.s., the cell-size estimates of Lemma \ref{lemma:cell-size-estimate} hold for all sufficiently small ${\eps}\in \mathcal{E}$.

		Therefore, by the uniform convergence of $\olwteps$ in $B_{\rho}$ and the
		uniform convergence of $\overline{g}^{{\eps}}$ in $\A_{r, \rho}$ 	we have by \eqref{eq:green-singularity-blow-up}
		\[
		\wteps(a)> \sup_{b \in {\VGeps}(\partial B_{2 r})}  -(t {\eps}^{-1}) \times \gr{{\eps}}{B_\rho}(0, b) , \quad \forall a \in {\VGeps}(B_{r}), \quad \forall   {\eps} \in\mathcal E \mbox{ sufficiently small}
		\]
		and thus by \eqref{eq:singularity-blow-up}
		\[
		\wteps(a)> - (t {\eps}^{-1}) \times \gr{{\eps}}{B_\rho}(0, a) , \quad \forall a \in {\VGeps}(B_{r}), \quad \forall   {\eps} \in\mathcal E \mbox{ sufficiently small.}
		\]
		Also, by uniform convergence of $\olwteps$ in $B_{\rho}$ and
		uniform convergence of $\overline{g}^{{\eps}}$ in $\A_{r, \rho}$ 	and \eqref{eq:strictly-above-in-interior} we have
		\[
		\begin{aligned}
		\wteps(a) &> - (t {\eps}^{-1}) \times \gr{{\eps}}{B_\rho}(0, a)  \\
		&\mbox{$\forall a \in {\VGeps}(\Lambda_t \setminus B_{r})$ such that $\dist(\eta(a), \partial \Lambda_t) \geq \delta/2$} \\
		&\mbox{$\forall {\eps} \in \mathcal{E} $ sufficiently small.}
		\end{aligned} 
		\]
		Combining the previous two indented equations with the definition of $\overline{D}_{{\eps}}$, \eqref{eq:cluster-definition}, and \eqref{eq:lower-bound-odometer}  completes the proof. 
	\end{proof}

	\subsection{Basic properties of the divisible sandpile} \label{subsec:basic-properties}
	For completeness, we recall some basic properties of the discrete least supersolution. 
	Since these results are standard, the proofs are given in Appendix \ref{sec:obstacle-appendix}.
	We also note that continuum analogs of these properties are proven in~\cite[Appendix A]{bourabee-sandpile-2021}.

	\begin{lemma} \label{lemma:discrete-basic-properties}
		Recall the definition of $\qeps{B_\rho}$ from \eqref{eq:normalized-exit-time}. 
		For all $t > 0$ and ${\eps} > 0$,  
		\[
		\max(-\qeps{B_{\rho}}, -(t {\eps}^{-1}) \times \gr{{\eps}}{B_\rho}(0, \cdot)) \leq \wteps \leq 0
		\]
		and  $\Delta^{{\eps}} \wteps \leq (\deg^{{\eps}})^{-1}$ on ${\VGeps}(B_{\rho}) \setminus {\VGeps}(\partial B_{\rho})$.
	\end{lemma}

	For the statement of the next lemma, we denote the {\it discrete cluster} by 
	\begin{equation} \label{eq:discrete-cluster}
	\Lambdateps = \{ a \in {\VGeps}(B_{\rho}) :  \wteps(a) > - (t {\eps}^{-1}) \gr{{\eps}}{B_{\rho}}(0, a)\}.
	\end{equation}
	Observe that by \eqref{eq:cluster-definition} for $t \in (0,\TT)$, 
	\begin{equation} \label{eq:cluster-and-lambdat}
	D_{{\eps}}(t {\eps}^{-1})  = \cl(\Lambdateps). 
	\end{equation}
	\begin{lemma} \label{lemma:discrete-laplacian-bound}
		For all $t > 0$ and ${\eps} > 0$, we have that $\Delta^{{\eps}} \wteps = (\deg^{{\eps}})^{-1}$ on $\Lambdateps$
		and $0 \leq \Delta^{{\eps}} \wteps \leq (\deg^{{\eps}})^{-1}$ on ${\VGeps}(B_{\rho}) \setminus {\VGeps}(\partial B_{\rho})$. 
		Moreover,  $\Lambdateps$ is connected and if $ t {\eps}^{-1}  > 1$,  then $0 \in \Lambdateps$.
	\end{lemma}

	We conclude with a discrete conservation of mass lemma. 
	\begin{lemma} \label{lemma:discrete-conservation-of-mass}
		For all $t > 0$ and ${\eps} > 0$, we have that $|\Lambdateps| \leq   t {\eps}^{-1} $. Moreover, 
		if $\Lambdateps \subset \inte({\VGeps}(B_{\rho}) \setminus {\VGeps}(\partial B_{\rho}))$, 
		then  
		\[
		\sum_{a \in {\VGeps}(B_{\rho}) \setminus {\VGeps}(\partial B_{\rho})} \Delta^{{\eps}} \wteps(a) \deg^{{\eps}}(a) =  t {\eps}^{-1} .
		\] 
	\end{lemma}

	\subsection{Convergence of the odometer along subsequences}
	In this subsection we prove tightness of the re-scaled odometer function $\overline \wteps$ defined in~\eqref{eq:discrete-least-supersolution} and \eqref{eq:discrete-scaling}. In subsequent subsections we show that, for $t \in (0, \TT)$, each subsequential limit is given uniquely by \eqref{eq:least-super-solution}. 
	\begin{lemma} \label{lemma:subsequential-convergence}
		Almost surely, for each $t > 0$ and each deterministic sequence of ${\eps}$s converging to 0, 
		there exists a $\sigma(h,\eta)$-measurable subsequence 
		and a H\"{o}lder continuous function $\overline{w}^*_t \in C(\cl(B_\rho))$ such that $\olwteps \to \overline{w}^*_t$
		uniformly in $\cl(B_{\rho})$. 
	\end{lemma}
	
	We will prove Lemma \ref{lemma:subsequential-convergence} by combining Lemma \ref{lemma:discrete-potential-tightness}, the results of the previous subsection,
	and the following lemma.  
	\begin{lemma} \label{lemma:equibounded}
		A.s., for each  deterministic sequence of ${\eps}$s converging to 0, there exists a deterministic subsequence $\mathcal{E}$ and a random constant $C$ so that for each $t > 0$, 
		\[
		\sup_{x \in B_{\rho}} |\olwteps(x)| \leq C, \quad \forall {\eps} \in \mathcal{E}.
		\]
	\end{lemma}
	
	\begin{proof}
		By Lemma \ref{lemma:discrete-basic-properties}, 
		\begin{equation} \label{eq:odometer-bounds}
		-\qeps{B_{\rho}} \leq \wteps \leq 0, \quad \forall t > 0, \forall {\eps} > 0 
		\end{equation}
		and by Lemma \ref{lemma:exit-time-upper-bound} and \eqref{eq:meps-scaling}
		\begin{equation} 
		-\qeps{B_{\rho}} \geq -C_0 {\meps}  \Mrho{B_{\rho}}
		\end{equation}
		for some deterministic constant $C_0 = C_0(\rho, \gamma)$
		with polynomially high probability as ${\eps} \to 0$, where $\Mrho{B_{\rho}}$ is as in \eqref{eq:largest-exit-time-rv}.
		By Lemma  \ref{lemma:bounded-exit-time},  $\Mrho{B_{\rho}}$ is finite almost surely. 
		Hence, the previous two indented inequalities complete the proof by the definition of $\olwteps$ and the Borel-Cantelli lemma. 
	\end{proof}
	We now use the tightness lemma,  Lemma \ref{lemma:discrete-potential-tightness}, to prove subsequential convergence. 
	\begin{proof}[Proof of Lemma \ref{lemma:subsequential-convergence}] 	
		By Lemma \ref{lemma:equibounded} and Lemma \ref{lemma:discrete-laplacian-bound}, for each $t>0$, the assumptions of Lemma \ref{lemma:discrete-potential-tightness}
		are satisfied along a deterministic subsequence for the functions $\wteps$. Hence, $\olwteps$ converges a.s.\ locally uniformly in $B_{\rho}$ along a random subsequence. To see that the convergence occurs uniformly in $\cl(B_{\rho})$, we use the standard barrier argument given in Step 3 of the proof of Lemma \ref{lemma:discrete-dirichlet-problem}. To implement the barrier argument we use that $0 \leq \Delta^{{\eps}} w^{{\eps}}_t \leq 1$. 
	\end{proof}
	
	\subsection{Properties of the subsequential limit}
	Convergence of the discrete Dirichlet problem, Lemma \ref{lemma:discrete-dirichlet-problem}, is used
	to establish some properties of each subsequential limit of the odometer.  
	We will use this in the next subsection to uniquely identify the limit. 
	
	\begin{lemma} \label{lemma:subsequential-properties}
		A.s., for each deterministic sequence of ${\eps} \to 0$, 
		there exists a deterministic subsequence so that a subsequential limit $\overline{w}^*_t$ of $\overline \wteps$ along this subsequence (as given by Lemma \ref{lemma:subsequential-convergence}) satisfies the following:
		\begin{enumerate}
			\item $0 \geq \overline{w}^*_t \geq -t \frac{m_0}{6} G_{B_{\rho}}(0, \cdot)$ on $\cl(B_{\rho})$;
			\item $\Delta \overline{w}^*_t \geq 0$ on $B_{\rho}$;
			\item $\Delta \overline{w}^*_t \leq \frac{m_0}{6} \mu_h$ on $B_{\rho}$ and $\Delta \overline{w}^*_t = \frac{m_0}{6} \mu_h$ on $\{x \in B_{\rho} : \overline{w}^*_t(x) > -t \frac{m_0}{6} G_{B_{\rho}}(0, x)\}$.
		\end{enumerate}
	\end{lemma}

	\begin{proof}
		By Lemma \ref{lemma:uniform-convergence-of-greens-function}, we may take a deterministic subsequence along which, 
		the convergence of 	$\overline g^{\eps}$ from that lemma, with $y = 0$ and $B_r(z) = B_{\rho}$, occurs almost surely. 
		Take another subsequence
		for which the convergence of Lemma \ref{lemma:discrete-dirichlet-problem}
		occurs a.s.\ for all $z \in \Q^2$ and rational $r > 0$,  $\psi \equiv 0$, and 
		all $\phi \in \mathcal{F}(B_r(z))$, the set of polynomials with rational coefficients
		restricted to $\cl(B_r(z))$. Further suppose that along this subsequence the convergence in Lemma \ref{lemma:convergence-of-the-normalized-exit-time} holds for all 
		$z \in \Q^2$ and rational $r > 0$.  By Lemma \ref{lemma:subsequential-convergence}, there is a random subsequence $\mathcal E$ of the preceding subsequence such that $\lim_{\mathcal{E} \ni {\eps} \to 0} \olwteps \to \overline{w}^*_t$.  
		\medskip
		
		{\it Step 1: Proof of (1).} \\
		By Lemma \ref{lemma:discrete-basic-properties}, 
		\begin{equation}
		0 \geq \wteps \geq -(t {\eps}^{-1}) \gr{{\eps}}{B_{\rho}}(0, \cdot) \quad \mbox{on ${\VGeps}(\cl(B_{\rho}))$} , \quad \forall {\eps}>0.
		\end{equation}
		This inequality persists in the limit by the uniform convergence of $\wteps$ and the convergence of the rescaled version of ${\eps}^{-1} \gr{\eps}{B_{\rho}}$ to $\frac{m_0}{6} G_{B_{\rho}}$, for $\mathcal{E} \ni {\eps} \to 0$ as given by Lemma \ref{lemma:uniform-convergence-of-greens-function}. 
		\medskip
		
		{\it Step 2: Proof of (2).}  \\
		Let $\delta > 0$, a rational $r > 0$, and a $z \in \Q^2$ such that $B_r(z) \Subset B_{\rho}$ be given. 
		For each ${\eps} \in \mathcal{E}$ consider the discrete harmonic function $f_\delta^{\eps}$ defined by
		\[
		\begin{cases}
		\Delta^{{\eps}} f^{{\eps}}_{\delta}(a) = 0 \quad\mbox{for $a \in {\VGeps}(B_r(z)) \setminus {\VGeps}(\partial B_r(z))$} \\
		f^{{\eps}}_{\delta}(b)  = m_{{\eps}} \phi^{\delta}(\eta(b)) + \delta \quad \mbox{for $b  \in {\VGeps}(\partial B_r(z))$}
		\end{cases}
		\]
		where $\phi^{\delta} \in \mathcal{F}(B_r(z))$ is a rational polynomial satisfying 
		\begin{equation} \label{eq:approximation-by-delta}
		\sup_{x \in \cl(B_r(z))} |\phi^{\delta}(x) - \overline{w}^*_t(x)| < \delta/4.
		\end{equation}
		By Lemma \ref{lemma:discrete-laplacian-bound},  $\wteps$ is $\Delta^{{\eps}}$-subharmonic. 
		Thus, by the maximum principle, Lemma \ref{lemma:maximum-principle}, we have
		\[
		\sup_{a \in {\VGeps}(B_r(z)) \setminus {\VGeps}(\partial B_r(z))}(\wteps - f^{{\eps}}_{\delta})(a) \leq \sup_{b  \in {\VGeps}(\partial B_r(z))}
		(\wteps - f^{{\eps}}_{\delta})(b).
		\]
		As $\olwteps$ converges uniformly to $\overline{w}^*_t$ in $\cl(B_{\rho})$, by definition of $f_\delta^{\eps}$
		and \eqref{eq:approximation-by-delta}, we have
		\[
		\sup_{b  \in {\VGeps}(\partial B_r(z))}
		(\wteps - f^{{\eps}}_{\delta})(b) < 0, \quad \forall \mbox{ small ${\eps} \in \mathcal{E}$}.
		\]
		By combining the previous two indented equations, we have 
		\begin{equation} \label{eq:lower-bound}
		f^{{\eps}}_{\delta}(a) \geq \wteps(a), \quad \forall a \in {\VGeps}(B_r(z)), \quad \forall \mbox{ small ${\eps} \in \mathcal{E}$}.
		\end{equation}
		
		By Lemma \ref{lemma:discrete-dirichlet-problem}, we have that $\overline f^{{\eps}}_{\delta}$ converges uniformly in 
		$\cl(B_r(z))$ to a harmonic function $\overline f^{*}_{\delta}: \cl(B_r(z)) \to \R$ which satisfies
		\begin{equation} \label{eq:harmonic-function-same-bdry-values}
		\begin{cases}
		\Delta \overline f^{*}_{\delta} = 0 \quad\mbox{on $B_r(z)$} \\
		\overline f^{*}_{\delta}  =  \phi^{\delta} + \delta \quad \mbox{on $\partial B_r(z)$}.
		\end{cases}
		\end{equation}
		Therefore, 
		\begin{align*}
		\overline{w}^*_t(z) &\leq \overline f^{*}_{\delta}(z) \qquad \mbox{(by \eqref{eq:lower-bound} and uniform convergence)} \\
		&= \int_{\partial B_{r}(z)} \overline{f}^*_{\delta} \qquad \mbox{(by harmonicity)} \\
		&= \int_{\partial B_{r}(z)} (\phi^{\delta} + \delta) \qquad \mbox{(by \eqref{eq:harmonic-function-same-bdry-values})}. 
		\end{align*}
		As the choice of rational $r > 0$ and $z \in \Q^2$ was arbitrary and the above inequality holds for any $\delta > 0$, by \eqref{eq:approximation-by-delta},
		and the continuity of $\overline{w}^*_t(z)$, we have
		\[
		\overline{w}^*_t(z) \leq \int_{\partial B_{r}(z)} \overline{w}^*_t, \quad \mbox{for all $r > 0$ and $z \in \C$ such that $B_r(z) \subset B_{\rho}$},
		\]
		completing the proof.
		\medskip
		
		{\it Step 3: Proof of (3).} \\ 
		We will show that
		\begin{equation} \label{eq:limit-admissible}
		\Delta \overline{w}^*_t \leq \frac{m_0}{6} \mu_h  \quad \mbox{on $B_{\rho}$},
		\end{equation}
		and then 
		\begin{equation} \label{eq:limit-equal}
		\Delta \overline{w}^*_t = \frac{m_0}{6} \mu_h  \quad \mbox{on $\{\overline{w}^*_t > -t \frac{m_0}{6} G_{B_{\rho}}\}$},
		\end{equation}
		appealing to the argument of Step 2. 
		
		Let rational $r > 0$ and $z \in \Q^2$ such that $B_r(z) \subset B_{\rho}$ and $\delta > 0$ be given. Recall that for each ${\eps} > 0$, 
		$\qeps{B_r(z)}$ satisfies 
		\[
		\begin{cases}
		\Delta^{{\eps}} \qeps{B_r(z)}(a) = -(\deg^{{\eps}}(a))^{-1} \quad\mbox{for $a \in {\VGeps}(B_r(z)) \setminus {\VGeps}(\partial B_r(z))$} \\
		\qeps{B_r(z)}(b)  = 0 \quad \mbox{for $b  \in {\VGeps}(\partial B_r(z))$}.
		\end{cases}
		\]
		By Lemma \ref{lemma:discrete-laplacian-bound}, 
		\begin{equation} \label{eq:difference-super-harmonic}
		\Delta^{{\eps}} (\qeps{B_r(z)} + \wteps)(a) \leq 0 \quad\mbox{for $a \in {\VGeps}(B_r(z)) \setminus {\VGeps}(\partial B_r(z))$}.
		\end{equation}
		By Lemma \ref{lemma:convergence-of-the-normalized-exit-time}, $\olqeps{B_r(z)}$ converges uniformly 
		in $\cl(B_r(z))$ to $\overline{\mathfrak q}_{B_r(z)}$
		where	
		\begin{equation} \label{eq:limit-equation-laplacian}
		\begin{cases}
		\Delta \overline{\mathfrak q}_{B_r(z)} = -m_0/6 \times \mu_h \quad \mbox{on $B_r(z)$} \\
		\overline{\mathfrak q}_{B_r(z)} = 0 \quad \mbox{ on $\partial B_{r}(z)$}.
		\end{cases}
		\end{equation}
		
		As $(\qeps{B_r(z)} + \wteps)$ is discrete superharmonic \eqref{eq:difference-super-harmonic}, the same argument as in Step 2 shows that
		\begin{equation} \label{eq:limit-equation-sub-harmonic}
		\Delta (\overline{\mathfrak q}_{B_r(z)} + \overline{w}^{*}_t) \leq 0 \quad\mbox{on $B_r(z)$}.
		\end{equation}
		Combining this with \eqref{eq:limit-equation-laplacian} shows \eqref{eq:limit-admissible}. 
		
		We now show \eqref{eq:limit-equal}. Fix rational $r > 0$ and $z \in \Q^2$ so that 
		\[
		B_r(z) \subset \{x \in B_{\rho}: \overline{w}^*_t(x) > -t \frac{m_0}{6} G_{B_{\rho}}(x)\}.
		\]
		By the continuity and uniform convergence of the two functions involved, for all ${\eps} \in \mathcal{E}$ sufficiently small, we have that 
		\begin{equation} \label{discrete-cluster-contained}
		\wteps(a) > -(t {\eps}^{-1}) \gr{B_{\rho}}{{\eps}}(0,a), \quad \forall a \in {\VGeps}(B_{r}(z)). 
		\end{equation}
		The inequality \eqref{discrete-cluster-contained}
		and Lemma \ref{lemma:discrete-laplacian-bound} together show that
		\begin{equation} \label{eq:difference-harmonic}
		\Delta^{{\eps}} (\qeps{B_r(z)} + \wteps)(a) = 0 \quad \forall a \in {\VGeps}(B_r(z)) \setminus {\VGeps}(\partial B_r(z))
		\end{equation}
		for all ${\eps} \in \mathcal{E}$ sufficiently small.
		Again applying the argument of Step 2 twice, we see that
		\begin{equation} \label{eq:limit-equation-harmonic}
		\Delta (\overline{\mathfrak q}_{B_r(z)} + \overline{w}^{*}_t) = 0 \quad\mbox{on $B_r(z)$},
		\end{equation}
		completing the proof by \eqref{eq:limit-equation-laplacian}.	
	\end{proof}

	\subsection{Proof of Theorem \ref{theorem:convergence-divisible}}
	Note that by the definition of $\TT$ and Theorem \ref{theorem:harmonic-balls},
	\[
	\Lambda_t = \inte\left( \cl \left( \{x \in B_\rho : \overline{w}^{B_\rho}_t > - G_{B_{\rho}}(0, \cdot) \} \right) \right), \quad \forall t \in (0, \TT).
	\]
	By (1) and (3) of Lemma \ref{lemma:subsequential-properties} we have that each possible subsequential limit $\frac{6}{m_0} \overline{w}^*_t$ as in Lemma \ref{lemma:subsequential-convergence} is admissible in  
	\eqref{eq:least-super-solution} 
	and hence
	\begin{equation} \label{eq:admissible-supersolution}
	\frac{m_0}{6} \overline{w}^{B_\rho}_t \leq \overline{w}^*_t, \quad \forall t \in (0, \TT).
	\end{equation}
	For the other direction, we note that by part (c) of Theorem~\ref{theorem:harmonic-balls}, we have, for each $t \in (0,\TT)$, 
	\[
	\Delta \overline w_t^{B_\rho} = \mu_h |_{\Lambda_t}  \quad \mbox{on $B_\rho$} 
	\]
	and by (3) of Lemma~\ref{lemma:subsequential-properties}
	\[
	\Delta \overline{w}^*_t = \frac{m_0}{6} \mu_h \quad \mbox{on}  \quad \{x \in B_{\rho} : \overline{w}^*_t(x) > -t \frac{m_0}{6} G_{B_{\rho}}(0, x)\}
	\]
	and by~\eqref{eq:admissible-supersolution}, 
	\[
	\Lambda_t \subset   \inte \left( \cl \left( \{x \in B_\rho : \overline{w}^*_t > - \frac{m_0}{6} G_{B_{\rho}}(0, \cdot) \} \right) \right), \quad \forall t \in (0, \TT).
	\]
	Consequently, by the above three displays we have for each $t \in (0,\TT)$, (in the weak sense)
	\[
	\Delta(\overline{w}^*_t - \frac{m_0}{6} \overline{w}_t^{B_{\rho}}) = 0 \qquad \mbox{on $\Lambda_t$} 
	\]
	and 
	\[
	\Delta(\overline{w}^*_t - \frac{m_0}{6} \overline{w}_t^{B_{\rho}}) = \Delta \overline{w}^*_t \geq 0 \qquad \mbox{on $B_{\rho} \setminus \Lambda_t$}  \, 
	\]
	with the last inequality coming from (2) of Lemma~\ref{lemma:subsequential-properties}.
	
	Therefore, for each $t \in (0,\TT)$,
	\[
	\begin{cases}
	\Delta(\overline{w}^*_t - \frac{m_0}{6} \overline{w}^{B_\rho}_t) \geq 0 \quad \mbox{on $B_\rho$} \\
	\overline{w}^*_t - \frac{m_0}{6} \overline{w}^{B_\rho}_t = 0 \quad \mbox{on $\partial B_\rho$} 
	\end{cases}
	\]
	and hence by the maximum principle in the continuum, 
	\begin{equation} \label{eq:admissible-subsolution}
	\frac{m_0}{6} \overline{w}^{B_\rho}_t \geq \overline{w}^*_t
	\end{equation}
	which, together with \eqref{eq:admissible-supersolution} shows that $\overline{w}^*_t = \overline{w}^{B_\rho}_t$. \qed

	\section{IDLA lower bound} \label{sec:idla-lower-bound}
	We use the convergence of the divisible sandpile and the argument of Lawler-Bramson-Griffeath \cite{lawler1992internal}
	to prove the following lower bound on the IDLA cluster. As in the previous section we fix a $\rho \in (0,1)$ so that $B_{\rho} \Subset B_1$. 
	Recall the notation for the subset of $\C$ corresponding to a subset of ${\VGeps}$, \eqref{eq:vertex-scaling}.

	\begin{prop} \label{prop:IDLA-lower-bound}
		Recall the time $\TT$ from~\eqref{eq:cluster-stopping-time}.
		For each $\delta \in (0, \rho)$ and $t > 0$, on the event $\{t < \TT\}$, 
		it holds except on an event of probability tending to 0 as ${\eps} \to 0$  that
		\[
		\Bdeltam(\Lambda_t) \subset \overline{A}_{{\eps}}(\lfloor t {\eps}^{-1} \rfloor)
		\]
		where $\Lambda_t$ is as in Theorem \ref{theorem:harmonic-balls}.
		Moreover, this lower bound holds even when walkers are stopped upon 
		exiting ${\VGeps}(\Lambda_t)$. 	
	\end{prop}

	Our proof of Proposition \ref{prop:IDLA-lower-bound} relies on the following result of Lawler-Bramson-Griffeath \cite{lawler1992internal}. 
	Recall Remark \ref{remark:overload-notation-potential}.
	\begin{lemma}[Reformulation of Section 3 in \cite{lawler1992internal}] \label{lemma:general-idla-lower-bound}
		Suppose  $O_1^{{\eps}} \subset O_2^{{\eps}} \subset {\VGeps}(B_{\rho})$ are $\sigma(h,\eta)$-measurable domains containing the origin
		such that for $K_{\eps} > 0$,  some deterministic $q > 0$ and random $\sigma(h,\eta)$-measurable finite random variable $\mathfrak{s} > 0$, 
		\begin{equation} \label{eq:mean-value-ineq}
		K_{\eps} \Gr{{\eps}}{O_2^{{\eps}}}(0, a) \geq (1 + \mathfrak{s}) \sum_{b \in O_2^{{\eps}}} \Gr{{\eps}}{O_2^{{\eps}}}(b, a), \quad \forall a \in O_1^{{\eps}},
		\end{equation}
		and 
		\begin{equation} \label{eq:L-lower-bound}
		\frac{ \sum_{b \in O^{{\eps}}_2} \Gr{{\eps}}{O_2^{{\eps}}}(b, a) }{ {\Gr{{\eps}}{O_2^{{\eps}}}(a, a)}} \geq {\eps}^{-q}, \quad \forall a \in O_1^{{\eps}}, 
		\end{equation}
		with probability approaching one as ${\eps}$ goes to zero,  then 
		\begin{equation} \label{eq:lower-bound-arbitrary}
		O_1^{{\eps}} \subset A_{{\eps}}(\lfloor K_{\eps} \rfloor)
		\end{equation}
		with probability approaching one as ${\eps}$ goes to zero.
		Moreover, the bound \eqref{eq:lower-bound-arbitrary} holds even when walkers are stopped upon exiting $O_2^{{\eps}}$. 
	\end{lemma}
	
	We will give the proof of Lemma~\ref{lemma:general-idla-lower-bound} at the end of this section, following the argument of \cite{lawler1992internal}. We will first prove Proposition \ref{prop:IDLA-lower-bound} by verifying the two conditions of Lemma \ref{lemma:general-idla-lower-bound}. 
	Specifically, we will show that an approximate mean-value inequality \eqref{eq:mean-value-ineq} is satisfied for the discrete divisible sandpile cluster $\Lambdateps$.
	
	\begin{lemma} \label{lemma:mean-value-ineq}
		For each $\delta \in (0, \rho)$ and $t > 0$ on the event $\{t < \TT\}$, 
		there exists a $\sigma(h,\eta)$-measurable finite random variable $\mathfrak{s} = \mathfrak{s}(\delta) > 0$ such that, except on an event of probability tending to 0 as ${\eps} \to 0$ , 
		\begin{equation} \label{eq:mean-value-ineq-verify}
		(t {\eps}^{-1}) \Gr{{\eps}}{\Lambdateps}(0, a) \geq (1 + \mathfrak{s}) \sum_{b \in \Lambdateps} \Gr{{\eps}}{\Lambdateps}(b, a), \quad  \mbox{for all $a \in \Lambdateps$ s.t. $\dist(\eta(a), \partial \eta(\Lambdateps)) > \delta$}.
		\end{equation}
	\end{lemma}

	We will also prove the lower bound in  \eqref{eq:L-lower-bound}. 
	\begin{lemma} \label{lemma:L-lower-bound}
		There exists a deterministic $q > 0$ such that, for each $\delta \in (0, \rho)$ and $t > 0$ on the event $\{t < \TT\}$, 
		it holds except on an event of probability tending to 0 as ${\eps} \to 0$  that
		\begin{equation} \label{eq:L-lower-bound-verify}
		\frac{\sum_{b \in \Lambdateps} \Gr{{\eps}}{\Lambdateps}(b, a)  }{ {\Gr{{\eps}}{\Lambdateps}(a, a)}} \geq {\eps}^{-q},
		\quad  \mbox{for all $a \in \Lambdateps$ s.t. $\dist(\eta(a), \partial \eta(\Lambdateps)) > \delta$}.
		\end{equation} 
	\end{lemma}
	These lemmas imply Proposition \ref{prop:IDLA-lower-bound}. 
	
	\begin{proof}[Proof of Proposition \ref{prop:IDLA-lower-bound}, assuming Lemmas \ref{lemma:general-idla-lower-bound}, \ref{lemma:mean-value-ineq}, and \ref{lemma:L-lower-bound}]
		We restrict to the event that $\{t < \TT\}$ for $t > 0$. 
		By Lemmas \ref{lemma:mean-value-ineq} and  \ref{lemma:L-lower-bound} (with $\delta/4$ instead of $\delta$) we may apply Lemma \ref{lemma:general-idla-lower-bound} 
		with the sets $O^{{\eps}}_1 = \{a \in \Lambdateps:\dist(\eta(a), \partial \eta(\Lambdateps)) > \delta/4 \}$ and $O_2^{{\eps}} = \Lambdateps$ 
		and $K_{\eps} = t \eps^{-1}$ 
		which shows, together 
		with the cell-size estimates given by 
		Lemma \ref{lemma:cell-size-estimate}
		\[
		\Bdeltamm{\delta/2}(\overline{\Lambdateps}) \subset \overline{A}_{{\eps}}(\lfloor t {\eps}^{-1} \rfloor),
		\]
		except on an event of probability going to zero as ${\eps} \to 0$,  with the bound holding even when walkers are stopped upon exiting $\Lambdateps$. 
		By Proposition \ref{prop:div-sandpile-lower-bound} and \eqref{eq:cluster-and-lambdat}, (and the cell-size estimates again) we have
		\[
		\Bdeltamm{\delta/2}(\Lambda_t) \subset \overline{\Lambdateps},
		\]
		except on an event of probability going to zero as ${\eps} \to 0$. 
	\end{proof}

	We now prove Lemmas \ref{lemma:general-idla-lower-bound}, \ref{lemma:mean-value-ineq}, and \ref{lemma:L-lower-bound}.
	We start with Lemma~\ref{lemma:L-lower-bound}, which relies on the estimates established in \cite{berestycki2020random} as recalled in Section \ref{sec:preliminaries}. 
	\begin{proof}[Proof of Lemma \ref{lemma:L-lower-bound}]
		We restrict to the event that $\{t < \TT\}$ for $t > 0$.
		By the fact $\Lambdateps \subset {\VGeps}(B_{\rho})$ and Lemma \ref{lemma:greens-function-log-bound}, we have
		\begin{equation}
		\Gr{{\eps}}{\Lambdateps}(a,a) \leq \Gr{{\eps}}{B_{\rho}}(a,a) \leq {\eps}^{-1/2}, \quad \forall a \in {\VGeps}(B_{\rho})
		\end{equation} 
		except on an event of probability going to zero as ${\eps} \to 0$. 
		Therefore, it suffices to show that there exists a deterministic constant $C > 0$ independent of ${\eps}$ such that
		\begin{equation} \label{eq:lower-bound-exit}
		\sum_{b \in \Lambdateps} \Gr{{\eps}}{\Lambdateps}(b,a) \geq C \delta {\eps}^{-1}, \quad \forall 
		a \in \Lambdateps:\dist(\eta(a), \partial \eta(\Lambdateps)) > \delta 
		\end{equation}
		except on an event of probability going to zero as ${\eps} \to 0$. Indeed, we just need the desired estimate for ${\eps}$ small
		and we may absorb the factor of $C \delta$ into ${\eps}^{1/2}$ for ${\eps}$ small enough, depending on $C$ and $\delta$. 
		
		Let $a \in \Lambdateps$ such that $\dist(\eta(a), \partial \eta(\Lambdateps)) > \delta$ be given 
		and observe that
		\[
		\sum_{b \in \Lambdateps}  \Gr{{\eps}}{\Lambdateps}(b, a) \geq 
		\sum_{b \in {\VGeps}(B_{\delta/2}(a))} \Gr{{\eps}}{B_{\delta/2}(a)}(b,a).
		\]
		By Lemma \ref{lemma:greens-kernel-lower-bound} and Lemma \ref{lemma:mu_h-convergence}
		\[
		\sum_{b \in {\VGeps}(B_{\delta/2}(a))}  \Gr{{\eps}}{B_{\delta/2}(a)}(b,a)   \geq C \delta {\eps}^{-1}, 
		\]
		for some deterministic constant $C = C(\rho, \gamma)$ except on an event of probability going to zero as ${\eps} \to 0$. 
	\end{proof}
	
	We next prove the approximate mean-value inequality using the basic properties of the discrete obstacle problem
	and the convergence afforded by Theorem \ref{theorem:convergence-divisible}. 
	\begin{proof}[Proof of Lemma \ref{lemma:mean-value-ineq}]
		We restrict to the event that $\{t < \TT\}$ for $t > 0$. 
		As the function $\deg^{{\eps}}$ is reversible for simple random walk on ${\VGeps}$, by, for example, \cite[Exercise 2.1]{lyons-peres}, 
		\[
		\Gr{{\eps}}{\Lambdateps}(a,b)/\deg^{{\eps}}(b) = 	\Gr{{\eps}}{\Lambdateps}(b,a)/\deg^{{\eps}}(a), \quad \forall a,b \in \Lambdateps.
		\]
		Therefore, we can rewrite the desired inequality as 
		\[
		(t \eps^{-1}) \Gr{{\eps}}{\Lambdateps}(a, 0) \frac{\deg^{{\eps}}(a)}{\deg^{{\eps}}(0)} \geq (1 + \mathfrak{s}) \sum_{b \in \Lambdateps} \Gr{{\eps}}{\Lambdateps}(a, b)  \frac{\deg^{{\eps}}(a)}{\deg^{{\eps}}(b)} .
		\]
		By the definitions in Section \ref{subsec:discrete-potential-definitions} and Remark \ref{remark:overload-notation-potential}, this shows that it suffices to establish 
		\begin{equation} \label{eq:mean-value-ineq-verify-changed}
		(t \eps^{-1})  \gr{{\eps}}{\Lambdateps}(a, 0) \geq (1 + \mathfrak{s}) \qeps{\Lambdateps}(a) , \quad \forall 
		a \in \Lambdateps:\dist(\eta(a), \partial \eta(\Lambdateps)) > \delta,
		\end{equation}
		for some finite $\sigma(h,\eta)$-measurable random variable $\mathfrak{s} = \mathfrak{s}(\rho, \gamma, \delta)> 0$ not depending on ${\eps}$,
		except on an event of probability going to zero as ${\eps} \to 0$. 
		Define 
		\begin{equation} \label{eq:odometer-tilde}
		\begin{aligned}
		v_t^{{\eps}}(a) &= \wteps(a) + (t \eps^{-1})  \gr{{\eps}}{B_{\rho}}(a, 0) \\
		\tilde v_t^{{\eps}}(a) &= (t \eps^{-1})  \gr{{\eps}}{\Lambdateps}(a, 0) - \qeps{\Lambdateps}(a) 
		\end{aligned}
		\end{equation} 
		and note that on $\Lambdateps$
		\begin{align*}
		&\Delta^{{\eps}} (v_t^{{\eps}} - \tilde v_t^{{\eps}}) \\
		&= (\deg^{{\eps}})^{-1} - \deg^{{\eps}}(0)^{-1} (t \eps^{-1})  1\{\cdot = 0\} - \Delta^{{\eps}} \tilde v_t^{{\eps}} 
		\quad \mbox{(by \eqref{eq:discrete-normalized-green-laplacian} and Lemma \ref{lemma:discrete-laplacian-bound})} \\
		&= (\deg^{{\eps}})^{-1} - \deg^{{\eps}}(0)^{-1} (t \eps^{-1})  1\{\cdot = 0\} -  (-\deg^{{\eps}}(0)^{-1}  (t \eps^{-1})  1\{\cdot = 0 \}
		+ (\deg^{{\eps}})^{-1}) \\
		&\quad \mbox{(by \eqref{eq:discrete-normalized-green-laplacian} and \eqref{eq:discrete-normalized-exit-laplacian})} \\
		&= 0 \, . 
		\end{align*}
		Since $\wteps = \gr{{\eps}}{B_{\rho}}(\cdot, 0) = \qeps{\Lambdateps} = 0$ on $\partial \Lambdateps$,
		this implies that 
		\[
		\begin{cases}
		\Delta^{{\eps}} (v_t^{{\eps}} - \tilde v_t^{{\eps}})= 0 \quad \mbox{on $\Lambdateps$} \\
		(v_t^{{\eps}} - \tilde v_t^{{\eps}}) = 0 \quad \mbox{on $\partial \Lambdateps$}.
		\end{cases}
		\]
		By the maximum principle, Lemma \ref{lemma:maximum-principle}, this implies that
		\[
		\tilde v_t^{{\eps}} \geq v_t^{{\eps}} \quad \mbox{on $\Lambdateps$}.
		\]
		By uniform convergence of $\overline v^{{\eps}}_t$,  (scaling defined in \eqref{eq:discrete-scaling})
		which comes from Theorem \ref{theorem:convergence-divisible} and Lemma \ref{lemma:uniform-convergence-of-greens-function}, for some finite
		$\sigma(h,\eta)$-measurable random variable $C = C(\delta) > 0$, we have that
		\[
		\tilde v_t^{{\eps}}(a) \geq	v_t^{{\eps}}(a) > C {\eps}^{-1} , \quad \forall  a \in \Lambdateps:\dist(\eta(a), \partial \eta(\Lambdateps)) > \delta,
		\]
		except on an event of probability going to zero as ${\eps} \to 0$. 
		Also, by Lemma \ref{lemma:exit-time-upper-bound} 
		\[
		\sup_{a \in \Lambdateps}  \qeps{\Lambdateps}(a) \leq C_1 {\eps}^{-1}, 
		\]
		with probability approaching one as ${\eps} \to 0$ for some finite $\sigma(h,\eta)$-measurable random variable $C_1>0$ not depending on ${\eps}$. By combining the previous two indented equations with the definition~\eqref{eq:odometer-tilde} of $\tilde v_t^{{\eps}}$, 
		we have \eqref{eq:mean-value-ineq-verify-changed} 
		with $\mathfrak{s} := (C/C_1)$. 
	\end{proof}
	
	We conclude by including, for completeness, the arguments of \cite{lawler1992internal}.

	\begin{proof}[Proof of Lemma \ref{lemma:general-idla-lower-bound}] 
		Recall that for each ${\eps} > 0$, the IDLA cluster is formed by running independent random walks
		$X^{0, {\eps}}(i) := X^{{\eps}}(i)$ started at the origin on ${\Geps}$ and stopped upon exiting the cluster. 
		We let each independent random walk evolve forever even after it has left the occupied cluster
		and introduce the stopping times:
		\begin{align*}
		\sigma^i &= \min \{ t: X^{{\eps}}_t(i) \not \in A_{{\eps}}(i -1) \} \\
		&= \mbox{the time it takes the $i$th walker to leave the occupied cluster} \\
		\tau_b^i &= \min \{ t: X^{{\eps}}_t(i) = b \}  \\
		&= \mbox{the time it takes the $i$th walker to hit site $b$} \\
		T^i &= \min\{ t : X^{{\eps}}_t(i) \not \in O_2^{{\eps}} \} \\ 
		&= \mbox{the time it takes the $i$th walker to leave $O_2^{{\eps}}$}.
		\end{align*}
		
		For $a\in {\VGeps}$, let $E_a({\eps})$ denote the event that $a$ does not belong to the cluster $A_{{\eps}}(K_{\eps})$.
		We claim that to prove \eqref{eq:lower-bound-arbitrary}, it suffices to show that
		with probability approaching one as ${\eps} \to 0$, 
		\begin{equation} \label{eq:exponential-bound}
		\P[E_a({\eps}) | h, \eta] \leq \exp(-c {\eps}^{-c}), \quad \forall a \in O_1^{{\eps}},
		\end{equation}
		for some deterministic $c > 0$ which is independent of ${\eps}$ and $a$. 
		Indeed, since the cells of ${\VGeps}$ have $\mu_h$-mass ${\eps}$ and intersect only along their boundaries there exists a random $C  = C(\rho, \gamma) > 0$ so that
		\[
		|{\VGeps}(B_{\rho})| \leq C {\eps}^{-1},
		\]
		with polynomially high probability as ${\eps} \to 0$.  Therefore, \eqref{eq:lower-bound-arbitrary} can be obtained from \eqref{eq:exponential-bound} and a union bound over all $a \in O_1^{{\eps}}$.
		
		Our aim now is to show \eqref{eq:exponential-bound}. Fix an ${\eps}$ for which \eqref{eq:mean-value-ineq} holds and a $b \in O_2^{{\eps}}$. 
		Consider the random variables 
		\begin{align*}
		M &= \sum_{i \leq K_{\eps}}  1\{ \tau_b^i \leq T^i \}  \\
		&= \mbox{number of walks that visit $b$ before leaving $O_2^{{\eps}}$} \\
		L &= \sum_{i \leq K_{\eps}} 1\{ \sigma^i \leq \tau_b^i < T^i  \} \\
		&= \mbox{number of walks that visit $b$ before leaving $O_2^{{\eps}}$ but after exiting the current cluster}  .
		\end{align*}
		Observe that for all $k \geq 0$, 
		\begin{equation} \label{eq:upper-bound}
		\begin{aligned}
		\P[E_b({\eps}) | h , \eta]
		&\leq \P[M = L | h, \eta] \\
		&\leq \P[\{M \leq k \}\cup \{L \geq k\} | h , \eta]  \\
		&\leq \P[M \leq k| h, \eta] + \P[L \geq k | h , \eta].
		\end{aligned}
		\end{equation} 
		We will choose $k$ below to minimize this bound.  First note that by, for example, \cite[Lemma 4.6.1]{lawler-limic-walks}
		\begin{equation} \label{eq:markov-chain-renewal}
		\P^a[\tau_b < T | h ,\eta] =  \frac{\Gr{{\eps}}{O_2^{{\eps}}}(a, b)}{\Gr{{\eps}}{O_2^{{\eps}}}(b, b)}	
		\end{equation}
		for all $a \in {\VGeps}$ where $\P^a[\tau_b < T | h ,\eta]$ denotes the probability
		that an independent walker started at $a \in {\VGeps}$ hits site $b$ before it leaves  $O^{{\eps}}_2$.

		As $M$ is a sum of i.i.d.\ random variables, by \eqref{eq:markov-chain-renewal}
		\begin{equation} \label{eq:expected-value-of-m}
		\E[M | h, \eta] = K_{\eps} \P[ \tau_b^{1} \leq T^{1} | h , \eta] = K_{\eps} \times \frac{\Gr{{\eps}}{O_2^{{\eps}}}(0, b)}{\Gr{{\eps}}{O_2^{{\eps}}}(b, b)}.
		\end{equation}
		The random variable $L$ is not a sum of independent random variables, but each walk which contributes to the sum in $L$ can be mapped to the point at which it exits the current cluster. Therefore, 
		by the strong Markov property, 
		\begin{equation} \label{eq:strong-markov-reduction}
		L \leq \hat{L} := \sum_{a \in O^{{\eps}}_2} 1^a\{\tau_b < T\}
		\end{equation}
		where $1^a\{\tau_b < T\}$ denotes the event that an independent walker started at $a \in {\VGeps}$ hits site $b$ before it leaves $O^{{\eps}}_2$.
		Hence, by \eqref{eq:markov-chain-renewal} we have that 
		\begin{equation} \label{eq:expected-value-of-l}
		\E[\hat{L} | h , \eta] = \sum_{a \in O^{\eps}_2} \P^a[\tau_b < T | h, \eta] =  \frac{ \sum_{a \in O^{{\eps}}_2} \Gr{{\eps}}{O_2^{{\eps}}}(a, b)}{ {\Gr{{\eps}}{O_2^{{\eps}}}(b, b)}}.
		\end{equation}

		By the assumption \eqref{eq:mean-value-ineq}, we have that \eqref{eq:expected-value-of-m} and  \eqref{eq:expected-value-of-l} imply that
		there exists a deterministic $q > 0$ and $\sigma(h,\eta)$-measurable finite random variable $\mathfrak{s} > 0$ so that with probability approaching one as ${\eps} \to 0$, 
		\begin{equation} \label{eq:m-bounded-below-by-l}
		\E[M | h ,\eta] \geq (1 + \mathfrak{s}) \E[\hat{L} | h , \eta] \geq (1 + \mathfrak{s}) {\eps}^{-q},
		\end{equation}
		with the latter inequality following by the assumption \eqref{eq:L-lower-bound}.
		
		Recall that if $S$ is a finite sum of independent indicators with mean $\mu$, then 
		\begin{equation} \label{eq:concentration-inequality}
		\P\left[ |S - \mu| \geq \mu^{3/4} \right] \leq 2 \exp\left( -\frac{1}{4} \mu^{1/2} \right)
		\end{equation} 
		for all sufficiently large $\mu$ (see, \eg, \cite[Lemma 3.3]{lawler1992internal}).  Now, set $k = (1 + \mathfrak{s}/4) \E[\hat{L} |h,\eta]$ in \eqref{eq:upper-bound}
		and compute
		\begin{align*}
		&\P\left[L \geq (1 + \mathfrak{s}/4) \E[\hat{L} | h , \eta] | h , \eta \right] \\
		&\leq \P \left[ \hat{L} \geq (1 + \mathfrak{s}/4) \E[\hat{L} | h ,\eta] | h , \eta \right] \quad \mbox{(by \eqref{eq:strong-markov-reduction})} \\
		&\leq \P \left[ \hat{L} \geq \E[\hat{L} | h ,\eta] + \E[\hat{L} | h ,\eta]^{3/4} | h , \eta \right] 
		\quad\mbox{(by \eqref{eq:m-bounded-below-by-l} for all ${\eps}$ small depending on $\mathfrak{s}$)} \\
		&\leq 2 \exp \left(-\frac{1}{4} \E[\hat{L} | h , \eta]^{1/2}\right)  \quad \mbox{(by \eqref{eq:concentration-inequality} with $\mu = \E[\hat{L} | h, \eta])$} \\
		&\leq 2 \exp\left(-\frac{1}{4} {\eps}^{-q/2} \right) \quad \mbox{(by \eqref{eq:m-bounded-below-by-l})}.
		\end{align*}
		Similarly, 
		\begin{align*}
		&\P \left[ M \leq (1 + \mathfrak{s}/4) \E[\hat{L} | h , \eta] | h , \eta \right]  \\
		&\leq  \P \left[ M \leq \E[M | h , \eta] - \E[M | h , \eta]^{3/4} | h , \eta \right]
		\quad\mbox{(by \eqref{eq:m-bounded-below-by-l} for all ${\eps}$ small depending on $\mathfrak{s}$)} \\
		&\leq 2 \exp\left(-\frac{1}{4} {\eps}^{-q/2} \right)  \quad \mbox{(by \eqref{eq:m-bounded-below-by-l} and \eqref{eq:concentration-inequality} with $\mu = \E[M | h, \eta])$}. 
		\end{align*}
		Therefore, if we set $k = (1 + \mathfrak{s}/4) \E[\hat{L} | h, \eta]$ and use \eqref{eq:upper-bound}, we have \eqref{eq:exponential-bound}, completing the proof.
	\end{proof}

	\section{IDLA upper bound}	\label{sec:idla-upper-bound}
	As in the previous section we fix a $\rho \in (0,1)$ so that $B_{\rho} \Subset B_1$. 
	Recall the notation for the subset of $\C$ corresponding to a subset of ${\VGeps}$, \eqref{eq:vertex-scaling}.
	Here we prove that the IDLA cluster is contained in a harmonic ball.  
	\begin{prop} \label{prop:IDLA-upper-bound}
		Recall the time $\TT$ from~\eqref{eq:cluster-stopping-time}.
		For each $\delta \in (0, \rho)$ and $t > 0$, on the event $\{t < \TT\}$, 
		it holds except on an event of probability tending to 0 as ${\eps} \to 0$  that	
		\[
		\overline{A}_{{\eps}}(\lfloor t {\eps}^{-1} \rfloor) \subset \Bdeltap(\Lambda_t),
		\]
		where $\Lambda_t$ is as in Theorem \ref{theorem:harmonic-balls}.
	\end{prop}
	Our proof of the upper bound combines ideas of Lawler \cite{lawler1995subdiffusive} and Duminil-Copin-Lucas-Yadin-Yehudayoff \cite{duminil2013containing}. 
	The idea is utilize the diffusive, `smoothing' nature of IDLA. That is, it is extremely unlikely that a random walk travels a
	long distance without hitting many sites while doing so. Since the IDLA cluster is built by running independent random walks, 
	the cluster itself has this same property. 
	
	In our setting, we would like to use this smoothing to assert the following: given the current cluster $\Theta \subset {\VGeps}(B_{\rho})$, after adding $\delta {\eps}^{-1}$ walkers to the origin, the resulting cluster must be contained in ${\VGeps}(\Bdeltapp{C \delta^c}(\overline{\Theta}))$, (for some constants $c, C$) with extremely high probability. 
	
	Duminil-Copin-Lucas-Yadin-Yehudayoff \cite{duminil2013containing} showed that IDLA does indeed have this property if the underlying graph 
	is regular enough. Roughly, they required the graph to obey volume doubling and have an intrinsic metric which can be controlled by the Euclidean metric. 
	Unfortunately, these required properties do not hold on the mated-CRT map. In particular, with positive probability, it is possible for random walkers
	to travel a long Euclidean distance in the mated-CRT map without hitting many cells. Fortunately, this is not the case
	for `most' of the mated-CRT map. 
	
	We formalize `most' by considering the behavior of random walk on a Euclidean annulus intersected with the mated-CRT map (under the usual SLE/LQG embedding). 
	Roughly, for a `good' annulus, for all sufficiently small ${\eps}$, it is difficult for a random walk on ${\VGeps}$ 
	to travel across the annulus without hitting many cells along the way. 
	
	We start in Section \ref{subsec:good-annuli} by defining what it means for an annulus to be good. 
	We then use the definition of good annulus in Section \ref{subsec:comparison-estimates} to provide 
	estimates for the hitting measure of random walk on the mated-CRT map. We then use these estimates 
	together with an iterative argument to prove the upper bound in Section \ref{subsec:iteration}. 
	
	Below we frequently make use of the Abelian nature of IDLA. This Abelian property was first observed by Diaconis-Fulton in \cite{diaconis1991growth}
	and states that the distribution of the final IDLA cluster is unaffected by the order in which walkers are sent out. In particular, 
	we may use this to `start' and `stop' walkers as they go through certain sets.  To that end, we introduce the following definitions. 
	For continuity of literature, we use the same notation as \cite{duminil2013containing}. 
	\begin{definition} \label{def:stopped-idla}
		Given sets $\Theta \subset \Theta' \subset {\VGeps}$, and vertices $a_1, \ldots, a_k \in \Theta$, 
		denote by 
		\[
		A(\Theta; a_1, \ldots, a_k \to \Theta') \subset {\VGeps} 
		\]
		the IDLA aggregate with the following initial condition: the set $\Theta$ is completely occupied, $k$ independent simple random walks are started at $a_1, \ldots, a_k$, and the walkers are stopped upon exiting $\Theta'$. Let
		\[
		P(\Theta; a_1, \ldots, a_k \to \Theta') \subset \partial \Theta'
		\] 
		be the final positions of those walkers among $a_1, \ldots, a_k$, which are stopped before being absorbed into the aggregate. 
		We write
		\[
		A(\Theta; a_1, \ldots, a_k) := A(\Theta; a_1, \ldots, a_k \to {\VGeps}),
		\]
		when the walkers are not stopped, \ie, $\Theta' = {\VGeps}$. 
	\end{definition}
	In this notation, the Abelian property of IDLA is the following. 
	\begin{lemma}[Section 4 \cite{diaconis1991growth}] \label{lemma:abelian-IDLA}
		For every $\sigma(h,\eta)$-measurable choice of $\Theta \subset \Theta' \subset {\VGeps}$ and $a_1, \ldots, a_k \subset \Theta$, the conditional laws given $(h,\eta)$ of
		\[
		\begin{aligned}
		A(\Theta; a_1, \ldots, a_k) \quad \text{and} \quad
		A( A(\Theta; a_1, \ldots, a_k \to \Theta'); P(\Theta; a_1, \ldots, a_k \to \Theta'))
		\end{aligned}
		\]
		agree. That is, the IDLA aggregate has the same (conditional) distribution as any restarted, stopped IDLA aggregate.
	\end{lemma}

	\subsection{Good annuli} \label{subsec:good-annuli}
	Let $\boldsymbol c_0 =\boldsymbol c_0 (\rho, \gamma) \in (0,1/5)$ be a deterministic constant chosen so that 
	\begin{equation} \label{eq:choice-of-constant-annulus}
	C_0^{-1} \log \left(\frac{1}{\boldsymbol c_0} \right) -1 \geq \frac{1}{2}. 
	\end{equation}
	where $C_0 = C_0(\rho, \gamma)$ is the constant from Lemma \ref{lemma:greens-kernel-lower-bound}. This will be used in Lemma \ref{lemma:chance-of-being-absorbed} below. 
	
	For $z \in \C$, $r > 0$, and parameter $C >0$, let $E_r(z) := E_r(z; C)$ be the event that the following
	is true:
	\begin{equation} \label{eq:good-annulus-event}
	\sup_{v \in \A_{r, 2r}(z)} \int_{\A_{r,2r}(z)} \left(\log \frac{1}{|u-v|} + 1 \right) d \mu_h(u) \leq C \min_{w \in \A_{r, 2 r}(z)} \mu_h(B_{\boldsymbol{c_0} r}(w)).
	\end{equation}
	An annulus $\A_{r, 2 r}(z)$ for which $E_r(z)$ holds it said to be {\it good}. In this subsection we prove that there are many good annuli in $B_{\rho}$. 
	\begin{lemma} \label{lemma:good-prevalance}
		There exists a constant $C$ depending only on $\gamma$ such that the following holds with polynomially high 
		probability as $\delta \to 0$.
		For each $z \in ( B_{\rho - 3\sqrt{\delta}} \setminus B_{10 \sqrt{\delta}} ) \cap \frac{\delta}{100} \Z^2$ there is at least one radius $r \in [\delta,\delta^{1/2}] \cap \{7^{-n}\}_{n \in \N}$ for which $E_r(z)$ occurs, 
		where $E_r(z) = E_r(z;C)$ is as in \eqref{eq:good-annulus-event}. 
	\end{lemma}
	The argument is similar to \cite[Section 6.2]{bou2022harmonic}.  We start by showing that for each $z \in \C$ and $r > 0$, the event $E_r(z)$ occurs 
	with high probability provided $C > 0$ is chosen to be large (Lemma \ref{lemma:annulus-hit}).
	We will then use the near-independence of the GFF across disjoint concentric annuli (Lemma~\ref{lemma:annulus-iterate}) 
	to show that for each fixed $z \in \C$, it holds with very high probability when $\delta$ is small that there is at least one radius $r \in [\delta, \delta^{1/2}] \cap \{7^{-n}\}_{n \in \N}$ for which $E_r(z)$ occurs. 
	Finally, we will take a union bound over all $z \in B_{\rho - 3\sqrt{\delta}} \setminus B_{10 \sqrt{\delta}} \cap \frac{\delta}{100} \Z^2$. 
	
	Throughout this section, we use the fact, recalled in \eqref{eq:gff},  that in $B_1$ the $\gamma$-quantum cone field $h$ agrees in law with  
	the whole-plane GFF plus $-\gamma \log|\cdot|$ normalized so that its circle average over $B_1$ is 0. Also recall 
	the notation for annuli from \eqref{eq:annulus}.
	\begin{lemma} \label{lemma:annulus-hit}
		For each $p  \in (0,1)$, there exists $C  = C(p) > 0$ such that the event $E_r(z) := E_r(z; C)$ of \eqref{eq:good-annulus-event} satisfies 
		\[
		\begin{aligned}
		\P\left[ E_r(z) \right] \geq p,&\quad \forall r > 0, \\
		&\quad \forall z \in\C \mbox{ such that $\A_{r/3, 3r}(z) \subset B_{\rho}$ and $\dist(\A_{r/3, 3r}(z), \{0\}) \geq r/100$}.
		\end{aligned}
		\]
	\end{lemma}
	
	\begin{proof}
		Denote the left and right sides of the inequality in the event \eqref{eq:good-annulus-event}
		by $\mathcal{L}^h_r(z)$ and $\mathcal{R}^h_r(z)$ respectively.  We first show that it suffices to prove the lemma with $h$ replaced by $h^{\C}$, \ie, the GFF without a log-singularity ($\gamma = 0$ in \eqref{eq:gff}). We then prove the lemma for $h^{\C}$. 
		\medskip
		
		{\it Step 1: Reduction to $h^{\C}$.} \\
		Recall that $h$ restricted to $B_1$ agrees in law with $h^\C -\gamma \log |\cdot|$ where $h^\C$ is a whole-plane GFF. 
		So, we can couple $h^{\C}$ and $h$ so that $h|_{B_1} = h^\C|_{B_1} - \gamma\log|\cdot|$.
		Suppose the statement of the lemma holds for $h^{\C}$ with $C_1 > 0$ in place of $C$. 
		
		Fix $r > 0$ and then $z$ such that the annulus $\A_{r/3, 3 r}(z)$ lies at Euclidean distance at least $r/100$ from the origin. 
		Note that $E_r(z)$ depends only on $h|_{\A_{r/3,3r}(z)}$. Thus, the log function is bounded above and below in the region of dependence of $E_r(z)$. 
		Hence, by Weyl scaling (Fact~\ref{fact:lqg-measure}), there is a constant $A>0$ depending only on $\gamma$ such that 
		\begin{equation}
		\frac{\mathcal L^{h}_{r}(z)}{\mathcal{R}^h_r(z)} \leq A  \frac{\mathcal L^{h^\C}_{r}(z)}{\mathcal{R}^{h^\C}_r(z)}.
		\end{equation} 
		Thus, 
		\[
		\frac{\mathcal L^{h^\C}_{r}(z)}{\mathcal{R}^{h^\C}_r(z)}  \leq C_1 \implies \frac{\mathcal L^{h}_{r}(z)}{\mathcal{R}^h_r(z)} \leq C_1 A
		\]	
		and hence \eqref{eq:good-annulus-event} occurs for $C := A \times C_1$ if it occurs for $h^{\C}$.
		\medskip
		
		{\it Step 2: Whole-plane GFF.} \\
		We use the fact that the law of $  h^{\C}$ is both scale and translation invariant modulo additive constant \eqref{eq:gff-scaling}. By the Weyl scaling property of the measure $\mu_h$ (Fact~\ref{fact:lqg-measure}), the event $E_r(z)$ is a.s.\ determined by $h$ viewed modulo additive constant. From this and the LQG coordinate change formula for $\mu_h$, we infer that the probability of $E_r(z)$ does not depend on $r$ or $z$. Hence, it suffices to find $C > 0$ as in the lemma statement such that $\P[E_1(0)] \geq p$. This however, follows immediately from the fact that $\frac{\mathcal{L}^h_1(0)}{\mathcal{R}^h_1(0)}$ is a positive, finite random variable. 
	\end{proof}
	
	The following lemma is a special case of~\cite[Lemma 3.1]{gwynne2020local}. 
	\begin{lemma}[\cite{gwynne2020local}] \label{lemma:annulus-iterate}
		Fix $0 < s_1<s_2 < 1$. Let $\{r_k\}_{k\in\N}$ be a decreasing sequence of positive numbers such that $r_{k+1} / r_k \leq s_1$ for each $k\in\N$ and let $\{E_{r_k} \}_{k\in\N}$ be events such that $E_{r_k}$ is a.s.\ determined by $h |_{\A_{s_1 r_k , s_2 r_k}(0)  } $, viewed modulo additive constant, for each $k\in\N$. 
		For $K\in\N$, let $N(K)$ be the number of $k\in [1,K] \cap \Z$ for which $E_{r_k}$ occurs. 
		For each $\alpha > 0$ there exists $p = p(\alpha,s_1,s_2) \in (0,1)$ and $C = C(\alpha,s_1,s_2) > 0$ (independent of the particular choice of $\{r_k\}$ and $\{E_{r_k}\}$) such that if  
		\begin{equation} \label{eqn-annulus-iterate-prob}
		\P\left[ E_{r_k}  \right] \geq p , \quad \forall k\in\N  ,
		\end{equation} 
		then 
		\begin{equation} \label{eqn-annulus-iterate}
		\P\left[ N(K)  = 0 \right] \leq C e^{-\alpha K} ,\quad\forall K \in \N. 
		\end{equation}  
	\end{lemma}
	
	We now prove the desired claim.

	\begin{proof}[Proof of Lemma~\ref{lemma:good-prevalance}]
		The event $E_r(z)$ depends only on the measure $\mu_h|_{\A_{r/3,3r}(z)}$ 
		Moreover, multiplying this measure by a constant does not change whether $E_r(z)$ occurs. 
		Therefore, $E_r(z)$ is a.s.\ determined by $h|_{\A_{r/3,3r}(z)}$ viewed modulo additive constant.
		
		We now apply Lemma~\ref{lemma:annulus-iterate} with $K = \lfloor \log_7  \epsilon^{-1/2} \rfloor$, the radii $r_1,\dots,r_K \in [ \delta, \delta^{1/2}] \cap \{7^{-n}\}_{n\in\N}$, the events $E_{r_k} = E_{r_k}(z)$, and an appropriate universal constant choice of $\alpha$. We find that there exist universal constants $p\in (0,1)$ and $c>0$ such that if $\P[E_r(z)] \geq p$
		for each $r > 0$ and each $z \in B_{\rho} \setminus B_{10 r}$,  then for all $z \in B_{\rho} \setminus B_{10 \sqrt{\delta}}$, 
		\begin{equation} \label{eqn-annulus-hit-iterate}
		\P\left[ \text{$E_r(z)$ occurs at least once for $r \in [ \delta, \delta^{1/2}] \cap \{7^{-n}\}_{n\in\N}$} \right] \geq 1 - O_\delta(\delta^3) 
		\end{equation} 
		with a universal implicit constant in the $O_\delta(\cdot)$. 
		
		By Lemma~\ref{lemma:annulus-hit}, there exists $C = C(\gamma) > 0$ such that for this choice of $C$, one has $\P[E_r(z)] \geq p$ for each $r > 0$ and each $z \in B_{\rho} \setminus B_{10 r}$. Therefore, the estimate~\eqref{eqn-annulus-hit-iterate} holds for this choice of $C$. We then conclude by means of a union bound over all $z\in  (B_{\rho-3 \sqrt{\delta}} \setminus B_{10 \sqrt{\delta}}) \cap \frac{\delta}{100} \Z^2$. 
	\end{proof}

	\subsection{Harmonic measure estimates} \label{subsec:comparison-estimates}
	Equipped with the prevalence of good annuli, we are now ready to prove the following random walk estimate.
	Roughly, it states that with probability bounded away from zero, random walk cannot travel through an annulus of constant size without hitting at least a mesoscopic number of vertices.	See Figure \ref{fig:harmonic-measure-exit} for an illustration. 
	
	Much stronger versions of this estimate have been proved on $\Z^d$ \cite[Lemma 11]{lawler1995subdiffusive} and the supercritical percolation cluster \cite[Lemma 5]{duminil2013containing}.	

	\begin{lemma} \label{lemma:constant-order-lower-bound}
		There exists deterministic $\omega = \omega(\gamma) > 0$ so that for each $s_0 \in (0,1)$, there exists $s_1, s_2 \in (0,1)$ 
		and a random  $\sigma(h)$-measurable $\boldsymbol{\delta}>0$ such that with probability approaching one as ${\eps}$ goes to zero, the following is true for all $r \in [10 {\eps}^{\omega}, \boldsymbol{\delta}]$ and all sets $\Theta \subset {\VGeps}$ such that $\overline{\Theta} \Subset B_{\rho/2}$ and $B_{s_0} \subset \overline{\Theta}$ with $\overline{\Theta}$ as in \eqref{eq:vertex-scaling}.
		
		Let $a \in \cl(\Theta)$ and let $S \subset {\VGeps}(\overline{\Theta} + B_r)$ be such that $|S \setminus \Theta| < s_1 r^{2 \beta^-} {\eps}^{-1}$
		where $\beta^-$ is as in Lemma \ref{lemma:volume-growth}.
		Let $\Xi$ denote the trace of a random walk started at $a$ and stopped upon exiting 
		${\VGeps}(\overline{\Theta} + B_r)$. Then, 
		\[
		\P[\Xi \cap {\VGeps}(\overline{\Theta} + B_r) \setminus (S \cup \Theta) \neq \emptyset | h, \eta] \geq s_2.
		\]
	\end{lemma}
	A law of large numbers version of this estimate is Lemma \ref{lemma:chance-of-being-absorbed} below and this will be used to show that the IDLA cluster
	cannot develop long tentacles in Lemma \ref{lemma:iterative-bound}. 
	
	The set $\Theta$ in Lemma \ref{lemma:constant-order-lower-bound} may be thought of as the current IDLA cluster
	and $S$ as the cluster after releasing an additional, small number of walkers (fewer than $s_1 r^{2 \beta^-} {\eps}^{-1}$). 
	The lemma states that with (conditional) probability bounded from below by a constant, $s_2$, which is independent 
	of $r$ and ${\eps}$, a random walk started at an arbitrary vertex in $\cl(\Theta)$ 
	will be absorbed into the intermediate cluster (\ie, hit a vertex not in $S \cup \Theta$) before exiting the $B_r$ neighborhood
	of the initial cluster, ${\VGeps}(\overline{\Theta} + B_r)$. In the case when $\Theta$ is a large IDLA cluster, the condition $B_{s_0} \subset \overline{\Theta}$
	is satisfied for some small $s_0 \in (0,\rho)$ for all small ${\eps}$ due the lower bound, Proposition \ref{prop:IDLA-lower-bound} 
	and the corresponding property for harmonic balls, Theorem \ref{theorem:harmonic-balls}.

	To prove the divisible sandpile upper bound in Section \ref{sec:div-upper-bound}, we require the following stronger version of Lemma \ref{lemma:constant-order-lower-bound}. 
	\begin{lemma} \label{lemma:constant-order-lower-bound-div}
		There exists deterministic $\omega = \omega(\gamma) > 0$ so that for each $s_0 \in (0,1)$, there exists $s_1, s_2 \in (0,1)$ 
		and a random  $\sigma(h)$-measurable $\boldsymbol{\delta}>0$ such that with probability approaching one as ${\eps}$ goes to zero, the following is true for all $r \in [10 {\eps}^{\omega}, \boldsymbol{\delta}]$ and all sets $\Theta \subset {\VGeps}$ such that $\overline{\Theta} \Subset B_{\rho/2}$ and $B_{s_0} \subset \overline{\Theta}$ with $\overline{\Theta}$ as in \eqref{eq:vertex-scaling}.  
		
		For all $A \subset {\VGeps}(\overline{\Theta} + B_{r}) \setminus \Theta$
		such that, for all $z \in \C$,  $|A \cap {\VGeps}(B_r(z))| < s_1 r^{2 \beta^-} {\eps}^{-1}$, where $\beta^-$ is as in Lemma \ref{lemma:volume-growth}, we have
		\[
		\sup_{a \in {\VGeps}(\partial (\overline{\Theta} + B_{r/2}))} \P[\mbox{$X^{a, {\eps}}$ exits $ {\VGeps}(\overline{\Theta} + B_{r}) \setminus \Theta $ through $A$} | h , \eta] \leq  1 - s_2 ,
		\]
		where $X^{a, {\eps}}$ is simple random walk started at $a \in {\VGeps}$. 
	\end{lemma}
	\begin{proof}[Proof of Lemma \ref{lemma:constant-order-lower-bound} assuming Lemma \ref{lemma:constant-order-lower-bound-div}]	
		By the strong Markov property and Lemma \ref{lemma:cell-size-estimate}, the trace of a random walk started at $a \in \cl(\Theta)$ 
		contains the trace of a random walk started at some site $b \in {\VGeps}(\partial (\overline{\Theta} + B_{r/2}))$ with probability approaching one as ${\eps} \to 0$. Moreover, if $S \subset {\VGeps}(\overline{\Theta} + B_r)$ satisfies $|S \setminus \Theta| < s_1 r^{2 \beta^-} {\eps}^{-1}$
		then $A := S   \setminus \Theta$ satisfies  $|A| < s_1 r^{2 \beta^-} {\eps}^{-1}$.
		Therefore, for all $z \in \C$,  $|A \cap {\VGeps}(B_r(z))| < s_1 r^{2 \beta^-} {\eps}^{-1}$. 
		
		The previous three sentences imply, by Lemma \ref{lemma:constant-order-lower-bound-div} (under the conditions of the lemma)
		that 
		\[
		\P[\Xi \cap {\VGeps}(\overline{\Theta} + B_r) \setminus (S \cup \Theta) = \emptyset | h, \eta] \leq  1-s_2 ,
		\]
		which implies Lemma \ref{lemma:constant-order-lower-bound}. 
	\end{proof}

	\begin{figure}
		\includegraphics[width=0.5\textwidth]{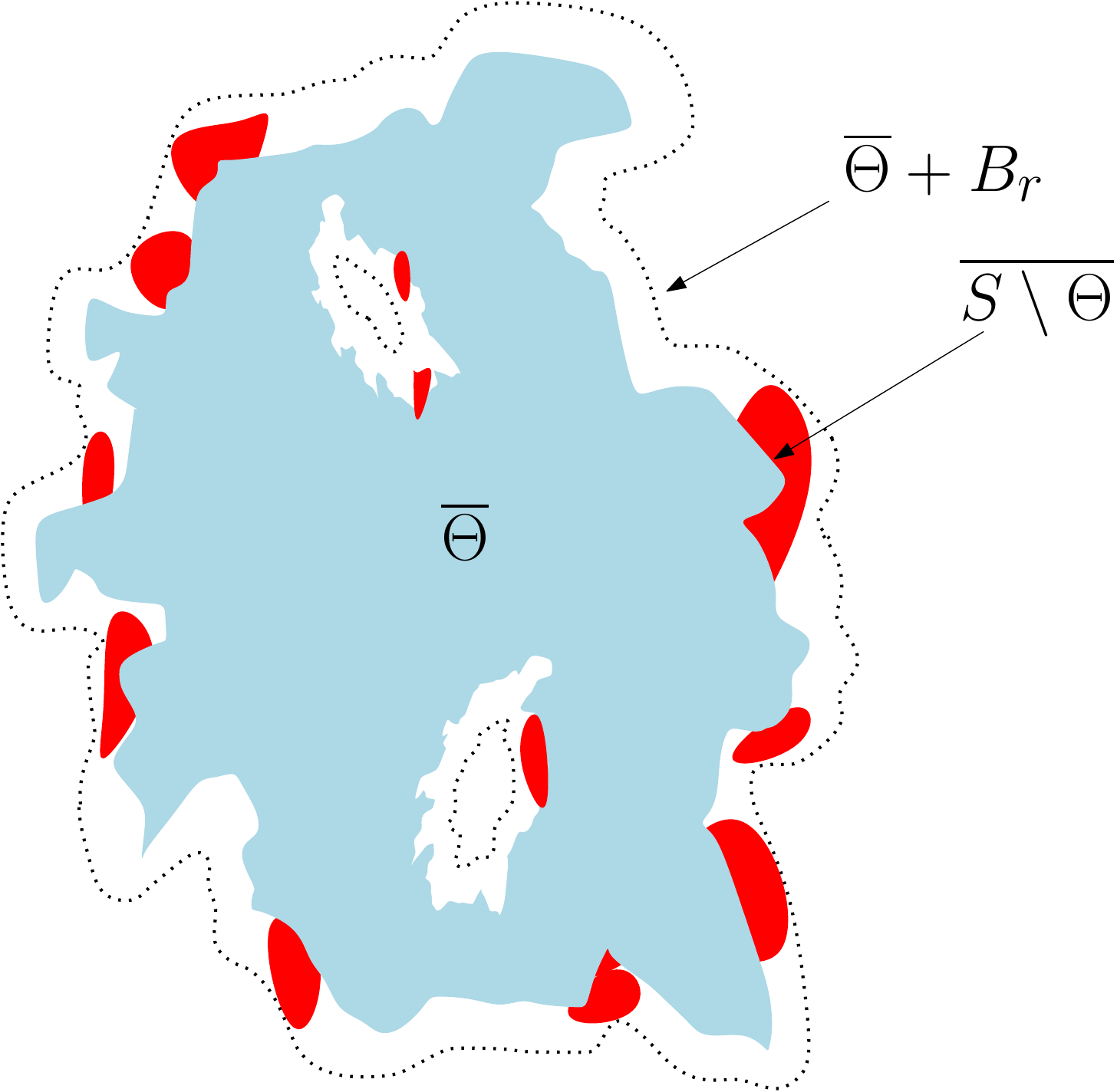}
		\caption{An illustration of the statement in Lemma \ref{lemma:constant-order-lower-bound}. The filled in cells corresponding to vertices in the cluster $\overline{\Theta}$ are shown in light blue and a neighborhood, $\overline{\Theta} + B_r$ is outlined by a dashed black line.  The set $\overline{S \setminus \Theta}$ is shown in red. Random walk started at a vertex in $\cl(\Theta)$
			is unlikely to exit ${\VGeps}(\overline{\Theta} + B_r)$ via paths which stay in $S$.}
		\label{fig:harmonic-measure-exit}
	\end{figure}

	\begin{proof}[Proof of Lemma \ref{lemma:constant-order-lower-bound-div}] 
		Let $\omega = \omega(\gamma) > 0$ be as in Lemma \ref{lemma:greens-kernel-lower-bound}.
		We may assume, by possibly taking $\omega$ smaller, that $2 \beta^{-} \omega \leq \alpha$, where $\beta^-$ and $\alpha$ are as in 
		Lemmas \ref{lemma:volume-growth} and \ref{lemma:mu_h-convergence}. This will be used in \eqref{eq:use-of-assumption-on-omega} below. 
		
		Let $E_{r}(z) = E_r(z; C)$
		and $C > 0$ be the event and parameter from Lemma \ref{lemma:good-prevalance}.
		By Lemma \ref{lemma:good-prevalance}, it holds with polynomially high probability as $\delta \to 0$ that for each $z \in (B_{\rho - 3 \sqrt{\delta}} \setminus B_{10 \sqrt{\delta}}) \cap \frac{\delta}{100} \Z^2$, there exists $\mathbf{r}_z = \mathbf{r}_z(\delta) \in [\delta,\delta^{1/2}]$ such that $E_{\mathbf{r}_z}(z)$ occurs. 
		By Lemma~\ref{lemma:volume-growth}, it also holds with polynomially high probability as $\delta \to 0$ that 
		\begin{equation} \label{eq:volume-growth-lower-bound}
		\mu_h\left( B_{\delta}(x) \right)  \geq \delta^{\beta^-}, \quad \forall x \in B_{1 - \delta}.
		\end{equation}
		By the Borel-Cantelli lemma, a.s.\ there exists a random $\boldsymbol{\delta}' > 0$ such that the preceding two sentences hold for all dyadic $\delta < \boldsymbol{\delta}'$. 
		We restrict to the event that the statements in Lemmas \ref{lemma:mu_h-convergence}, \ref{lemma:greens-kernel-lower-bound}, and \ref{lemma:exit-time-upper-bound} all hold, which happens with probability tending to 1 as ${\eps}\to 0$.  (For Lemma \ref{lemma:mu_h-convergence} we suppose that convergence holds for some countable family of polynomials which are dense in the uniform topology in $B_{\rho}$.) Our calculations below will require ${\eps}$ to be small in a deterministic way, so we fix a small deterministic $\boldsymbol{\pi}$ (which we will determine below) and an ${\eps} < \boldsymbol{\pi}$.  
		
		\medskip
		{\it Step 1: Reduction to random walk in a good annulus.}  \\
		We surround $\Theta$ by good annuli. By our assumptions on $\Theta$, by possibly taking $\boldsymbol{\pi}$ smaller, depending only on $s_0$, we can arrange that
		\begin{equation} \label{eq:boundary-inside}
		(\overline{\Theta} + B_{20 \sqrt{\delta }}) \setminus (\overline{\Theta} + B_{\sqrt{\delta }}) \Subset  B_\rho \setminus B_{10 \sqrt{\delta }}, \quad \forall \delta < \boldsymbol{\pi} \wedge \boldsymbol{\delta}'.
		\end{equation}
		By~\eqref{eq:boundary-inside} and the fact that $\mathbf{r}_z \in [\delta , \delta^{1/2}]$, we obtain that for each $\delta < \boldsymbol{\pi} \wedge \boldsymbol{\delta}'$ there exists $\mathbf{Z} = \mathbf{Z}(\delta) \subset (B_{\rho - \delta } \setminus B_{10 \sqrt{\delta }})  \cap \frac{\delta }{100} \Z^2$ for which
		\begin{equation} \label{eq:cover-continuous}
		(\overline{\Theta} + B_{4 \sqrt{\delta }}) \setminus (\overline{\Theta} + B_{3 \sqrt{\delta }}) \Subset \bigcup_{{z} \in \mathbf{Z}} B_{\mathbf{r}_{{z}}/2}({z}) \Subset B_{\rho- 10 \sqrt{\delta}}
		\end{equation} 
		and 
		\begin{equation} \label{eq:cover-in-complement}
		\bigcup_{{z} \in \mathbf{Z}} B_{3 \mathbf{r}_{{z}}}({z}) \Subset (\overline \Theta + B_{2 \sqrt{\delta}})^c.
		\end{equation}
		We now fix a dyadic $\delta \in (10 {\eps}^{\omega}, \boldsymbol{\pi} \wedge \boldsymbol{\delta}')$ 	
		and let $r \in [\delta, \sqrt{\delta}] < \rho/4$. We continue with this choice of $r$ as the lemma statement follows once we choose
		$\boldsymbol{\delta} = \sqrt{\boldsymbol{\pi} \wedge \boldsymbol{\delta}'}$.

		By \eqref{eq:cover-continuous} and the strong Markov property of simple random walk, it suffices to prove the lemma 
		with $a$ replaced by $b \in {\VGeps}(\partial B_{{(4/3)} \mathbf{r}_{z}}({z}))$ for ${z} \in \mathbf{Z}$. We in fact prove the following stronger statement: 
		a simple random walk $X^{b, {\eps}}_t$ started at such $b$ hits
		\[
		\mathfrak{E} := {\VGeps}(B_{\mathbf{r}_{z/4}}(\eta(b))) \setminus A
		\]
		before exiting ${\VGeps}(B_{\mathbf{r}_{z/4}}(\eta(b)))$ with conditional probability given $(h,\eta)$ bounded from below by a constant.
		See Figure \ref{fig:path-decomposition}.  
		\medskip
		
		{\it Step 2: Random walk in a good annulus.} \\
		Let $\xi_{\mathbf{r}_{z}}$ 
		be the first time that $X^{b, {\eps}}_t$ hits ${\VGeps}(\partial B_{\mathbf{r}_{z/4}}(\eta(b)))$ and let 
		\begin{equation} \label{eq:visit-count}
		V = \sum_{t=0}^{\xi_{\mathbf{r}_{z}}} 1 \{ X^{b, {\eps}}_t \in \mathfrak{E}\}  = \left( \text{number of visits of $\mathfrak{E}$ by $X^{b, {\eps}}$ before time $\xi_{\mathbf{r}_{z}}$} \right) .
		\end{equation}
		By the second moment method, 
		\begin{equation} \label{eq:second-moment-lower-bound}
		\P[V \geq 1 | h, \eta] \geq \frac{(\E[V | h, \eta])^2}{\E[V^2 | h , \eta]}.
		\end{equation}		
		In the next two steps we show that 
		\begin{equation} \label{eq:lower-bound-on-V}
		\E[ V | h ,\eta] \geq  {\eps}^{-1} C^{-1} \times \mu_h(B_{\boldsymbol{c_0} \mathbf{r}_{z}}(\eta(b)))
		\end{equation}
		and
		\begin{equation} \label{eq:upper-bound-on-V}
		\E[ V^2 | h ,\eta]   \leq  C {\eps}^{-2} \left(\sup_{v \in \A_{\mathbf{r}_{z}, 2\mathbf{r}_{z}}(z)} \int_{\A_{\mathbf{r}_{z},2 \mathbf{r}_{z}}(z)} \left( \log \frac{1}{|u-v|} + 1  \right) d \mu_h(u)\right)^2,
		\end{equation}	
		for some large deterministic constant $C = C(\rho, \gamma) > 0$. 
		Combining \eqref{eq:lower-bound-on-V} and \eqref{eq:upper-bound-on-V}  
		together with the definition of the event $E_r(z)$ in \eqref{eq:good-annulus-event}, leads to 
		a constant order lower bound in \eqref{eq:second-moment-lower-bound}. This completes the proof by the definition~\eqref{eq:visit-count} of $V$. 
		\medskip 
		
		{\it Step 3: Lower bound on the numerator.} \\ 
		We lower bound the numerator:
		\begin{align*}
		\E[ V | h ,\eta] &= \sum_{b' \in \mathfrak{E}} \Gr{{\eps}}{B_{\mathbf{r}_{z}/4}(\eta(b))}(b, b') \qquad \mbox{(definition of $V$ and $\Gr{{\eps}}{B_{\mathbf{r}_{z}}(\eta(b))}$)} \\
		&\geq \sum_{b' \in {\VGeps}(B_{\boldsymbol{c_0} \mathbf{r}_{z}}(\eta(b))) \cap \mathfrak{E}} \Gr{{\eps}}{B_{\mathbf{r}_{z}/4}(\eta(b))}(b, b') \quad \mbox{(where $\boldsymbol{c_0}$
			is as in \eqref{eq:choice-of-constant-annulus})} \\
		&\geq \left( \inf_{b' \in {\VGeps}(B_{\boldsymbol{c_0} \mathbf{r}_{z}}(\eta(b)))} \Gr{{\eps}}{B_{\mathbf{r}_{z}/4}(\eta(b))}(b, b') \right) (|{\VGeps}(B_{\boldsymbol{c_0} \mathbf{r}_{z}}(\eta(b))| - |A \cap {\VGeps}(B_{\mathbf{r}_{z}}(\eta(b))| ) \\	
		&\geq \left( \inf_{b' \in {\VGeps}(B_{\boldsymbol{c_0} \mathbf{r}_{z}}(\eta(b)))} \Gr{{\eps}}{B_{\mathbf{r}_{z}/4}(\eta(b))}(b, b') \right) {\eps}^{-1} (\mu_h(B_{\boldsymbol{c_0} \mathbf{r}_{z}}(\eta(b))) -  s_1 r^{2 \beta^-} 
		- {\eps}^{\alpha}) 		\stepcounter{equation}\tag{\theequation}\label{eq:lower-bound-on-V-intermediate} \\
		&\qquad \mbox{(by the assumption on $A$ and Lemma \ref{lemma:mu_h-convergence})}.
		\end{align*}
		We claim that for a sufficiently small, deterministic choice of $\boldsymbol{\pi} > 0$ and $s_1 > 0$, 
		\begin{equation}\label{eq:use-of-assumption-on-omega}
		\mu_h(B_{\boldsymbol{c_0} \mathbf{r}_{z}}(\eta(b))) -  s_1 r^{2 \beta^-} 
		- {\eps}^{\alpha} \geq  C'\mu_h(B_{\boldsymbol{c_0} \mathbf{r}_{z}}(\eta(b)))
		\end{equation} 
		for some deterministic constant $C' = C'(\rho, \gamma)$.
		Indeed, for a deterministic constant $c = c(\rho, \gamma)$ which will change from line to line we have
		\begin{align*}
		&\mu_h(B_{\boldsymbol{c_0} \mathbf{r}_{z}}(\eta(b))) -  s_1 r^{2 \beta^-} 
		- {\eps}^{\alpha} \\
		&= \frac{1}{2} \mu_h(B_{\boldsymbol{c_0} \mathbf{r}_{z}}(\eta(b)))  + \frac{1}{2} \mu_h(B_{\boldsymbol{c_0} \mathbf{r}_{z}}(\eta(b))) - 
		s_1 r^{2 \beta^-} 
		- {\eps}^{\alpha}  \\
		&\geq \frac{1}{2} \mu_h(B_{\boldsymbol{c_0} \mathbf{r}_{z}}(\eta(b)))  + c \times \mathbf{r}_{z}^{\beta^-} -
		s_1 r^{2 \beta^-} 
		- {\eps}^{\alpha}  \quad \mbox{(by Lemma \ref{lemma:volume-growth})} \\
		&\geq \frac{1}{2} \mu_h(B_{\boldsymbol{c_0} \mathbf{r}_{z}}(\eta(b)))  + c \times \delta^{\beta^-} -
		s_1 \times \delta^{\beta^-}
		- {\eps}^{\alpha}  \quad \mbox{(since $\mathbf{r}_{z} > \delta$ and $r < \sqrt{\delta}$)} \\
		&\geq \frac{1}{2} \mu_h(B_{\boldsymbol{c_0} \mathbf{r}_{z}}(\eta(b)))  + c \times \delta^{\beta^-} 
		- {\eps}^{\alpha}  \quad \mbox{(for small $s_1 < c$)}  \\
		&\geq \frac{1}{2} \mu_h(B_{\boldsymbol{c_0} \mathbf{r}_{z}}(\eta(b)))  + c \times {\eps}^{\omega \beta^-} 
		- {\eps}^{2 \beta^- \omega}  \quad \mbox{($\delta > {\eps}^{\omega}$ and $\alpha \geq 2 \beta^- \omega$ by assumption)}  \\
		&\geq \frac{1}{2} \mu_h(B_{\boldsymbol{c_0} \mathbf{r}_{z}}(\eta(b))) \quad \mbox{(for $\boldsymbol{\pi}$ small depending only on $c$).}
		\end{align*}
		Therefore, 
		\begin{align*}
		&\E[ V | h ,\eta] \\
		&\geq \left( \inf_{b' \in {\VGeps}(B_{\boldsymbol{c_0} \mathbf{r}_{z}}(\eta(b)))} \Gr{{\eps}}{B_{\mathbf{r}_{z}/4}(\eta(b))}(b, b') \right) {\eps}^{-1} (\mu_h(B_{\boldsymbol{c_0} \mathbf{r}_{z}}(\eta(b))) -  s_1 r^{2 \beta^-} 
		- {\eps}^{\alpha}) 
		\quad \mbox{(by \eqref{eq:lower-bound-on-V-intermediate})} \\
		&\geq \frac{1}{2} {\eps}^{-1} (\mu_h(B_{\boldsymbol{c_0} \mathbf{r}_{z}}(\eta(b))) -  s_1 r^{2 \beta^-} 
		- {\eps}^{\alpha})  
		\quad \mbox{(by \eqref{eq:choice-of-constant-annulus} and Lemma \ref{lemma:greens-kernel-lower-bound})} \\
		&\geq \frac{1}{2} {\eps}^{-1} \times  C'\mu_h(B_{\boldsymbol{c_0} \mathbf{r}_{z}}(\eta(b)))
		\quad \mbox{(by \eqref{eq:use-of-assumption-on-omega})},
		\end{align*}
		which implies \eqref{eq:lower-bound-on-V} with $C^{-1} := C'/2$.
		\medskip
		
		{\it Step 4: Upper bound on the denominator} \\ 
		We finally upper bound the denominator in~\eqref{eq:second-moment-lower-bound}; here the deterministic constant $C = C(\rho, \gamma) > 0$ also changes from line to-line: 
		\begin{align*}
		&\E[ V^2 | h ,\eta] \\
		&\leq \E[ \xi_{\mathbf{r}_{z}}^2 | h ,\eta] \\
		&\leq C {\eps}^{-2} \left(\sup_{v \in B_{\mathbf{r}_{z/4}}(\eta(b))} \int_{B_{\mathbf{r}_{z/4}+{\eps}^{\alpha}}(\eta(b))} \left( \log \frac{1}{|u-v|} + 1  \right) d \mu_h(u) \right)^2  \quad \mbox{(by Lemma \ref{lemma:exit-time-upper-bound})} \\
		&\leq C {\eps}^{-2} \left(\sup_{v \in \A_{\mathbf{r}_{z}, 2\mathbf{r}_{z}}(z)} \int_{\A_{\mathbf{r}_{z},2 \mathbf{r}_{z}}(z)} \left( \log \frac{1}{|u-v|} + 1  \right) d \mu_h(u) \right)^2 \\
		&\qquad \mbox{(since $r_{z} \geq \delta \geq 10 {\eps}^{\alpha}$ and $b \in {\VGeps}(\partial B_{(4/3) \mathbf{r}_z}(z)$)} .
		\end{align*}
		This shows \eqref{eq:upper-bound-on-V}. 	\end{proof}
	\begin{figure}
		\includegraphics[width=0.5\textwidth]{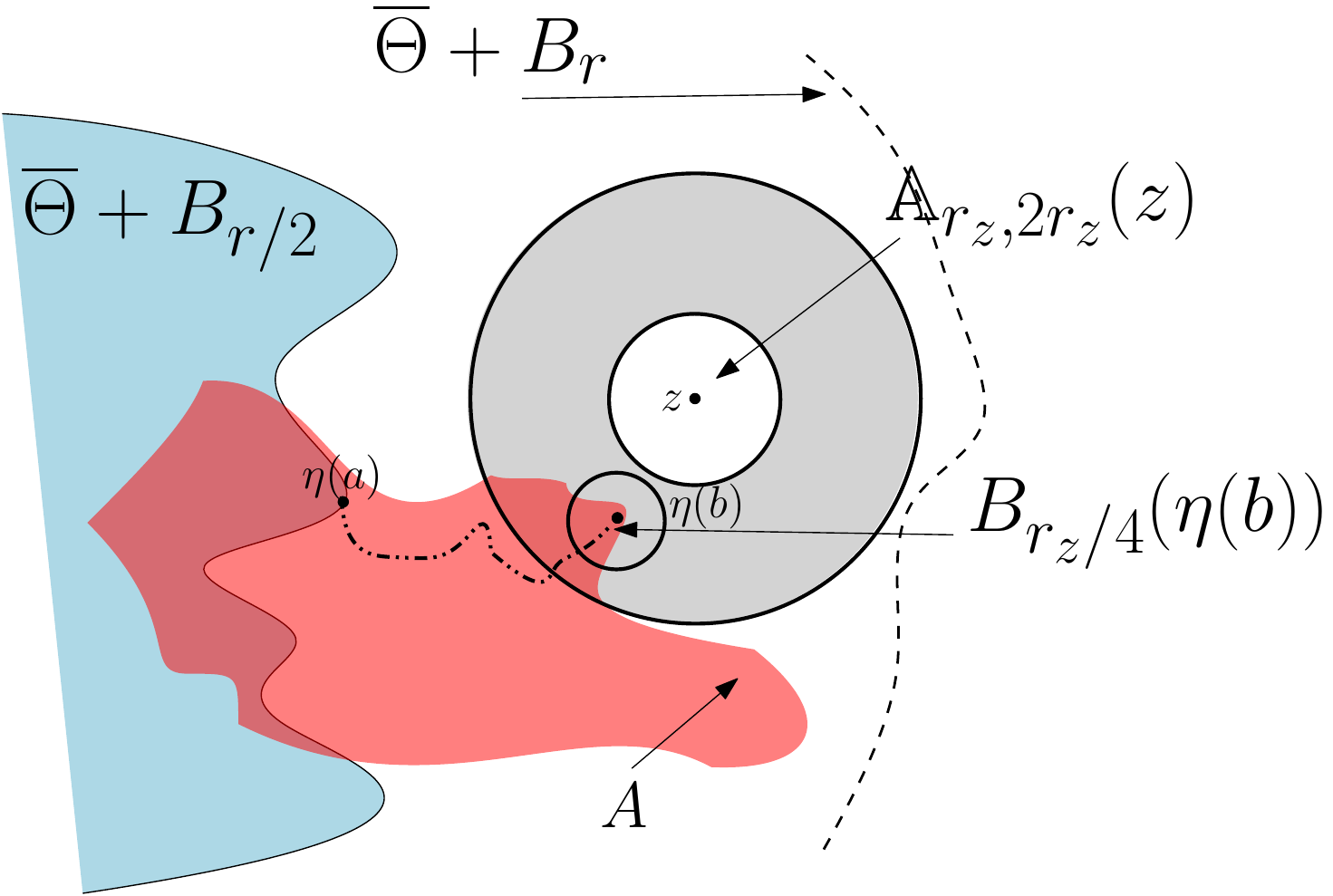}
		\caption{Illustration of part of the proof of Lemma \ref{lemma:constant-order-lower-bound}; only parts of the sets are shown and aspect ratios are not drawn to scale.
			The set $\overline{\Theta} + B_{r/2}$ is shown in blue with a solid black boundary, 
			the boundary of $\overline{\Theta} + B_{r}$ is shown as a dashed black line.
			A good annulus $\A_{r_z, 2 r_z}(z)$ is shown in gray, an interior ball is outlined in black, and the set $A$ is shown 
			in light-red. A random walk started at $a$ must hit ${\VGeps}(B_{{(4/3)} r_{z}}(z))$ for some ${z} \in \mathbf{Z}$
			at a point $b \in {\VGeps}(\partial B_{{(4/3)} r_{z}}(z))$. The proof then reduces to showing a random walk started at $b$ with 
			probability bounded from below must exit 
			$A$ before exiting ${\VGeps}(B_{r_{z}/4}(b))$.} \label{fig:path-decomposition}
	\end{figure}

	We next show that with exponentially high probability in $r$, if one starts $k \leq r^{2 \beta^-} {\eps}^{-1}$ walkers somewhere 
	in the closure of the current cluster $\cl(\Theta)$, then the number of walkers which are absorbed 
	before exiting ${\VGeps}(\overline\Theta + B_r)$ is at least a constant fraction of $k$. 
	The proof uses the constant-order lower bound given by Lemma \ref{lemma:constant-order-lower-bound} together with a concentration inequality for negative binomials. 
	The statement is a modification of \cite[Lemma 6]{duminil2013containing}. Note that \cite[Lemma 6]{duminil2013containing} is much stronger than what we have below, 
	but that stronger statement is not needed to complete the proof of the IDLA upper bound. For the statement, we recall the definition of the modified IDLA aggregate from Definition~\ref{def:stopped-idla}.
	
	\begin{lemma} \label{lemma:chance-of-being-absorbed}		
		There exists deterministic $\omega = \omega(\gamma) > 0$ and $s_3 > 0$ so that for each $s_0 \in (0,1)$, there exists $s_4, s_5 \in (0,1)$ 
		and a random $\sigma(h)$-measurable $\boldsymbol{\delta} >0$ such that with probability approaching one as ${\eps}$ goes to zero, the following holds for $r \in [10 {\eps}^{\omega}, \boldsymbol{\delta}]$ 
		for all sets $\Theta \subset {\VGeps}$ 
		such that $\overline{\Theta} \Subset B_{\rho/2}$ and $B_{s_0} \subset \overline{\Theta}$ with $\overline{\Theta}$ as in \eqref{eq:vertex-scaling}:
		for each $k \leq s_5 r^{2 \beta^-} {\eps}^{-1}$ and $a_1, \ldots, a_{k} \in \cl(\Theta)$, 
		\[
		\P[ |A(\Theta; a_1, \ldots, a_{k} \to {\VGeps}(\overline{\Theta} + B_r)) \setminus \Theta| \leq s_4 k | h ,\eta] \leq \exp(-s_3 k).
		\]
	\end{lemma}
	
	\begin{proof}
		For $s_0 \in (0,1)$, let $\omega$, $s_1, s_2$ be as in Lemma \ref{lemma:constant-order-lower-bound} 
		and fix ${\eps}$ and $r \geq 10 {\eps}^{\omega}$ for which the event of Lemma \ref{lemma:constant-order-lower-bound}  occurs. Set $s_5 = \min(\frac{s_1}{s^{-1}_2 + C}, s_1)$ for a
		deterministic constant $C > 0$ from \eqref{eq:number-of-walkers-nb-bound} below.  Let $k \leq s_5 r^{2 \beta^-} {\eps}^{-1}$ and $a_1, \ldots, a_{k} \in \Theta$
		be given. 
		
		Start with the set $\Theta$ completely occupied and consider an infinite number of walkers each started at an arbitrary site in $\cl(\Theta)$ and stopped upon exiting ${\VGeps}(\overline\Theta + B_r)$. Let $\mathcal{N}(k_1)$ denote the number of such walkers 
		needed until $k_1$ walkers have been absorbed into the cluster. That is, $\mathcal N(k_1)$ counts the number of walkers needed
		until the aggregate has an additional $k_1$ walkers in ${\VGeps}(\overline\Theta + B_r) \setminus \Theta$. 
		
		We seek to iteratively apply Lemma \ref{lemma:constant-order-lower-bound}. Define $S_0 = \Theta$ and for each $k' \in \{1, \ldots, k_1\}$
		condition on the first $(k'-1)$ walkers to determine the current aggregate $S_{k'-1}$. 
		We then apply Lemma \ref{lemma:constant-order-lower-bound} with $S := S_{k'-1}$ to the $k'$-th walker under the conditional law given the first $(k'-1)$ 
		walkers. By the Markov property of IDLA this implies that
		\[
		\mathcal{N}(k_1) \leq \mathrm{NegBinom}(k_1, s_2), \quad \forall k_1 \leq s_1 r^{2 \beta^-} {\eps}^{-1}
		\]
		where $\mathrm{NegBinom}(k_1,s_2)$ denotes a random variable drawn from a negative binomial distribution with $n = k_1$ and $p = s_2$, sampled independently from $(h,\eta)$. 
		Therefore, by a concentration inequality for negative binomial random variables, for each $k_1 \leq s_1 r^{2 \beta^-} {\eps}^{-1}$, 
		\begin{equation} \label{eq:number-of-walkers-nb-bound}
		\P[ \mathcal{N}(k_1) \leq (s_2^{-1} + C) k_1 | h, \eta] \leq \exp(-c k_1), \quad \mbox{for deterministic constants $C, c > 0$}.
		\end{equation}
		Thus, if we choose $k_1 = (s_2^{-1} + C)^{-1} \times k$, we have by our choice of $s_5$ and 
		\eqref{eq:number-of-walkers-nb-bound}, 
		\[
		\P\left[ \mathcal{N}((s_2^{-1} + C)^{-1} \times k) \leq k | h , \eta \right]  \leq \exp(-c k),
		\] 
		where the deterministic constant $c$ has changed compared to \eqref{eq:number-of-walkers-nb-bound}.
		By definition of $\mathcal{N}$, this implies the desired conclusion with $s_3 = c$ and $s_4 = (s_2^{-1} + C)^{-1}$.
	\end{proof}

	\subsection{Iteration} \label{subsec:iteration}
	We now prove a general bound on how far walkers can spread by iterating Lemma \ref{lemma:chance-of-being-absorbed}. 
	The iteration uses the Abelian property of IDLA and involves starting and stopping walkers which reach a certain distance
	of the current cluster. The proof is similar to that of \cite[Theorem 2]{duminil2013containing}.

	\begin{lemma}  \label{lemma:iterative-bound}
		There exists a deterministic exponent $\beta = \beta(\rho, \gamma) > 0$ and a constant $C = C(\rho, \gamma)>0$ so that 
		for each $s_0 \in (0,1)$ there exists a $\sigma(h)$-measurable random variable $\boldsymbol{\delta}_0$
		such that with probability approaching one as ${\eps} \to 0$,  for each $\overline{\Theta} \Subset B_{\rho/2}$ with $B_{s_0} \subset \Theta$ 
		and each $\delta_0 < \boldsymbol{\delta}_0$ the following is true. 
		The IDLA cluster started with  $\Theta$ completely occupied and
		$j \leq \lceil \delta_0 {\eps}^{-1} \rceil$ walkers at sites $\{a_1, \ldots, a_{j}\} \subset \cl(\Theta)$ 
		is contained in a $C \delta_0^\beta$ neighborhood of $\Theta$:
		\[
		A(\Theta; a_1, \ldots, a_{j}) \subset {\VGeps}(\Bdeltapp{C \delta_0^\beta}(\overline{\Theta})) \subset {\VGeps}(B_{\rho}).
		\]
	\end{lemma}
	\begin{proof}
		For $s_0 \in (0,1)$, let the parameters $\omega, s_3, s_4, s_5$ and random variable $\boldsymbol{\delta}$ be as in  Lemma \ref{lemma:chance-of-being-absorbed}.
		We truncate on the event of Lemma \ref{lemma:chance-of-being-absorbed} and fix an ${\eps}$. 
		Below we use the fact that Lemma \ref{lemma:chance-of-being-absorbed} is stated for {\it all}  choices of $\Theta$ satisfying the hypotheses and IDLA is a Markov process, 
		Lemma \ref{lemma:abelian-IDLA}. 	Throughout the proof we introduce several other events 
		which occur with probability approaching one as ${\eps} \to 0$ and will truncate on those events. 
		
		Let $\beta^-$ be as in Lemma \ref{lemma:volume-growth}.
		Let $\boldsymbol{\delta_0} = \boldsymbol{\delta_0}(\boldsymbol{\delta}, \rho, \gamma) < \boldsymbol{\delta}$ be such that $(\boldsymbol{\delta_0} s_5^{-1})^{1/2 \beta^-} \leq \boldsymbol{\delta}$, so that Lemma \ref{lemma:chance-of-being-absorbed} holds with $r \in ( 10 {\eps}^{\omega}, (\boldsymbol{\delta_0} s_5^{-1})^{1/2 \beta^-})$.  
		At the end of the proof, we reveal, in a deterministic fashion,  how small  $\boldsymbol{\delta_0}$ must be
		so that the final cluster is contained in ${\VGeps}(B_{\rho})$. Fix $\delta_0 < \boldsymbol{\delta_0}$.

		Fix $\omega_0 \in (0,\omega)$ which we will choose in \eqref{eq:choice-of-omega0} below. 
		Set $\Theta^0 = \Theta$, $P_0 =\{a_1, \ldots, a_j\}$, 
		and for $k \geq 1$, inductively define (using the notation of Definition~\ref{def:stopped-idla})
		\[
		\begin{cases}
		\mbox{Current radius: } & r_k :=  \max\left( (|P_{k-1}| \times {\eps} \times s_5^{-1})^{1/{2 \beta^-}}, 10 {\eps}^{\omega_0} \right) \\
		\mbox{Current cluster: } & \Theta^{k} := A(\Theta^{k-1}; P_{k-1} \to {\VGeps}(\overline{\Theta^{k-1}} + B_{r_k})) \\
		\mbox{Stopped walkers: } & P_k := P(\Theta^{k-1}; P_{k-1} \to {\VGeps}(\overline{\Theta^{k-1}} + B_{r_k})).
		\end{cases}
		\]
		Define the stopping time
		\begin{equation} \label{eq:iterate-stop-time3}
		\tau_{{\eps},3} = \min \{ k \geq 0 : \overline{\Theta^{k}} \not \subset B_{\rho}\}.
		\end{equation}
		In the analysis below, we will truncate on a high probability event which ensures that $\tau_{{\eps},3}  =\infty$.
		
		Denote the first time there are fewer than 
		${\eps}^{-\omega_0}$ stopped walkers by 
		\begin{equation} \label{eq:iterate-stop-time2}
		\tau_{{\eps},2} = \min \left( \min \{ k \geq 0 : |P_k| \leq {\eps}^{-\omega_0}\}, \tau_{{\eps},3} \right)
		\end{equation}
		and the first time (before $\tau_{{\eps},2}$) that the current radius is less than $10 {\eps}^{\omega_0}$ by 
		\begin{equation}
		\tau_{{\eps},1} = \min \left( \min \{ k \geq 0 : r_k \leq 10 {\eps}^{\omega_0} \}, \tau_{{\eps},2} \right).
		\end{equation}
		By the Abelian property of IDLA, Lemma \ref{lemma:abelian-IDLA}, the distribution of the final IDLA cluster $A(\Theta; a_1, \ldots, a_{\lceil \delta_0 {\eps}^{-1} \rceil})$
		coincides with that of $\Theta^{\infty}$. Thus, it suffices to iteratively control how far the clusters $\Theta^{k}$ spread. In the first step, we control how far the aggregate spreads until 
		time $\tau_{{\eps},1}$, in the second, until time $\tau_{{\eps},2}$. In the third step, we control how far ${\eps}^{-\omega_0}$ walkers can spread.
		We show in the final step, after making $\boldsymbol{\delta_0}$ smaller in a deterministic way, that in fact we have $\tau_{{\eps},2} < \tau_{{\eps},3}$ and 
		$\tau_{{\eps},3} = \infty$. 
		
		Before proceeding, we note that by Lemma \ref{lemma:cell-size-estimate} (after truncating on another high probability event),  for each $\zeta > 0$, 
		\begin{equation} \label{eq:cell-size-estimate-local}
		\diam(H^{{\eps}}_a) \leq {\eps}^{2/(2+\gamma)^2 - \zeta} \quad \forall a \in {\VGeps}(B_{\rho}). 
		\end{equation}
		Choose $\zeta$ in a way which depends only on $\gamma$ and then choose $\omega_0$ small enough so that 
		\begin{equation} \label{eq:choice-of-omega0}
		{\eps}_0^{-\omega_0}  \times {\eps}_0^{2/(2+\gamma)^2 - \zeta} \leq {\eps}_0^{\omega_0/2} ,\quad \forall {\eps}_0 \in (0,1) .
		\end{equation}
		\medskip
		
		{\it Step 1: $t < \tau_{{\eps},1}$.} \\
		All of our parameter choices so far have been so that we can iteratively apply 
		Lemma \ref{lemma:chance-of-being-absorbed} for each $k \in \{1, \ldots, \tau_{{\eps},1}-1\}$ with $\Theta = \Theta^{k}$, 
		$r = r_k$, and $k = |P_{k-1}| \geq {\eps}^{-\omega_0}$.
		
		By \eqref{eq:cell-size-estimate-local} and \eqref{eq:choice-of-omega0} and the fact $r_k \geq 10 {\eps}^{\omega_0}$, there is a `shell' of additional vertices in
		$\Theta^{k}$ that are not in $\Theta^{k-1}$; that is, $\partial \Theta^{k} \cup \Theta^{k-1}$ is not connected. 
		Therefore, at least one walker is absorbed in each step and hence we have $\tau_{{\eps},1} \leq \tau_{{\eps},2} \leq {\eps}^{-1}$.
		Thus, by Lemma \ref{lemma:chance-of-being-absorbed} (with $s_3 > 0$ and $s_4 \in (0,1)$ from there)  and a union bound over all $k=1,\dots,\lfloor{\eps}^{-1} \rfloor$, 
		\begin{equation} \label{eqn:use-absorb}
		\P \left[ \exists k \in \{1, \ldots, \tau_{{\eps},1}-1\} : |P_k| \geq (1-s_4) |P_{k-1}| | h, \eta \right] \leq {\eps}^{-1} \exp(- s_3 {\eps}^{-\omega_0}).
		\end{equation}
		We henceforth truncate on the event in~\eqref{eqn:use-absorb}, which happens with probability tending to 1 as ${\eps}\to 0$
		to see that 	
		\begin{equation} \label{eq:good-event} 
		|P_k| \leq (1-s_4) |P_{k-1}|, \quad \forall k \in \{1, \ldots, \tau_{{\eps},1}-1\} .
		\end{equation}
		This implies that $|P_k| \leq (1-s_4)^k |P_0| \leq (1-s_4)^k \delta_0 {\eps}^{-1}$ for each $k\leq \tau_{{\eps},1}-1$. By the definition of $r_k$, this, in turn, implies that
		\begin{equation} \label{eq:good-event2} 
		r_k \leq (\delta_0\times (1-s_4)^{k-1} \times s_5^{-1})^{1/2\beta^-}, \quad \forall k \in \{1, \ldots, \tau_{{\eps},1}-1\} .
		\end{equation} 
		Therefore, 
		\begin{equation} \label{eq:bound-on-radii-sum}
		\sum_{k=1}^{\tau_{{\eps},1}-1} r_k \leq \sum_{k=1}^{\tau_{{\eps},1}-1} (\delta_0\times (1-s_4)^{k-1} \times s_5^{-1})^{1/2\beta^-} \leq C \times \delta_0^{1/2\beta^-} \times \frac{1}{1 - (1-s_4)^{1/2\beta^-}},
		\end{equation}
		where $C = C(s_5) > 0$ is deterministic. 
		\medskip

		{\it Step 2: $k \in [\tau_{{\eps},1}, \tau_{{\eps},2})$.} \\
		In this step we again apply Lemma \ref{lemma:chance-of-being-absorbed} for each $k \in \{\tau_{{\eps},1}, \ldots, \tau_{{\eps},2}\}$ 
		with the same choice of parameters as in Step 1 except now the radius is fixed $r_k = 10 {\eps}^{\omega_0}$. 
		Note that $r_k \geq (|P_{k-1}| \times {\eps} \times s_5^{-1})^{1/{2 \beta^-}}$ by definition and hence the conditions of Lemma \ref{lemma:chance-of-being-absorbed}
		are satisfied. 
		We show that 
		\begin{equation} \label{eq:time-upper-bound}
		\tau_{{\eps},2} \leq C \log {\eps}^{-1} + \tau_{{\eps},1}, \quad \mbox{for deterministic $C = C(s_4, \gamma) > 0$} 
		\end{equation}
		and hence 
		\begin{equation}
		\sum_{k=\tau_{{\eps},1}}^{\tau_{{\eps},2}} r_k \leq C \times {\eps}^{\omega_0} \times \log {\eps}^{-1}, \quad \mbox{for deterministic $C = C(s_4, \gamma) > 0$}.
		\end{equation}
		To see \eqref{eq:time-upper-bound}, we first use exactly the same argument as in Step 1 (and truncate on another event) to get that	
		\[
		|P_k| \leq (1-s_4) |P_{k-1}|, \quad \forall k \in \{\tau_{{\eps},1}, \ldots, \tau_{{\eps},2} \} .
		\]
		Hence, by the definition~\eqref{eq:iterate-stop-time2} of $\tau_{{\eps},2}$, 
		\[
		{\eps}^{-\omega_0} \leq |P_{\tau_{{\eps},2}-1}| \leq (1-s_4)^{\tau_{{\eps},2}- \tau_{{\eps},1}-1} {\eps}^{-1}
		\]
		which implies \eqref{eq:time-upper-bound} upon re-arranging. 
		\medskip
		
		{\it Step 3: $k \geq \tau_{{\eps},2}$.} \\
		It remains to control how far the remaining ${\eps}^{-\omega_0}$ walkers can go before being absorbed. 
		We do this crudely by appealing to the upper bound on the Euclidean diameter of cells in the mated-CRT map:
		by \eqref{eq:cell-size-estimate-local} and \eqref{eq:choice-of-omega0}
		we have that ${\eps}^{-\omega_0}$ walkers cannot go farther than Euclidean distance ${\eps}^{\omega_0/2}$ before being absorbed into the cluster,
		\begin{equation} \label{eq:distance-of-final-cluster}
		\dist(\overline{\Theta^{\infty}}, \overline{\Theta^{\tau_{{\eps},2}}}) < {\eps}^{\omega_0/2},
		\end{equation}
		since each walker occupies at least one cell.
		\medskip
		
		{\it Step 4: Reduce ${\eps}$ and $\boldsymbol{\delta_0}$ and conclude.} \\
		By combining Steps 1-3, the final aggregate $\overline{\Theta^{\infty}}$ is contained in $\overline{\Theta} + B_R$, where
		\[
		\begin{aligned}
		R &=  C \times \delta_0^{1/2\beta^-}  \quad \mbox{(by \eqref{eq:bound-on-radii-sum})} \\
		&+  C \times {\eps}^{\omega} \times \log {\eps}^{-1}  \quad \mbox{(by \eqref{eq:time-upper-bound})} \\
		&+ {\eps}^{\omega_0/2} \quad \mbox{(by \eqref{eq:distance-of-final-cluster})} 
		\end{aligned}
		\]
		for deterministic constants $C = C(s_4, s_5, \beta^-, \gamma)$. 
		The first term in $R$ dominates for small ${\eps}$.
		This completes the proof after we decrease $\boldsymbol{\delta_0}$ in a deterministic fashion and then ${\eps}$ depending on $\delta_0$ so that $R < \rho/4$ and hence the time $\tau_{{\eps},3} $ of \eqref{eq:iterate-stop-time3} is infinite, as $\overline{\Theta} \Subset B_{\rho/2}$ by assumption.
	\end{proof}

	\subsection{Proof of upper bound} \label{subsec:proof-of-idla-upper-bound}
	We now prove the upper bound by combining the asymptotically correct lower bound together with the general bound of the previous subsection. 
	The idea is the following. By the Abelian property, we may construct the IDLA cluster by first letting walkers evolve until they exit ${\VGeps}(\Lambda_t)$,
	where $\Lambda_t$ is as in Theorem \ref{theorem:harmonic-balls}.
	By the lower bound, Proposition \ref{prop:IDLA-lower-bound}, 
	most of the walkers will have been absorbed into the aggregate at this point. In fact, as the LQG measure of the boundary of harmonic balls is zero (Theorem \ref{theorem:harmonic-balls}),
	the number of remaining walkers can be made arbitrarily small. Thus, we may apply Lemma  \ref{lemma:iterative-bound} to see that the remaining walkers do not spread too far. 
	
	\begin{proof}[Proof of Proposition \ref{prop:IDLA-upper-bound}]
		Let the parameters $\beta, C$ be as in Lemma \ref{lemma:iterative-bound}. 
		Let $\delta \in (0, \rho)$ be given. Truncate on the event that $t < \TT$ and let $\Lambda_t$ be the harmonic ball satisfying the conditions in Theorem \ref{theorem:harmonic-balls}. 
		Since $\Lambda_t$ is open and contains the origin, we have that there exists a random $\sigma(h)$-measurable $s_0 \in (0,1)$ so that $B_{s_0} \Subset \Lambda_t$.
		Recall the notation $B_{\delta}^{\pm}$ from \eqref{eq:outer-inner-neighborhoods}. 
		Since $\mu_h(\partial \Lambda_t) = 0$ (Theorem \ref{theorem:harmonic-balls}), 
		\begin{equation} \label{eq:zero-measure-boundary}
		\lim_{\delta \downarrow 0} \left( \mu_h(\Bdeltap(\Lambda_t)) - \mu_h(\Bdeltam(\Lambda_t)) \right) = 0.
		\end{equation}
		Denote by 
		\[
		A_{{\eps}}(\lfloor t {\eps}^{-1} \rfloor \to {\VGeps}(\Lambda_t)) := A(\emptyset ; 0,\dots,0 \to {\VGeps}(\Lambda_t))
		\]
		the IDLA aggregate formed by $\lfloor t {\eps}^{-1} \rfloor$ walkers started at the origin 
		stopped upon exiting ${\VGeps}(\Lambda_t)$  and let $P_{{\eps}}(\lfloor t {\eps}^{-1} \rfloor \to {\VGeps}(\Lambda_t))$ denote the positions of the stopped (not absorbed) walkers
		on the boundary of ${\VGeps}(\Lambda_t)$. 
		We have by Proposition \ref{prop:IDLA-lower-bound}  
		that for each $\delta_0 \in (0, \rho)$, 
		\begin{equation} \label{eq:asymptotic-lower-bound}
		\Bdeltamm{\delta_0}(\Lambda_t)  \subset \overline{A}_{{\eps}}(\lfloor t {\eps}^{-1} \rfloor \to {\VGeps}(\Lambda_t))
		\end{equation}
		except on an event of probability tending to 0 as ${\eps} \to 0$.
		
		Choose, as in the statement of Lemma \ref{lemma:iterative-bound}, a $\sigma(h)$-measurable $\boldsymbol{\delta_0}$ depending on $s_0$. Assume that $C \boldsymbol{\delta_0}^\beta < \delta$. 
		By \eqref{eq:zero-measure-boundary}, we may choose a $\delta_0 < \boldsymbol{\delta_0}$ sufficiently small so that 
		\begin{equation} \label{eq:small-boundary}
		\mu_h\left(	\Bdeltapp{\delta_0}(\Lambda_t) \setminus	\Bdeltamm{\delta_0}(\Lambda_t) \right) \leq  \boldsymbol{\delta_0}/2. 
		\end{equation}
		The aggregate lower bound \eqref{eq:asymptotic-lower-bound} implies 
		\[
		|P_{{\eps}}(\lfloor t {\eps}^{-1} \rfloor \to {\VGeps}(\Lambda_t))| \leq t {\eps}^{-1} - |{\VGeps}(\Bdeltamm{\delta_0}(\Lambda_t))| 
		\]
		except on an event of probability tending to 0 as ${\eps} \to 0$.
		Let $\alpha = \alpha(\gamma) > 0$ be the parameter from Lemma \ref{lemma:mu_h-convergence}.
		Since by Theorem \ref{theorem:harmonic-balls}, $\mu_h(\Lambda_t) = t$, we have by Lemma \ref{lemma:mu_h-convergence},
		except on an event of probability tending to 0 as ${\eps} \to 0$, 
		\[
		{\eps}^{-1} \left( t - {\eps}^{\alpha}  \right) \leq |{\VGeps}(\Bdeltapp{\delta_0}(\Lambda_t))|  \leq {\eps}^{-1} \left( \mu_h(\Bdeltapp{\delta_0}(\Lambda_t)) + {\eps}^{\alpha} \right)
		\]
		and  
		\[
		|{\VGeps}(\Bdeltamm{\delta_0}(\Lambda_t))| \geq {\eps}^{-1} \left( \mu_h(\Bdeltamm{\delta_0}(\Lambda_t)) - {\eps}^{\alpha} \right). 
		\]
		The previous three indented equations imply that except on an event of probability tending to 0 as ${\eps} \to 0$,
		\[
		|P_{{\eps}}(\lfloor t {\eps}^{-1} \rfloor \to {\VGeps}(\Lambda_t))| \leq {\eps}^{-1} \left( \mu_h(\Bdeltapp{\delta_0}(\Lambda_t) \setminus \Bdeltamm{\delta_0}(\Lambda_t)) 
		+ 3 {\eps}^{\alpha} \right).
		\]		
		This implies by \eqref{eq:small-boundary} that
		\begin{equation}
		|P_{{\eps}}(\lfloor t {\eps}^{-1} \rfloor \to {\VGeps}(\Lambda_t))|  
		\leq \boldsymbol{\delta}_0 {\eps}^{-1}
		\end{equation}
		except on an event of probability tending to 0 as ${\eps} \to 0$.
		
		By the Abelian property of IDLA (Lemma \ref{lemma:abelian-IDLA}), the distribution of the final aggregate 
		coincides with that created after releasing the paused walkers: 
		\[
		A_{{\eps}}(\lfloor t {\eps}^{-1}\rfloor ) \overset{d}{=} A( A_{{\eps}}(\lfloor t {\eps}^{-1} \rfloor \to {\VGeps}(\Lambda_t));P_{{\eps}}(\lfloor t {\eps}^{-1} \rfloor \to {\VGeps}(\Lambda_t))),
		\]
		where the equality in distribution is conditional on $h$ and $\eta$. 
		Since IDLA is a Markov process and $\overline{\Lambda}_t \Subset B_{\rho/2}$  ($t < \TT$), we may apply Lemma \ref{lemma:iterative-bound} with $\Theta = A_{{\eps}}(\lfloor t {\eps}^{-1} \rfloor \to {\VGeps}(\Lambda_t))$ and at most $\boldsymbol{\delta}_0 {\eps}^{-1}$ walkers started at $P_{{\eps}}(\lfloor t {\eps}^{-1} \rfloor \to {\VGeps}(\Lambda_t))$ to see that 
		except on an event of probability approaching zero as ${\eps} \to 0$, 
		\[
		\overline{A}_{{\eps}}(\lfloor t {\eps}^{-1} \rfloor )  
		\subset 
		\Bdeltapp{C \boldsymbol{\delta_0}^\beta}(\overline{A_{{\eps}}}(\lfloor t {\eps}^{-1} \rfloor \to {\VGeps}(\Lambda_t))
		\subset 
		\Bdeltapp{C \boldsymbol{\delta_0}^\beta}(\Lambda_t).
		\]
		This concludes the proof upon recalling $\delta > C \boldsymbol{\delta_0}^\beta$. 	
	\end{proof}

	\section{Divisible sandpile upper bound} \label{sec:div-upper-bound}
	We combine Lemma \ref{lemma:constant-order-lower-bound} together with ideas from \cite[Section 6.5]{bou2022harmonic} 
	to prove an upper bound on the divisible sandpile cluster. For the statement, recall the notation of the cluster $D_{{\eps}}$ from \eqref{eq:cluster-definition}
	and $\overline{D_{{\eps}}}$ from \eqref{eq:vertex-scaling}. 
	\begin{prop} \label{prop:div-sandpile-upper-bound}
		Recall the time $\TT$ from~\eqref{eq:cluster-stopping-time}.
		For each $\delta \in (0, \rho)$ and $t > 0$, on the event $\{t < \TT\}$, 
		it holds except on an event of probability tending to 0 as ${\eps} \to 0$  that	
		\[
		\overline{D}_{{\eps}}(t {\eps}^{-1}) \subset \Bdeltap(\Lambda_t),
		\]
		where $\Lambda_t$ is as in Theorem \ref{theorem:harmonic-balls}. 
	\end{prop}
	Recall from \eqref{eq:least-action-principle} the least action principle for the divisible sandpile odometer, 
	\[
	\frac{v^{{\eps}}_t}{\deg^{{\eps}}}  = \min \{ w:{\VGeps} \to [0,\infty) : \deg^{{\eps}} \times \Delta^{{\eps}} w + t {\eps}^{-1} \delta_0 \leq 1 \},
	\] 
	and that 
	\begin{equation} \label{eq:cluster-contained-in}
	\begin{aligned}
	\Lambdateps &= \{a \in \VGeps: v^{{\eps}}_t > 0 \} \\
	D_{{\eps}}(t {\eps}^{-1}) &=  \cl(\Lambdateps).
	\end{aligned}
	\end{equation}
	In this section we mainly work with the stopped odometer, defined by 
	\begin{equation} \label{eq:stopped-odometer-de}
	u_t^{{\eps}} = \min \{ w: {\VGeps}(B_{\rho}) \to [0,\infty) : \Delta^{{\eps}} w + {\deg^{{\eps}}(0)}^{-1} t {\eps}^{-1} \delta_0 \leq (\deg^{{\eps}})^{-1}\},
	\end{equation}
	which, as discussed in the beginning of Section \ref{sec:divisible-sandpile}, relates to $\wteps$ via  $u_t^{{\eps}} = \wteps + (t {\eps}^{-1}) \times  \gr{{\eps}}{B_\rho}(0, \cdot)$. 
	The reason for doing so is that the results of Section \ref{sec:divisible-sandpile}, a priori, apply just to the stopped odometer.  However, as we will show, for small $t$, the odometer 
	$\frac{v^{{\eps}}_t}{\deg^{{\eps}}}$ agrees with the stopped odometer. In particular, 
	it is straightforward to see that 
	\begin{equation} \label{eq:consistency}
	u_t^{{\eps}} = 0 \quad \mbox{on $\cl({\VGeps}(\partial B_{\rho}))$} 
	\implies
	u_t^{{\eps}} = \frac{v^{{\eps}}_t}{\deg^{{\eps}}} \quad  \mbox{in ${\VGeps}(\cl(B_{\rho}))$}.
	\end{equation}
	Therefore, the upper bound for the divisible sandpile cluster will follow from \eqref{eq:cluster-contained-in} once we show that $u_t^{{\eps}}$ is zero near the boundary of $\Lambda_t$ for $t < \TT$. 
	
	Our strategy for the upper bound is similar to the argument in \cite[Section 6]{bou2022harmonic}, but with some simplifications. 
	We start by establishing harmonic comparison lemmas in Section \ref{subsec:harmonic-comparison} which allow us to compare the mass
	of the cluster to the size of the stopped odometer.  We then use these lemmas together with an iteration to show the upper bound in Section \ref{subsec:div-sandpile-iteration} as follows. By convergence of the stopped odometer away from the origin, we have that $u_t^{{\eps}}$ is small near ${\VGeps}(\partial \Lambda_t)$. This implies that the amount of mass in the shell ${\VGeps}(\Lambda_t + B_r) \setminus {\VGeps}(\Lambda_t)$ is small. 
	Hence, by Lemma \ref{lemma:constant-order-lower-bound-div} the probability that a random walk exits the cluster
	before exiting this shell is bounded from below. By comparison, this implies the size of the odometer must be small on ${\VGeps}(\partial (\Lambda_t + B_r))$. Iterate this argument with shells of decreasing size to see that the odometer must be zero outside of 
	${\VGeps}(\Lambda_t + B_R)$ for some small $R$.
	
	This is simpler than \cite[Section 6]{bou2022harmonic} as we do not have to redo \cite[Sections 6.1-6.3]{bou2022harmonic} since we have already proved a harmonic measure estimate in Lemma \ref{lemma:constant-order-lower-bound}. In particular, we do not consider `very good annuli' explicitly. Further, since we have an a priori strong upper bound on the size of the odometer on ${\VGeps}(\partial \Lambda_t)$ due to its convergence (Theorem \ref{theorem:convergence-divisible}), the inductive counterpart of \cite[Section 6.5]{bou2022harmonic} is shorter. 
	
	\subsection{Harmonic comparison} \label{subsec:harmonic-comparison}
	In this subsection, we prove discrete analogues of the harmonic comparison lemmas for the Hele-Shaw odometer from \cite[Section 6.4]{bou2022harmonic}. First note that since  $\Delta^{\eps} u_t^{\eps} = \Delta^{\eps} \wteps + (t {\eps}^{-1}) \times \Delta^{\eps} \gr{{\eps}}{B_{\rho}}$,  by combining \eqref{eq:discrete-normalized-green-laplacian} and Lemma \ref{lemma:discrete-basic-properties} we have that 
	\begin{equation} \label{eq:laplacian-of-uteps}
	\begin{aligned}
	0 \leq \Delta^{{\eps}} u_t^{{\eps}}  + {\deg^{{\eps}}(0)}^{-1} t {\eps}^{-1} \delta_0 &\leq (\deg^{{\eps}})^{-1}
	\quad \mbox{on ${\VGeps}(B_{\rho})$} \\
	u_t^{{\eps}} &\geq 0  	\quad \mbox{on ${\VGeps}(B_{\rho})$} \\
	\Delta^{{\eps}} u_t^{{\eps}}  + {\deg^{{\eps}}(0)}^{-1} t {\eps}^{-1} \delta_0 &= (\deg^{{\eps}})^{-1}
	\quad \mbox{on $\Lambdateps$} \\
	u_t^{{\eps}} &= 0 \quad \mbox{on ${\VGeps}(B_{\rho}) \setminus \Lambdateps$}.
	\end{aligned}
	\end{equation}
	We use this to show the following. 
	\begin{lemma} \label{lemma:mass-bound-odometer}
		There exists deterministic constants $C = C(\rho, \gamma) > 0$, $\boldsymbol{c}_0 = \boldsymbol{c}_0(\rho, \gamma) \in (0, 1)$ such that with polynomially high probability as ${\eps} \to 0$, the following holds. For all $t > 0$, 
		all connected sets $\Theta \subset {\VGeps}(B_{\rho})$ containing the origin, 
		and all $s > 0$, $z \in \C$ such that ${\VGeps}(B_{s}(z)) \subset {\VGeps}(B_{\rho}) \setminus \Theta$, 
		\[
		\left| {\VGeps}(B_{\boldsymbol{c}_0 s}(z)) \cap \Lambdateps \right| \leq C \times \sup_{\partial \Theta} u_t^{{\eps}}.
		\]
	\end{lemma}
	
	\begin{proof}
		Truncate on the events of Lemmas \ref{lemma:greens-kernel-lower-bound} and \ref{lemma:cell-size-estimate}.
		By \eqref{eq:laplacian-of-uteps} we have that $\Delta^{{\eps}} u_t^{{\eps}} \geq 0$ away from the origin and $\Delta^{{\eps}} u_t^{{\eps}} = (\deg^{{\eps}})^{-1}$ on $\Lambdateps \setminus \{0\}$. 
		Hence, by the maximum principle,  as $u_t^{{\eps}} = 0$ on $\partial B_{\rho}$ and $u_t^{{\eps}}$ is subharmonic away from the origin, 
		\begin{equation} \label{eq:upper-bound-on-ut}
		\sup_{\Lambdateps \cap {\VGeps}(B_s(z))} u_t^{{\eps}} \leq  \sup_{\partial \Theta} u_t^{{\eps}}.
		\end{equation}
		We will prove the desired inequality by comparing $u_t^{{\eps}}$ to the  following function 
		\[
		\mathfrak{f}^{{\eps}}(\cdot) =  \sum_{b \in {\VGeps}(B_s(z)) \cap \Lambdateps \setminus  {\VGeps}(\partial B_s(z))} \gr{{\eps}}{B_s(z)}(\cdot,b)
		\]
		which, for the same reasoning leading to \eqref{eq:discrete-normalized-exit-laplacian}, satisfies 
		\[
		\begin{cases}
		\Delta^{{\eps}} \mathfrak{f}^{{\eps}}(\cdot) = -\deg^{{\eps}}(\cdot)^{-1} 1\{ \cdot \in \Lambdateps\} \quad &\mbox{in ${\VGeps}(B_s(z)) \setminus {\VGeps}(\partial B_s(z))$} \\
		\mathfrak{f}^{{\eps}}(\cdot) = 0 \quad &\mbox{otherwise}.
		\end{cases} 
		\]
		We then compute:
		\begin{align*}
		\sup_{{\VGeps}(B_s(z))} \mathfrak{f}^{{\eps}}  &\leq 	\sup_{{\VGeps}(B_s(z))} ( u_t^{{\eps}} + \mathfrak{f}^{{\eps}} ) \quad \mbox{($u_t^{{\eps}} \geq 0$)} \\
		&\leq \sup_{{\VGeps}(\partial B_s(z))} (	u_t^{{\eps}} + \mathfrak{f}^{{\eps}} )  \quad \mbox{($u_t^{{\eps}} + \mathfrak{f}^{{\eps}}$ is subharmonic)} \\
		&= \sup_{{\VGeps}(\partial B_s(z))} 	u_t^{{\eps}} \quad \mbox{($\mathfrak{f}^{{\eps}} = 0$ on ${\VGeps}(\partial B_s(z))$)} \\
		&\leq \sup_{\partial \Theta} u_t^{{\eps}} \quad \mbox{(by \eqref{eq:upper-bound-on-ut})}.
		\end{align*}	
		Let $\boldsymbol{c}_0$ be as in \eqref{eq:choice-of-constant-annulus} so that
		by Lemma \ref{lemma:greens-kernel-lower-bound} 
		\[
		\mathfrak{f}^{{\eps}} \geq \sum_{b \in {\VGeps}(B_{\boldsymbol{c_0} s}(z)) \cap \Lambdateps}  \gr{{\eps}}{B_s(z)}(\cdot,b)
		\geq \frac{1}{2} 	 \left| {\VGeps}(B_{\boldsymbol{c}_0 s}(z)) \cap \Lambdateps \right|.
		\]
		Combining the previous two indented equations completes the proof. 
	\end{proof}

	\begin{lemma} \label{lemma:odometer-bound-mass}
		The following holds for all $t > 0$ and all sets $B \subset {\VGeps}(B_{\rho}) \setminus \{0\}$:
		\[
		u_t^{{\eps}}(b) \leq \left( \sup_{\partial B} u_t^{{\eps}} \right) \P[\mbox{$X^{b, {\eps}}$ exits $B$ before hitting $B \setminus \Lambdateps$} | h, \eta], \quad \forall b \in B,
		\]	
		where $X^{b, {\eps}}$	is a simple random walk on ${\Geps}$ started from the vertex $b \in {\VGeps}$.
	\end{lemma}
	\begin{proof}
		The function 
		\[
		f(b) := \left( \sup_{\partial B} u_t^{{\eps}} \right) \P[\mbox{$X^{b, {\eps}}$ exits $B$ before hitting $B \setminus \Lambdateps$} | h, \eta]
		\]
		satisfies 
		\[
		\begin{cases}
		\Delta^{{\eps}} f = 0 \quad \mbox{on $B \cap \Lambdateps$} \\
		f = 0 \quad \mbox{on $B \cap (\Lambdateps)^c$} \\
		f = \left( \sup_{\partial B} u_t^{{\eps}} \right) \quad \mbox{on $B^c \cap \Lambdateps$} 
		\end{cases}
		\]
		Therefore, by \eqref{eq:laplacian-of-uteps}, we have 
		\[
		\begin{cases}
		\Delta^{{\eps}} (u_t^{{\eps}} - f) \geq 0 \quad \mbox{on $B \cap \Lambdateps$} \\
		(u_t^{{\eps}} - f) = 0 \quad \mbox{on $B \cap (\Lambdateps)^c$} \\
		(u_t^{{\eps}} - f) \leq 0 \quad \mbox{on $B^c \cap \Lambdateps$},
		\end{cases}
		\]
		which implies the claim by the maximum principle and the fact $u_t^{{\eps}} = 0$ on $(\Lambdateps)^c$. 
	\end{proof}

	\subsection{Iteration} \label{subsec:div-sandpile-iteration}
	The proof of Proposition \ref{prop:div-sandpile-upper-bound} proceeds by a similar iteration to the proof of Proposition \ref{prop:IDLA-upper-bound}
	but with the odometer replacing the role of the random walks.  We show that the maximum of the odometer decreases across shells of 
	increasing radii. 
	Eventually the odometer will be so small the upper bound we derive from the following estimate will suffice. 
	\begin{lemma}[Lemma 4.2 \cite{levine2009strong}] \label{lemma:crude-upper-bound}
		The following is true for all $t > 0$ with polynomially high probability as ${\eps} \to 0$. 
		For every point $a \in D_{{\eps}}(t {\eps}^{-1}) \setminus \{0\}$  there is a path $a = a_0 \sim a_1 \sim \cdots \sim a_m = 0$ in $D_{{\eps}}(t {\eps}^{-1})$ 
		with 
		\[
		u^{{\eps}}_t(a_{i+1}) \geq u^{{\eps}}_t(a_i) + \log({\eps}^{-1})^{-2}.
		\]
	\end{lemma}

	\begin{proof}
		By \cite[Lemma 2.6]{gwynne2019harmonic} it holds with polynomially high probability as ${\eps} \to 0$ that 
		\begin{equation} \label{eq:degree-bound}
		\max_{a \in {\VGeps}(B_{\rho})} \deg^{{\eps}}(a) \leq (\log {\eps}^{-1})^2.
		\end{equation}
		Let $a_0 = a$. Inductively, assume that $i\geq 0$ and $a_i \in  D_{{\eps}}(t {\eps}^{-1}) $ has been defined. If $a_i = \{0\}$, set $m=i$ and conclude the induction. If $a_i \not=0$, let $a_{i+1}$ be a neighbor of $a_i$ maximizing $u^{{\eps}}_t(a_{i+1})$.
		Then $a_{i+1} \in  \Lambdateps$. Indeed, either $a_i \in \Lambdateps$ and hence $a_{i+1} \in \Lambdateps$
		or $i = 0$ and $a_i$ is on the boundary of $\Lambdateps$. In the latter case, by definition 
		there is some site in $\Lambdateps$ neighboring $a_i$ with a strictly positive odometer. Since $\Delta^{{\eps}} u^{{\eps}}_t = (\deg^{{\eps}})^{-1}$ on $\Lambdateps$,
		\eqref{eq:laplacian-of-uteps},	
		we have that 
		\begin{align*}
		u^{{\eps}}_t(a_{i+1}) &\geq \frac{1}{\deg^{{\eps}}(a_i)} \sum_{y \sim a_i} u^{{\eps}}_t(y) \quad \text{(choice of $a_{i+1}$)} \\
		&=  u^{{\eps}}_t(a_i) + \Delta^{{\eps}} u^{{\eps}}_t(a_i) \quad \text{(definition of $\Delta^{\eps}$)} \\
		&= u^{{\eps}}_t(a_i) + \deg^{{\eps}}(a_i)^{-1} \\
		&\geq u^{{\eps}}_t(a_i) + \log({\eps}^{-1})^{-2} \quad\text{(by~\eqref{eq:degree-bound})}.
		\end{align*} 
		As the number of vertices in $D_{{\eps}}(t {\eps}^{-1})$ is finite and contains the origin, Lemma \ref{lemma:discrete-conservation-of-mass}, there must be a finite $i$ for which $a_i = 0$. 
	\end{proof}

	\begin{figure}
		\includegraphics[width=0.5\textwidth]{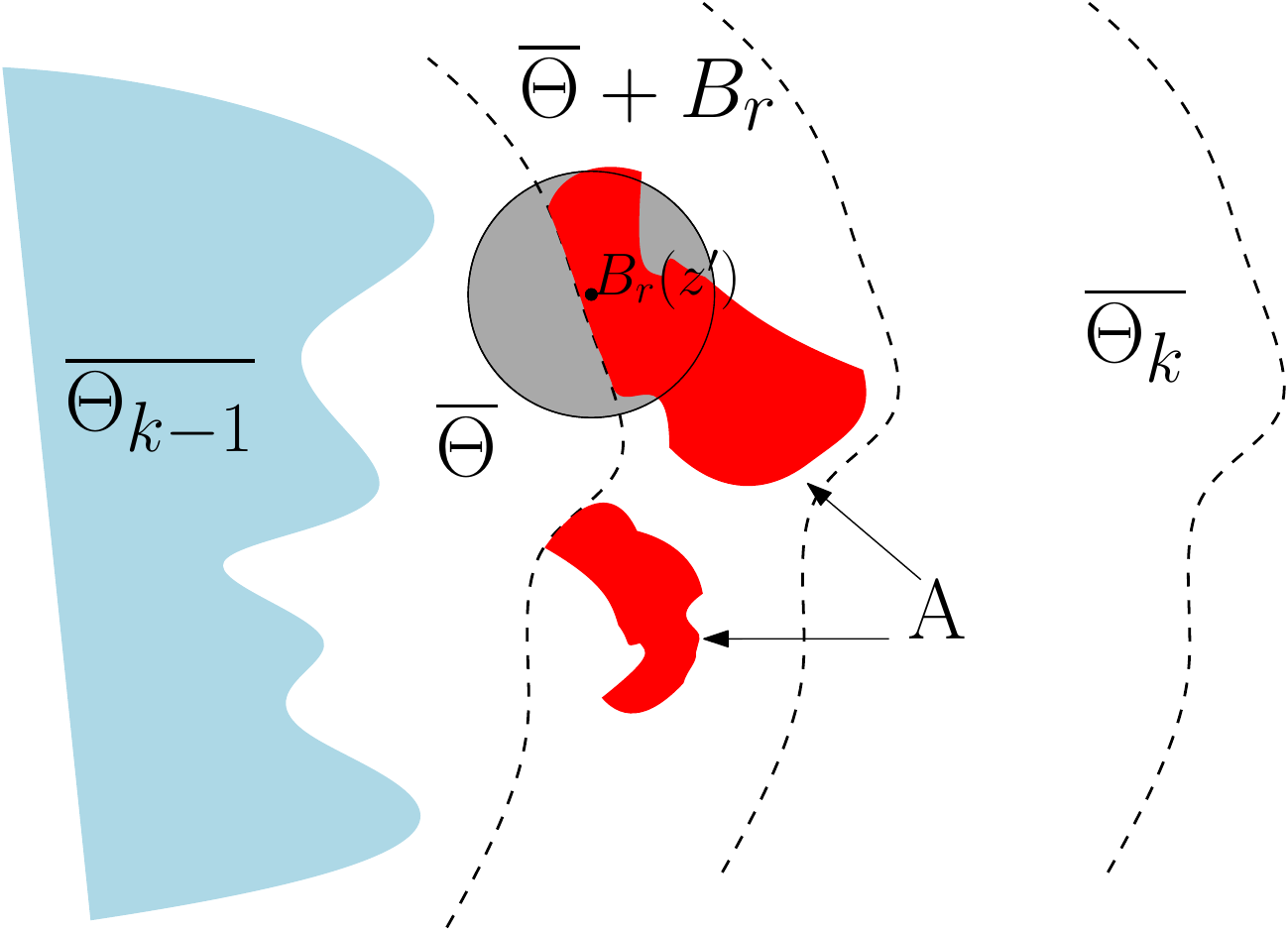}
		\caption{Illustration of the objects appearing in Step 1 of the proof of Proposition \ref{prop:div-sandpile-upper-bound}.
			Part of the current cluster $\overline{\Theta_{k-1}}$ is shown in light blue
			and parts of the boundaries of the intermediate clusters $\overline{\Theta}$ and $\overline{\Theta} + B_r$, defined in \eqref{eq:parameter-choices}
			and the next cluster $\overline{\Theta_k}$ are outlined by dashed lines. The ball $B_r(z')$ is shown in light gray 
			and part of the set $A$, defined in \eqref{eq:parameter-choices}, is shown in red. } \label{fig:div-sandpile-upper-bound-step1}
	\end{figure}
	
	\begin{proof}[Proof of Proposition \ref{prop:div-sandpile-upper-bound}] 
		Let $\delta \in (0, \rho)$ be given. Truncate on the event that $t < \TT$ and let $\Lambda_t$ be the harmonic ball satisfying the conditions in Theorem \ref{theorem:harmonic-balls}. 
		Since $\Lambda_t$ is open and contains the origin, we have that there exists a random $\sigma(h)$-measurable $s_0 \in (0,1)$ so that $B_{s_0} \Subset \Lambda_t$.
		
		For this $s_0$,  let the parameters $\omega, s_1, s_2, \boldsymbol{\delta}$ be as in Lemma \ref{lemma:constant-order-lower-bound-div}. Let $\beta^-$ be as in Lemma \ref{lemma:volume-growth} and
		\begin{equation} \label{eq:s6}
		s_6 := s_1^{-1} \boldsymbol{c_0}^{-2 \beta^-} 4^{2 \beta^-} \times C
		\end{equation}
		where $C$ and $\boldsymbol{c_0}$ are the deterministic constants from Lemma \ref{lemma:mass-bound-odometer}. Let $\boldsymbol{\delta_0} = \boldsymbol{\delta_0}(\boldsymbol{\delta}, \rho, \gamma) < \boldsymbol{\delta}$ be such that $(\boldsymbol{\delta_0} s_6)^{1/2 \beta^-} \leq \boldsymbol \delta$, so that Lemma \ref{lemma:constant-order-lower-bound-div} holds with $r \in ( 10 {\eps}^{\omega}, (\boldsymbol{\delta_0} s_6)^{1/2 \beta^-})$.  
		At the end of the proof we state how small $\boldsymbol{\delta_0}$ must be so that the final cluster is contained in ${\VGeps}(B_{\rho})$. Fix $\delta_0 <
		\boldsymbol{\delta_0}$ and decrease $\boldsymbol{\delta_0}$ so that $2 \times C_1 \times \boldsymbol{\delta_0}^\beta < \delta$ for a random constant $C_1$ to be determined below. 
		
		By the uniform convergence of the scaled version $\overline u_t^{\eps}$ of $u_t^{{\eps}}$ away from the origin (Theorem \ref{theorem:convergence-divisible} and Lemma \ref{lemma:uniform-convergence-of-greens-function}) and the fact that $\Lambda_t$ a.s.\ contains a neighborhood of 0, on the event $\{t < \TT\}$, 
		\begin{equation} \label{eq:start-iteration}
		\sup_{{\VGeps}(\partial \Lambda_t)} u_t^{{\eps}} < \delta_0 {\eps}^{-1},
		\end{equation}
		except on an event of probability tending to 0 as ${\eps} \to 0$.
		Fix $\omega_0 \in (0,\omega)$ which we will choose in \eqref{eq:choice-of-omega0-div} below. Set 
		$\Theta_0 = {\VGeps}(\Lambda_t)$, $r_0 = \delta_0$, $O_0 = \max_{\partial \Theta_0} u^{{\eps}}_t$
		and for $k \geq 1$, inductively define 
		\begin{equation} \label{eq:div-sandpile-cluster-iteration}
		\begin{cases}
		\mbox{Current cluster: } & \Theta_k := {\VGeps}(\overline{\Theta_{k-1}} + B_{r_{k-1}}) \\
		\mbox{Current radius: } & r_k :=  \max\left( (|O_{k-1}| \times {\eps} \times s_6)^{1/{2 \beta^-}}, 40 \boldsymbol{c_0}^{-1} {\eps}^{\omega_0} \right) \\
		\mbox{Current maximum odometer: } & O_k := \max_{\partial \Theta_k} u^{{\eps}}_t.
		\end{cases}
		\end{equation}
		Denote the first time the odometer is smaller than 
		${\eps}^{-\omega_0}$ by 
		\begin{equation} \label{eq:iterate-stop-time2-div}
		\tau_{{\eps},2} = \min\{ k \geq 0 : |O_k| \leq {\eps}^{-\omega_0} \}
		\end{equation}
		and the first time (before $\tau_{{\eps},2}$) that the current radius is less than $40 \boldsymbol{c_0}^{-1} {\eps}^{\omega_0}$ by 
		\begin{equation}
		\tau_{{\eps},1} = \min \left( \min\{ k \geq 0 : r_k \leq 40 \boldsymbol{c_0}^{-1} {\eps}^{\omega_0} \}, \tau_{{\eps},2} \right).
		\end{equation}
		By Lemma \ref{lemma:cell-size-estimate}, for each $\zeta > 0$, 
		\begin{equation} \label{eq:cell-size-estimate-local-div}
		\diam(H^{{\eps}}_a) \leq {\eps}^{2/(2+\gamma)^2 - \zeta} \quad \forall a \in {\VGeps}(B_{\rho}),
		\end{equation}
		w.p.\ tending to 1 as ${\eps} \to 0$. 
		Choose $\zeta$ in a way which depends only on $\gamma$ and then choose $\omega_0$ small enough so that 
		\begin{equation} \label{eq:choice-of-omega0-div}
		{\eps}_0^{-\omega_0}  \times {\eps}_0^{2/(2+\gamma)^2 - \zeta} \leq {\eps}_0^{\omega_0/2} ,\quad \forall {\eps}_0 \in (0,1) .
		\end{equation}

		{\it Step 1: $k < \tau_{{\eps},1}$.} \\
		We first show that for each $t \in \{1, \ldots, \tau_{{\eps},1}-1\}$,
		\begin{equation} \label{eq:iterative-bound-of-o}
		O_k \leq (1-s_2) O_{k-1},
		\end{equation}
		except on an event of probability tending to 0 as ${\eps} \to 0$.
		By Lemma \ref{lemma:mass-bound-odometer}, for all $z \in \C$ and $s > 0$ such that $B_{s}(z) \subset B_{\rho} \setminus \overline{\Theta_{k-1}}$:
		\begin{equation} \label{eq:mass-bound-in-proof}
		\left| {\VGeps}(B_{\boldsymbol c_0 s}(z)) \cap \Lambdateps \right| \leq C \times \sup_{\partial \Theta_{k-1}} u_{t}^{{\eps}},
		\end{equation}
		where $\boldsymbol{c_0}$ is as in Lemma \ref{lemma:mass-bound-odometer}, except on an event of probability tending to 0 as ${\eps} \to 0$.
		As we justify just below, this implies we can apply Lemma \ref{lemma:constant-order-lower-bound-div} with parameters
		\begin{equation} \label{eq:parameter-choices}
		\begin{aligned}
		r &= \boldsymbol{c_0} r_k/4 \\
		\Theta &= {\VGeps}(\overline{\Theta_{k-1}} + B_{r_k/4}) \\
		A &= \Lambdateps \cap \{ {\VGeps}(\overline{\Theta} + B_r) \setminus \Theta \}.
		\end{aligned}
		\end{equation}
		See Figure \ref{fig:div-sandpile-upper-bound-step1} for an illustration of these sets. Indeed, we have that $A \subset {\VGeps}(\overline{\Theta} + B_r) \setminus \Theta$ and, by \eqref{eq:mass-bound-in-proof}, for $z \in \C$:
		\begin{align*}
		|A \cap {\VGeps}(B_r(z))|  &=
		|A \cap {\VGeps}(B_{\boldsymbol{c_0} r_k/4}(z))|  \\
		&\leq |\Lambdateps \cap {\VGeps}(B_{\boldsymbol{c_0} r_k/4}(z'))|  \\
		&\quad \mbox{(for some $z' \in \C$ s.t.\ $B_{r_k/4}(z') \subset B_{\rho} \setminus \overline{\Theta_{k-1}}$  by definition of $A$ \eqref{eq:parameter-choices})} \\
		&\leq C \times \sup_{\partial \Theta_{k-1}} u_t^{{\eps}} \quad \mbox{(by \eqref{eq:mass-bound-in-proof})} \\
		&= C \times O_{k-1}   \quad \mbox{(definition of $O_{k-1}$ \eqref{eq:div-sandpile-cluster-iteration})} \\
		&=  r_k^{2 \beta^-} \times {\eps}^{-1} \times s_6^{-1} \times C  \quad \mbox{(definition of $r_k$ \eqref{eq:div-sandpile-cluster-iteration})} \\
		&=  s_1  r^{2 \beta^-} {\eps}^{-1} \quad \mbox{(definition of $r$ and $s_6$  \eqref{eq:parameter-choices} and \eqref{eq:s6})}.
		\end{align*}
		Thus, as $r \geq 10 {\eps}^{\omega}$ the hypotheses of Lemma \ref{lemma:constant-order-lower-bound-div} are satisfied and we may apply it to get 
		\begin{equation} \label{eq:probability-bound-in-proof}
		\begin{aligned}
		\sup_{a \in {\VGeps}(\partial (\overline{\Theta_{k-1}} + B_{r_k/4 + \boldsymbol{c_0} r_k/8}))}
		\P[&\mbox{$X^{a, {\eps}}$ exits ${\VGeps}(\overline{\Theta_{k-1}} + B_{r_k/4 + \boldsymbol{c_0} r_k/4} )\setminus
			{\VGeps}(\overline{\Theta_{k-1}} + B_{r_k/4}))$} \\
		&\mbox{before exiting $A$} | h , \eta]  
		\leq (1 - s_2)
		\end{aligned}
		\end{equation}
		except on an event of probability tending to 0 as ${\eps} \to 0$.
		Next, by Lemma \ref{lemma:odometer-bound-mass} applied to the set 
		\[
		B := {\VGeps}(\overline{\Theta_{k-1}} + B_{r_k/4 + \boldsymbol{c_0} r_k/4} )\setminus
		{\VGeps}(\overline{\Theta_{k-1}} + B_{r_k/4})
		\]
		together with the maximum principle for $u_t^{{\eps}}$ ($u_t^{{\eps}} = 0$ on ${\VGeps}(\partial B_{\rho})$),
		\[
		u_t^{{\eps}}(b) \leq O_{k-1} \times \P[\mbox{$X^{b, {\eps}}$ exits $B$ before exiting $A$} | h , \eta], \quad \forall b \in B.
		\]
		Another application of the maximum principle shows 
		\[
		O_{k} \leq  \sup_{b \in {\VGeps}(\partial (\overline{\Theta_{k-1}} + B_{r_k/4 + \boldsymbol{c_0} r_k/8}))} u_{t}^{{\eps}}(b).
		\]
		The previous two indented equations together with \eqref{eq:probability-bound-in-proof} imply \eqref{eq:iterative-bound-of-o}.
		
		We now use \eqref{eq:iterative-bound-of-o}, the definition of $r_k$,  and the initial bound given by \eqref{eq:start-iteration} to see that
		\begin{equation} \label{eq:good-event3} 
		r_k \leq (\delta_0\times (1-s_2)^k \times s_6)^{1/2\beta^-}, \quad \forall k \in \{1, \ldots, \tau_{{\eps},1}-1\} 
		\end{equation} 
		and consequently, 
		\begin{equation} \label{eq:bound-on-radii-sum-div}
		\sum_{k=0}^{\tau_{{\eps},1}-1} r_k \leq \sum_{k=0}^{\tau_{{\eps},1}-1} (\delta_0\times (1-s_2)^k \times s_6)^{1/2\beta^-} \leq C \times \delta_0^{1/2\beta^-} \times \frac{1}{1 - (1-s_2)^{1/2\beta^-}},
		\end{equation}
		except on an event of probability tending to 0 as ${\eps} \to 0$ where $C = C(s_6) > 0$ is a random constant. 
		\medskip

		{\it Step 2: $t \in [\tau_{{\eps},1}, \tau_{{\eps},2})$.} \\
		We again apply Lemma \ref{lemma:constant-order-lower-bound-div} for each $k \in \{\tau_{{\eps},1}, \ldots, \tau_{{\eps},2}\}$ with the same choice of parameters as in Step 1.
		Reasoning similar to the proof of Step 1 shows that 
		\[
		O_k \leq (1-s_2) O_{k-1}, \quad \forall k \in \{\tau_{{\eps},1}, \ldots, \tau_{{\eps},2} \} 
		\]
		except on an event of probability tending to 0 as ${\eps} \to 0$.
		Hence, by the definition~\eqref{eq:iterate-stop-time2-div} of $\tau_{{\eps},2}$, 
		\[
		{\eps}^{-\omega_0} \leq |O_{\tau_{{\eps},2}-1}| \leq C (1-s_2)^{\tau_{{\eps},2}- \tau_{{\eps},1}-1} {\eps}^{-1}, \quad \mbox{for $C = C(s_6) > 0$} 
		\]
		which, upon re-arranging, implies
		\begin{equation} \label{eq:time-upper-bound-div}
		\tau_{{\eps},2} \leq C \log {\eps}^{-1} + \tau_{{\eps},1}, \quad \mbox{for random $C = C(s_2, s_6) > 0$} 
		\end{equation}
		and hence 
		\begin{equation}
		\sum_{k=\tau_{{\eps},1}}^{\tau_{{\eps},2}} r_k \leq C \times {\eps}^{\omega_0} \times \log {\eps}^{-1}, \quad \mbox{for random $C = C(s_2, s_6) > 0$}
		\end{equation}
		except on an event of probability tending to 0 as ${\eps} \to 0$.
		\medskip
		
		{\it Step 3: $k \geq \tau_{{\eps},2}$.} \\
		We control the size of the odometer by Lemma \ref{lemma:crude-upper-bound}
		and the upper bound on the Euclidean diameter of cells in the mated-CRT map. 
		By Lemma \ref{lemma:crude-upper-bound},  w.p.\ tending to 1 as ${\eps} \to 0$, there is a path 
		of at most $\log({\eps}^{-1})^2 {\eps}^{-\omega_0}$ cells
		from the boundary of the support of $u_t^{{\eps}}$ 
		to $\partial \Theta_{\tau_{{\eps},2}}$ along which the odometer increases by ${\eps}^{-\omega_0}$. 
		Thus, by the definition of $\tau_{{\eps},2}$, \eqref{eq:cell-size-estimate-local-div}, and \eqref{eq:choice-of-omega0-div}
		\begin{equation} \label{eq:distance-of-final-cluster-div-sandpile}
		u_t^{{\eps}}(a) = 0  \quad \mbox{for $a \in {\VGeps}(B_{\rho})$ such that $\dist(\eta(a), \eta(\Theta_{\tau_{{\eps},2}})) \geq  {\eps}^{\omega_0/2} \log({\eps}^{-1})^2$}
		\end{equation}
		except on an event of probability tending to 0 as ${\eps} \to 0$.
		\medskip
		
		{\it Step 4: Reduce ${\eps}$ and $\boldsymbol{\delta_0}$ and conclude.} \\
		Recall that we have set $\Theta_0 = {\VGeps}(\Lambda_t)$, so that $\overline\Theta_0$ is the union of the ${\eps}$-mated-CRT map cells which intersect $\Lambda_t$.
		By combining Steps 1-3, \eqref{eq:cluster-contained-in}, and \eqref{eq:consistency},
		the divisible sandpile cluster, $\overline{D}_{{\eps}}(t {\eps}^{-1})$ is contained in $\overline{\Theta_0} + B_R$, where
		\[
		\begin{aligned}
		R &=  C_1 \times \delta_0^{1/2\beta^-}  \quad \mbox{(by \eqref{eq:bound-on-radii-sum-div})} \\
		&+  C_1 \times {\eps}^{\omega} \times \log {\eps}^{-1}  \quad \mbox{(by \eqref{eq:time-upper-bound-div})} \\
		&+ {\eps}^{\omega_0/2} \log({\eps}^{-1})^2 \quad \mbox{(by \eqref{eq:distance-of-final-cluster-div-sandpile})} 
		\end{aligned}
		\]
		for random constants $C_1 = C_1(s_2, s_6, \beta^-, \gamma)$, provided that $R < \rho/2$,
		except on an event of probability tending to 0 as ${\eps} \to 0$.
		The first term in the definition of $R$ dominates for small ${\eps}$.
		Decrease $\boldsymbol{\delta_0}$ in a deterministic fashion and then decrease ${\eps}$ depending on $\delta_0$ so that $R < \rho/2$ and hence $u_t^{{\eps}} = v_t^{{\eps}}$.  
		This completes the proof after recalling that $\delta > 2 \times C_1 \boldsymbol{\delta_0}^\beta$. 
	\end{proof}
	
	\section{Convergence} \label{sec:proof-of-convergence}
	We start by combining Propositions \ref{prop:IDLA-lower-bound}, \ref{prop:IDLA-upper-bound}, \ref{prop:div-sandpile-lower-bound}, and  \ref{prop:div-sandpile-upper-bound}
	into the following statement. 
	\begin{prop} \label{prop:IDLA-constrained-convergence}
		Recall the time $\TT$ from~\eqref{eq:cluster-stopping-time}
		and the notation $B_{\delta}^{\pm}$ from \eqref{eq:outer-inner-neighborhoods}.  For each $\delta \in (0, \rho)$ and $t > 0$, on the event $\{t < \TT\}$, 
		it holds except on an event of probability tending to 0 as ${\eps} \to 0$  that
		\[
		\Bdeltam(\Lambda_t) \subset \overline{\mathcal{X}}_{{\eps}}(t {\eps}^{-1}) \subset \Bdeltap(\Lambda_t)
		\]
		for $\mathcal{X} \in \{A, D\}$,
		where $\Lambda_t$ is as in Theorem \ref{theorem:harmonic-balls},  \ie, both the IDLA cluster and divisible sandpile cluster converge. 
	\end{prop}
	To prove Theorems \ref{theorem:convergence-of-IDLA} and \ref{theorem:convergence-of-div-sandpile}, we will combine the previous proposition with a scaling argument to transfer from $t\in (0,\TT)$ to an arbitrary $t>0$.
	
	\subsection{Proof of Theorems \ref{theorem:convergence-of-IDLA} and \ref{theorem:convergence-of-div-sandpile}} \label{subsec:proof-of-idla}
	We use the scale-invariance property of the $\gamma$-quantum cone. Fix $b > 0$ and for $R_b$ as in 
	\eqref{eq:cone-scale-radius}, let
	\[
	h^b := h(R_b \cdot) + Q \log R_b - \frac{1}{\gamma} \log b \quad \text{and} \quad  \eta^b = R^{-1}_b \eta(b \cdot). 
	\]
	By Lemma \ref{lemma:scale-invariance}, $(h, \{\Lambda_t\}_{t > 0} ) \overset{d}{=} (h^b, \{ R_b^{-1} \Lambda_{b t} \}_{t>0})$.
	Hence, by the scale-invariance of the law of $\eta$, viewed modulo time parametrization, and the independence of $\eta$ and $(h, \{\Lambda_t\}_{t>0})$,
	\begin{equation} \label{eq:scale-invariance-of-field-and-eta}
	(h,\eta,\{\Lambda_t\}_{t > 0})  \overset{d}{=} (h^b,\eta^b , \{R_b^{-1} \Lambda_{b t}\}_{t > 0} ) .
	\end{equation}
	Hence, if we let
	\[
	T^b := \sup \{ t > 0 : R_b^{-1} \Lambda_{b t} \subset B_{\rho/3} \}
	\]
	then $T^b$ has the same law as $\TT$, defined in \eqref{eq:cluster-stopping-time}.
	By Proposition \ref{prop:IDLA-constrained-convergence}, for each $\delta \in (0, \rho)$ and $t > 0$, on the event $\{t < \TT\}$, 
	it holds except on an event of probability tending to 0 as ${\eps} \to 0$  that
	\[
	\Bdeltam(\Lambda_t) \subset \overline{\mathcal{X}}_{{\eps}}( t {\eps}^{-1} ) \subset \Bdeltap(\Lambda_t),
	\]
	for $\mathcal{X} \in \{A, D\}$, \ie, both the IDLA cluster and divisible sandpile cluster converge. 
	Let $\mathcal{X} \in \{A, D\}$ be given. 
	By \eqref{eq:scale-invariance-of-field-and-eta}, 
	this implies that for each $\delta \in (0, \rho)$ and $t > 0$, on the event $\{t < T^b\}$, 
	it holds except on an event of probability tending to 0 as ${\eps} \to 0$  that
	\[
	\Bdeltam(R^{-1}_b \Lambda_{b t}) \subset  R^{-1}_b \overline{\mathcal{X}}_{{\eps}}( b t {\eps}^{-1} ) \subset \Bdeltap(R^{-1}_b \Lambda_{b t}).
	\]
	By applying the above statement with $t/b$ in place of $t$ and $\delta / R_b$ in place of $\delta$, then re-scaling space by $R_b$, we get the following.
	For each $\delta \in (0, \rho)$ and $t > 0$, on the event $\{t < b \times T^b\}$, 
	it holds except on an event of probability tending to 0 as ${\eps} \to 0$  that
	\[
	\Bdeltam(\Lambda_t) \subset \overline{\mathcal{X}}_{{\eps}}( t {\eps}^{-1} ) \subset \Bdeltap(\Lambda_t).
	\]
	As $b > 0$ was arbitrary, it remains to show that in probability, 
	\begin{equation} \label{eq:btb-converge-to-inf}
	\lim_{b \to \infty} b \times T^b = \infty.
	\end{equation}
	This, however, follows immediately from the fact that $\TT$ is strictly positive and $\TT \overset{d}{=} T^b$.
	\qed

	\appendix	
	\section{Discrete obstacle problem} \label{sec:obstacle-appendix}
We provide the proofs which were omitted in Section \ref{subsec:basic-properties} in the general setting of an undirected, locally finite, and infinite graph $G$. We consider a finite, connected set of vertices of $G$, denoted by $V$, which contain a marked vertex which we call the {\it origin}, denoted by $0$.  
The graph Laplacian, $\Delta$, is defined by,
\[
\Delta u(a) =  \frac{1}{\deg(a)} \sum_{b \sim a}(u(b) - u(a)), \quad \mbox{for all vertices $a$ of $G$}
\]
where the sum $b \sim a$ is over the vertices $b$ which share an edge with $a$ and $\deg(a)$ denotes the number of such vertices.  Observe that the analogue of the maximum principle, Lemma \ref{lemma:maximum-principle}, and the following discrete divergence theorem are satisfied on $G$.

\begin{lemma} \label{lemma:discrete-divergence-theorem}
	Let $VG$ denote all the vertices of $G$ and suppose that $f,g: VG \to \R$ are functions which are zero outside of a finite subset of $VG$. 
	Then
	\begin{equation} \label{eq:gauss-green}
	\sum_{a \in VG} f(a) \Delta g(a) \deg(a) = \sum_{a \in VG} g(a) \Delta f(a) \deg(a).
	\end{equation}
\end{lemma}
\begin{proof}
	This is well known for general, undirected graphs --- see for example~\cite[Equation (2.17)]{telcs2006} or~\cite[Lemma 3.6]{berestycki2020random}.
\end{proof}


Let $\mathfrak{r}:V \to (0, \infty)$ denote a {\it threshold function} 
and $\phi : \cl(V) \to (-\infty, 0]$ the {\it obstacle}
which satisfies the following equation
\begin{equation} \label{eq:laplacian-of-phi}
\begin{cases}
\Delta \phi  = \mathfrak t \delta_0 \quad \mbox{on $V$} \\
\phi = 0 \quad \mbox{on $\partial V$}
\end{cases}
\end{equation}
where $\mathfrak t > 0$. Consider the following {\it discrete obstacle problem}
\begin{equation} \label{eq:discrete-least-supersolution-general}
\begin{aligned}
w := &\min \{ u: \cl(V) \to \R : \Delta u \leq \mathfrak{r} \mbox{ on $V$ and $u \geq \phi$}\}.
\end{aligned}
\end{equation}
We will make use of the following function $\mathfrak{q}: \cl(V) \to \R$ which is defined by the unique solution to
\begin{equation} \label{eq:definition-of-exit-time-general}
\begin{cases}
\Delta \mathfrak q  = \mathfrak{r} \quad \mbox{on $V$} \\
\mathfrak{q} = 0 \quad \mbox{on $\partial V$}.
\end{cases}
\end{equation}
We also denote the {\it discrete cluster} by 
\begin{equation} \label{eq:discrete-cluster-general}
\Lambda = \{ a \in \cl(V) :  w(a) > \phi(a) \}.
\end{equation}

The results of Section \ref{subsec:basic-properties} may be achieved from those of this section 
via the following map, for ${\eps} > 0$, $\rho \in (0,1)$, 
\begin{align*}
\mathfrak{t} &\to {\eps}^{-1} t \deg^{{\eps}}(0)^{-1} \\
w &\to \wteps \\
G &\to {\Geps} \\
V &\to {\VGeps}(B_{\rho}) \setminus {\VGeps}(\partial B_{\rho}) \\
\mathfrak{r} &\to (\deg^{{\eps}})^{-1} \\
\phi &\to -(t {\eps}^{-1}) \times  \gr{{\eps}}{B_\rho}(0, \cdot) \\
\mathfrak q &\to  -\qeps{B_{\rho}} \\
\Lambda &\to \Lambdateps.
\end{align*}
Lemma \ref{lemma:discrete-basic-properties} is a special case of the following. 
\begin{lemma} \label{lemma:discrete-basic-properties-general}
	We have that
	\[
	\max(\mathfrak q,\phi) \leq w \leq 0
	\]
	and  $\Delta w \leq \mathfrak{r}$ on $V$.
\end{lemma}
\begin{proof}
	{\it Step 1: $w \leq 0$.} \\ 
	As $0 \geq \phi$ and $\mathfrak{r} > 0$ the function which is identically 0 is admissible in \eqref{eq:discrete-least-supersolution-general}.
	\medskip
	
	{\it Step 2: $w \geq \phi$ and $\Delta w \leq \mathfrak{r}$.} \\ 
	Observe that in general, for two functions $u_1, u_2$ defined on the vertices
	of the graph $G$, we have
	\[
	\Delta \min(u_1, u_2) \leq \min(\Delta u_1 , \Delta u_2),
	\]
	where the minimum is pointwise. 
	Indeed, write $u = \min(u_1,u_2)$ and suppose for a vertex $a$ that $u(a) = u_1(a)$. Then, by definition, 
	\[
	\Delta u(a) = \frac{1}{\deg(a)} \sum_{b \sim a}(u(b) - u_1(a))
	\leq \frac{1}{\deg(a)} \sum_{b \sim a}( u_1(b) - u_1(a)) = \Delta u_1(b). 
	\]
	Therefore, by compactness, the inequalities in the definition of \eqref{eq:discrete-least-supersolution-general}
	are also satisfied by $w$. 
	\medskip
	
	{\it Step 3: $w \geq \mathfrak q$.} \\ 
	By Steps 1 and 2, since $\phi$ is zero on $\partial V$, we have that $w$ is zero on $\partial V$. 
	Thus, by \eqref{eq:definition-of-exit-time-general} and Step 2, 	$(w - \mathfrak{q})$ is superharmonic on $V$ and zero on 
	$\partial V$, and hence the claim follows by the maximum principle.
\end{proof}

Lemma \ref{lemma:discrete-laplacian-bound} is a special case of the following. 
\begin{lemma} \label{lemma:discrete-laplacian-bound-general}
	We have that $\Delta w = \mathfrak{r}$ on $\Lambda$ 
	and $\mathfrak{r} \geq \Delta w \geq 0$ on $V$. 
	Moreover,  $\Lambda$ is connected and if $\mathfrak{t} > \mathfrak{r}(0)$, contains the origin. 
\end{lemma}
\begin{proof}
	{\it Step 1: $\Delta w = \mathfrak{r}$ on $\Lambda$.} \\
	Note that the inequality $\Delta w \leq \mathfrak{r}$ was established in Lemma \ref{lemma:discrete-basic-properties-general}. 
	Suppose for sake of contradiction $w > \phi$
	but $\Delta w(a) < \mathfrak{r}(a)$ for some $a \in V$.  Let 
	\[
	\delta = \min \left( \mathfrak{r}(a) - \Delta w(a), w(a) - \phi(a)  \right) > 0
	\]
	and consider the function $u: V \to \R$ defined by, 
	\[
	u := w - \delta \times 1\{ \cdot = a\}.
	\]
	We will see that the function $u$ is admissible in \eqref{eq:discrete-least-supersolution-general}, but strictly less than $w$,  a contradiction. 
	By definition 
	\[
	\Delta u(b) \leq \Delta w(b) \quad \forall b \neq a 
	\]
	and
	\begin{align*}
	\Delta u(a) &= \frac{1}{\deg(a)} \sum_{b \sim a} (w(b) - u(a))  
	=  \Delta w(a) + \delta 
	\leq \mathfrak{r}(a).
	\end{align*}
	Similarly, 
	\[
	u(b) = w(b) \geq \phi(b) \quad \forall b \neq a 
	\]
	and
	\begin{align*}
	u(a)  
	= w(a) - \delta  
	\geq  w(a) - w(a) + \phi(a) 
	=\phi(a), 
	\end{align*}
	completing the proof. 
	\medskip
	
	{\it Step 2: $\Delta w \geq 0$.} \\
	If $a \in V \cap \Lambda^c$, we have
	\begin{align*}
	\Delta w(a) &= \frac{1}{\deg(a)} \sum_{b \sim a} (w(b) -  w(a)) \\
	&= \frac{1}{\deg(a)} \sum_{b \sim a} (w(b) - \phi(a)) \quad \mbox{(assumption on $a$)} \\
	&\geq \Delta \phi(a)  \quad \mbox{(Lemma \ref{lemma:discrete-basic-properties-general})} \\
	&\geq 0 \quad \mbox{(since $\phi$ is subharmonic)}.
	\end{align*}
	In light of Step 1, this completes the proof. 
	\medskip
	
	{\it Step 3:  $\Lambda$ is connected.} \\
	Consider the function $v =  w - \phi$. 
	If $\Lambda$ is disconnected then there is a connected component of $\Lambda$ not containing the origin;
	call this component $A$. 
	By Step 1 and \eqref{eq:laplacian-of-phi}, we have that 
	\[
	\begin{cases}
	\Delta v = \mathfrak{r} \quad \mbox{on $A$} \\
	v > 0 \quad \mbox{on $A$} \\
	v = 0 \quad \mbox{on $\partial A$}, 
	\end{cases}
	\]
	which contradicts the maximum principle. 
	\medskip
	
	{\it Step 4: If $t > \mathfrak{r}(0)$,  $\Lambda$ contains the origin.} \\
	We demonstrate the contrapositive. Observe that the argument in Step 3 shows that $\Lambda$ is either empty or contains the origin. 
	If $\Lambda$ is empty, then
	\begin{align*}
	\mathfrak{t} &= \Delta \phi(0) \qquad \mbox{(by \eqref{eq:laplacian-of-phi})} \\
	&= 	\Delta  w(0) \qquad \mbox{(as $\Lambda$ is empty)} \\
	&\leq \mathfrak{r}(0) \qquad \mbox{(by Lemma \ref{lemma:discrete-basic-properties-general})},
	\end{align*}
	completing the proof. 
\end{proof}
We conclude with a discrete conservation of mass lemma; this is a general version
of Lemma \ref{lemma:discrete-conservation-of-mass}.
\begin{lemma} \label{lemma:discrete-conservation-of-mass-general}
	If $\mathfrak{r} = \deg^{-1}$, we have that $|\Lambda| \leq  \mathfrak{t} \times \deg(0)$. Moreover, 
	if $\Lambda \subset \inte(V)$,
	then  
	\[
	\sum_{a \in V} \Delta w(a) \deg(a) =  \mathfrak{t} \times \deg(0).
	\] 
\end{lemma}
\begin{proof}
	Let	$v = w - \phi$ and extend $v$ to a function on $VG$, the set of all vertices of $G$, by defining it to be zero on $\cl(V)^c$. 
	We compute, 
	\begin{align*}
	-\mathfrak{t} \times  \deg(0) + |\Lambda| &\leq \sum_{a \in V} \Delta v(a) \deg(a)
	\quad \mbox{(by Lemma \ref{lemma:discrete-laplacian-bound-general}, \eqref{eq:laplacian-of-phi}, and $\mathfrak{r} = \deg^{-1}$)} \\
	&= \sum_{a \in VG} 1\{ a \in V\}  \Delta v(a) \deg(a) \\
	&= \sum_{a \in VG}   v(a) \deg(a) \Delta 1\{ a \in V\} 
	\quad \mbox{(by Lemma \ref{lemma:discrete-divergence-theorem})} \\
	&= \sum_{a \in V}   v(a) \deg(a) \Delta 1\{ a \in V\} 
	\quad \mbox{($v = 0$ on $V^c$)} \\
	&= -\sum_{a \in V}  v(a)  \left( \mbox{\# neighbors of $a$ not in $V$} \right) \\
	&\leq 0 \quad \mbox{(since $v \geq 0$ by Lemma \ref{lemma:discrete-basic-properties-general})}.
	\end{align*}
	If $\Lambda \subset \inte(V)$, then $v = 0$ on all vertices which share an edge with a vertex in $\partial V$. 
	Thus, an identical computation, where the first line is replaced by 
	\[
	-\mathfrak{t} \times  \deg(0) + \sum_{a \in V} \Delta w(a) \deg(a) = \sum_{a \in V} \Delta v(a) \deg(a)
	\]
	and the last inequality by an equality, shows the moreover clause. 
\end{proof}

\bibliographystyle{alpha}
\bibliography{refs.bib}

\end{document}

%% file: defs.tex
\newcommand{\R}{\mathbb{R}}
\newcommand{\C}{\mathbb{C}}
\newcommand{\Z}{\mathbb{Z}}
\newcommand{\N}{\mathbb{N}}
\newcommand{\Q}{\mathbb{Q}}
\newcommand{\A}{\mathbb{A}}
  
\newcommand{\diam}{{\mbox{diam}}}

\newcommand{\E}{\mathop{\bf E{}}}

\newcommand{\Cov}{\mathop{\bf Cov{}}}

\newcommand{\dist}{\mathop{\bf dist{}}}

\newcommand{\eg}{{\it e.g.}}
\newcommand{\ie}{{\it i.e.}}

\newtheorem{theorem}{Theorem}[section]
\newtheorem{prop}[theorem]{Proposition}
\newtheorem{fact}[theorem]{Fact}

\newtheorem{lemma}[theorem]{Lemma}
\newtheorem{remark}[theorem]{Remark}
\newtheorem{definition}[theorem]{Definition}
\newtheorem{problem}{Problem}